\documentclass[a4paper,reqno]{amsart}

\usepackage[english]{babel}
\usepackage[utf8]{inputenc}
\usepackage{amsfonts,amssymb,amsmath,amsthm}
\usepackage{booktabs}
\usepackage[all,rotate]{xy}
\usepackage{fixltx2e}
\usepackage{xspace}
\usepackage{aliascnt}

\usepackage{xcolor}
\definecolor{linkblue}{RGB}{1,1,255}
\usepackage[pdfborder={0 0 0.8 [2 3]}]{hyperref}
\hypersetup{pdftitle={Sets of lengths in maximal orders in central simple algebras},
            pdfauthor={Daniel Smertnig}}

\theoremstyle{definition}
\newtheorem{defi}{Definition}[section]

\newaliascnt{defiprop}{defi}

\aliascntresetthe{defiprop}

\newaliascnt{defilemma}{defi}
\newtheorem{defilemma}[defilemma]{Definition \& Lemma}
\aliascntresetthe{defilemma}

\theoremstyle{plain}

\newaliascnt{thm}{defi}
\newtheorem{thm}[thm]{Theorem}
\newtheorem*{thm*}{Theorem}
\aliascntresetthe{thm}

\newaliascnt{lemma}{defi}
\newtheorem{lemma}[lemma]{Lemma}
\aliascntresetthe{lemma}

\newaliascnt{prop}{defi}
\newtheorem{prop}[prop]{Proposition}
\aliascntresetthe{prop}

\newaliascnt{cor}{defi}
\newtheorem{cor}[cor]{Corollary}
\aliascntresetthe{cor}

\theoremstyle{remark}
\newaliascnt{remark}{defi}
\newtheorem{remark}[remark]{Remark}
\aliascntresetthe{remark}

\newaliascnt{exm}{defi}
\newtheorem{exm}[exm]{Example}
\aliascntresetthe{exm}

\addto\extrasenglish{

}

\newcommand{\bF}{ \mathbb{F} }

\newcommand{\bN}{ \mathbb{N} }

\newcommand{\bQ}{ \mathbb{Q} }

\newcommand{\bZ}{ \mathbb{Z} }
\newcommand{\bG}{ \mathbb{G} }

\newcommand{\cA}{ \mathcal{A} }
\newcommand{\cB}{ \mathcal{B} }
\newcommand{\cC}{ \mathcal{C} }

\newcommand{\cF}{ \mathcal{F} }
\newcommand{\cG}{ \mathcal{G} }
\newcommand{\cH}{ \mathcal{H} }
\newcommand{\cI}{ \mathcal{I} }

\newcommand{\cM}{ \mathcal{M} }

\newcommand{\cL}{ \mathcal{L} }
\newcommand{\cP}{ \mathcal{P} }
\newcommand{\cQ}{ \mathcal{Q} }
\newcommand{\cO}{ \mathcal{O} }

\newcommand{\cU}{ \mathcal{U} }

\newcommand{\cX}{ \mathcal{X} }

\newcommand{\sL}{ \mathsf{L} }
\newcommand{\sZ}{ \mathsf{Z} }

\newcommand{\fD}{ \mathfrak{D} }
\newcommand{\fN}{ \mathfrak{N} }
\newcommand{\fP}{ \mathfrak{P} }

\newcommand{\fS}{ \mathfrak{S} }

\newcommand{\fa}{ \mathfrak{a} }
\newcommand{\ff}{ \mathfrak{f} }
\newcommand{\fm}{ \mathfrak{m} }

\newcommand{\fp}{ \mathfrak{p} }

\newcommand{\fr}{ \mathfrak{r} }

\newcommand{\ideal}{ \,\triangleleft\, }

\newcommand{\quo}{ \mathbf{q} }

\newcommand{\lbar}{ \overline }
\renewcommand{\vec}[1]{\mathbf{#1}}
\providecommand{\length}[1]{\lvert#1\rvert}
\providecommand{\abs}[1]{\lvert#1\rvert}

\providecommand{\card}[1]{\lvert#1\rvert}

\providecommand{\val}{ \mathsf{v} }

\newcommand{\isomto}{\overset{\sim}{\rightarrow}}
\newcommand{\f}{\frac}

\providecommand{\meet}{\wedge}
\providecommand{\join}{\vee}

\newcommand{\lc}[2]{{({#1}\!\!:_l\!{#2})}}
\newcommand{\rc}[2]{{({#1}\!\!:_r\!{#2})}}

\DeclareMathOperator{\id}{id}

\DeclareMathOperator{\ord}{ord}
\DeclareMathOperator{\chr}{char}

\DeclareMathOperator{\nr}{nr}
\DeclareMathOperator{\tr}{tr}
\DeclareMathOperator{\Trace}{Tr}
\DeclareMathOperator{\Norm}{N}
\DeclareMathOperator{\Cl}{\mathcal C}
\DeclareMathOperator{\LCl}{\mathcal{LC}}

\DeclareMathOperator{\lcm}{lcm}
\DeclareMathOperator{\cntr}{center}
\DeclareMathOperator{\gen}{gen}
\DeclareMathOperator{\Pic}{Pic}

\renewcommand{\theenumi}{\arabic{enumi}}
\renewcommand{\labelenumi}{\theenumi.}

\newcommand{\enumequiv}{%
  \renewcommand{\theenumi}{(\alph{enumi})}%
  \renewcommand{\labelenumi}{\theenumi}%
}

\renewcommand{\theenumii}{(\roman{enumii})}
\renewcommand{\labelenumii}{\theenumii}

\def\rfop{*}
\makeatletter
\newcommand\rigidfactorization[2][]{%
  \def\rf@delim{\rfop}
  \newif\ifrf@notfirst
  #1
  \@for\next:=#2\do{%
    \ifrf@notfirst
      \rf@delim
    \fi
    \rf@notfirsttrue
    \next
  }%
}
\makeatother
\newcommand\rf\rigidfactorization

\newcommand{\cgrp}{C}

\begin{document}

\title{Sets of lengths in maximal orders in central simple algebras} 

\author{Daniel Smertnig}
\address{Institut f\"ur Mathematik und Wissenschaftliches Rechnen \\
Karl-Franzens-Uni\-ver\-si\-t\"at Graz \\
Heinrichstra\ss e 36\\
8010 Graz, Austria} \email{daniel.smertnig@uni-graz.at}
\thanks{The author is supported by the Austrian Science Fund (FWF): W1230, Doctoral Program ``Discrete Mathematics''}


\begin{abstract}
  Let $\cO$ be a holomorphy ring in a global field $K$, and $R$ a classical maximal $\cO$-order in a central simple algebra over $K$.
  We study sets of lengths of factorizations of cancellative elements of $R$ into atoms (irreducibles).
  In a large majority of cases there exists a transfer homomorphism to a monoid of zero-sum sequences over a ray class group of $\cO$,
  which implies that all the structural finiteness results for sets of lengths---valid for commutative Krull monoids with finite class group---hold also true for $R$.
  If $\cO$ is the ring of algebraic integers of a number field $K$, we prove that in the remaining cases no such transfer homomorphism can exist and that several invariants dealing with sets of lengths are infinite.
\end{abstract}

\maketitle

\section{Introduction}
Let $H$ be a (left- and right-) cancellative semigroup and $H^\times$ its group of units. An element $u \in H \setminus H^\times$ is called \emph{irreducible} (or an \emph{atom}) if $u=ab$ with $a,b \in H$ implies that $a \in H^\times$ or $b \in H^\times$. If $a \in H \setminus H^\times$, then $l \in \bN$ is a \emph{length of $a$} if there exist atoms $u_1, \ldots, u_l \in H$ with $a = u_1 \cdot\ldots\cdot u_l$, and the \emph{set of lengths} of $a$, written as $\sL (a)$, consists of all such lengths. If there is a non-unit $a \in H$ with $\card{\sL(a)} > 1$, say $1 < k < l \in \sL(a)$, then for every $n \in \bN$, we have $\sL(a^n) \supset \{\, kn + \nu (l-k) \mid \nu \in [0, n] \,\}$, which shows that sets of lengths become arbitrarily large.
If $H$ is commutative and satisfies the ACC on divisorial ideals, then all sets of lengths are finite and non-empty.

Sets of lengths (and all invariants derived from them, such as the set of distances) are among the most investigated invariants in factorization theory. So far research has almost been entirely devoted to the commutative setting, and it has focused on commutative noetherian domains, commutative Krull monoids, numerical monoids, and others (cf. \cite{andersondd97,chapman05,ghk06,geroldinger-hassler08,geroldinger-grynkiewicz09,fontana-houston-lucas12,baeth-wiegand13}). Recall that a commutative noetherian domain is a Krull domain if and only if the monoid of non-zero elements is a Krull monoid and this is the case if and only if the domain is integrally closed. Suppose that $H$ is a Krull monoid (so completely integrally closed and the ACC on divisorial two-sided ideals holds true). Then the monoid of divisorial two-sided ideals is a free abelian monoid. If $H$ is commutative (or at least normalizing), this gives rise to the construction of a transfer homomorphism $\theta \colon H \to \cB(G_P)$, where $\cB(G_P)$ is the monoid of zero-sum sequences over a subset $G_P$ of the class group $G$ of $H$. Transfer homomorphisms preserve sets of lengths, and if $G_P$ is finite, then $\cB(G_P)$ is a finitely generated commutative Krull monoid, whose sets of lengths can be studied with
methods from combinatorial number theory. This approach has lead to a large variety of structural results for sets of lengths in commutative Krull monoids (see \cite{ghk06, geroldinger09} for an overview).

Only first hesitant steps were taken so far to study factorization properties in a non-commutative setting (for example, quaternion orders are investigated in \cite{estes-nipp89, estes91, conway-smith03}), semifirs in (\cite{cohn85,cohn06}, semigroup algebras in \cite{jespers-wang01}). The present paper provides an in-depth study of sets of lengths in classical maximal orders over holomorphy rings in global fields.

Let $\cO$ be a commutative Krull domain with quotient field $K$, $A$ a central simple algebra over $K$, $R$ a maximal order in $A$, and $R^{\bullet}$ the semigroup of cancellative elements (equivalently, $R$ is a PI Krull ring). Any approach to study sets of lengths, which runs as described above and involves divisorial two-sided ideals, is restricted to normalizing Krull monoids (\cite[Theorem 4.13]{geroldinger13}). For this reason we develop the theory of divisorial one-sided ideals. In \autoref{sec:arithinvariants} we fix our terminology in the setting of cancellative small categories. Following ideas of Asano and Murata \cite{asano-murata53} and partly of Rehm \cite{rehm77-1, rehm77-2}, we provide in \autoref{sec:fact} a factorization theory of integral elements in arithmetical groupoids, and introduce an abstract transfer homomorphism for a subcategory of such a groupoid (\autoref{thm:abstract-transfer}). In \autoref{sec:ideals} the divisorial one-sided ideal theory of maximal orders in quotient semigroups is given, and \autoref{prop:ideals-are-lgrpd} establishes the relationship with arithmetical groupoids.
\autoref{thm:transfer-semigroup} is a main result in the abstract setting of arithmetical maximal orders (Remarks \hyperref[rem:ideal-struct:normalizing]{\ref*{rem:ideal-struct}.\ref*{rem:ideal-struct:normalizing}} and \hyperref[rtr:normalizing]{\ref*{rem:rtr}.\ref*{rtr:normalizing}} reveal how the well-known transfer homomorphisms for normalizing Krull monoids fit into our abstract theory). For maximal orders over commutative Krull domains, we see that all sets of lengths are finite and non-empty (\autoref{cor:pi-krull-ring}).
In \autoref{sec:maxord-transfer} we demonstrate that classical maximal orders over holomorphy rings in global fields fulfill the abstract assumptions of \autoref{thm:transfer-semigroup}, which implies the following structural finiteness results on sets of lengths.

\begin{thm} \label{thm:main-transfer}
  Let $\cO$ be a holomorphy ring in a global field $K$, $A$ a central simple algebra over $K$, and $R$ a classical maximal $\cO$-order of $A$.
  Suppose that every stably free left $R$-ideal is free. Then there exists a transfer homomorphism $\theta \colon R^\bullet \to \cB(\Cl_A(\cO))$, where
  \[
    \Cl_A(\cO) = \cF^\times(\cO) \,/\, \{\, a\cO \mid a \in K^\times,\, a_v > 0 \text{ for all archimedean places $v$ of $K$ where $A$ is ramified.} \,\}
  \]
  is a ray class group of $\cO$, and $\cB(\Cl_A(\cO))$ is the monoid of zero-sum sequences over $\Cl_A(\cO)$. In particular,
  \begin{enumerate}
    \item The set of distances $\Delta(R^\bullet)$ is a finite interval, and if it is non-empty, then $\min \Delta(R^\bullet) = 1$.
    \item For every $k \in \bN$, the union of sets of lengths containing $k$, denoted by $\cU_k(R^\bullet)$, is a finite interval.
    \item There is an $M \in \bN_0$ such that for every $a \in R^\bullet$ the set of lengths $\sL(a)$ is an AAMP with difference $d \in \Delta (R^{\bullet})$ and bound $M$.
  \end{enumerate}
\end{thm}

Thus, under the additional hypothesis that every stably free left $R$-ideal is free, we obtain a transfer homomorphism to a monoid of zero-sum sequences over a finite abelian group. Therefore, sets of lengths in $R$ are the same as sets of lengths in a commutative Krull monoid with finite class group.

If $A$ satisfies the Eichler condition relative to $\cO$, then every stably free left $R$-ideal is free by Eichler's Theorem. In particular, if $K$ is a number field and $\cO$ is its ring of algebraic integers, then $A$ satisfies the Eichler condition relative to $\cO$ unless $A$ is a totally definite quaternion algebra. Thus in this setting \autoref{thm:main-transfer} covers the large majority of cases, and the following complementary theorem shows that the condition that every stably free left $R$-ideal is free is indeed necessary.

\begin{thm} \label{thm:main-distances}
  Let $\cO$ be the ring of algebraic integers in a number field $K$, $A$ a central simple algebra over $K$, and $R$ a classical maximal $\cO$-order of $A$. If there exists a stably free left $R$-ideal that is not free, then there exists no transfer homomorphism $\theta \colon R^\bullet \to \cB(G_P)$, where $G_P$ is any subset of an abelian group. Moreover,
  \begin{enumerate}
    \item $\Delta(R^\bullet) = \bN$.
    \item For every $k \ge 3$, we have $\bN_{\ge 3} \subset \cU_k(R^\bullet) \subset \bN_{\ge_2}$.
  \end{enumerate}
\end{thm}

The proof of \autoref{thm:main-distances} is based on recent work of Kirschmer and Voight (\cite{kirschmer-voight10, kirschmer-voight10cor}),
and will be given in \autoref{sec:maxord-distances}.
If $H$ is a commutative Krull monoid with an infinite class group such that every class contains a prime divisor, then Kainrath showed that every finite subset of $\bN_{\ge2 }$ can be realized as a set of lengths (\cite{kainrath99}, or \cite[Section 7.4]{ghk06}), whence $\Delta(H) = \bN$ and $\cU_k(H) = \bN_{\ge 2}$ for all $k \ge 2$. However, we explicitly show that in the above situation no transfer homomorphism is possible, implying that the factorization of $R^\bullet$ cannot be modeled by a monoid of zero-sum sequences.
A similar statement about sets of lengths in the integer-valued polynomials, as well as the impossibility of a transfer homomorphism to a monoid of zero-sum sequences, was recently shown by Frisch \cite{frisch12}.

\section{Preliminaries} \label{sec:prelim}

Let $\bN$ denote the set of positive integers and put $\bN_0 = \{\, 0 \,\} \cup \bN$. For integers $a,b \in \bZ$, let $[a,b] = \{\, x \in \bZ \mid a \le x \le b \,\}$ denote the discrete interval.
All semigroups and rings are assumed to have an identity element, and all homomorphisms respect the identity.
By a factorization we always mean a factorization of a cancellative element into irreducible elements (a formal definition follows in \autoref{sec:arithinvariants}). 
In order to study factorizations in semigroups we will have to investigate their divisorial one-sided ideal theory, in which the multiplication of ideals only gains sufficiently nice properties if one considers it as a partial operation that is only defined for certain pairs of ideals. This is the reason why we introduce our concepts in the setting of groupoids and consider subcategories of these groupoids.

Throughout the paper there will be many statements that can be either formulated ``from the left'' or ``from the right'', and most of the time it is obvious how the symmetric statement should look like. Therefore often just one variant is formulated and it is left to the reader to fill in the symmetric definition or statement if required.

\subsection{Small categories as generalizations of semigroups} \label{sec:small-cat}
Let $H$ be a small category. In the sequel the objects of $H$ play no role, and therefore we shall identify $H$ with the set of morphisms of $H$. We denote by $H_0$ the set of identity morphisms (representing the objects of the category). There are two maps $s,t\colon H \to H_0$ such that two elements $a,b \in H$ are composable to a (uniquely determined) element $ab \in H$ if and only if $t(a)=s(b)$.
\footnote{This choice of $t$ and $s$ is compatible with the usual convention for groupoids, but unfortunately opposite to the usual convention for categories.}
For $e,f \in H_0$ we set $H(e,f) = \{\, a \in H \mid s(a) = e,\, t(a) = f \,\}$, $H(e) = H(e,e)$, $H(e,\cdot) = \bigcup_{f' \in H_0} H(e,f')$ and $H(\cdot,f) = \bigcup_{e' \in H_0} H(e',f)$. Note that an element $e \in H$ lies in $H_0$ if and only if $s(e) = t(e) = e$, $ea = a$ for all $a \in H(e,\cdot)$ and $ae = a$ for all $a \in H(\cdot,e)$. 

A semigroup may be viewed as a category with a single object (corresponding to its identity element), and elements of the semigroup as morphisms with source and target this unique object. In this way the notion of a small category generalizes the usual notion of a semigroup ($H$ is a semigroup if and only if $\card{H_0}=1$). We will consider a semigroup to be a small category in this sense whenever this is convenient, without explicitly stating this anymore.
For $A, B \subset H$ we write $AB = \{\, ab \in H \mid a \in A, b \in B \text{ and } t(a) = s(b) \,\}$ for the set of all possible products, and if $b \in H$, then $Ab = A \{b\}$ and $bA = \{b\}A$.

An element $a \in H$ is called \emph{left-cancellative} if it is an epimorphism ($ab=ac$ implies $b=c$ for all $b,c \in H(t(a),\cdot)$), and it is called \emph{right-cancellative} if it is a monomorphism ($ba=ca$ implies $b=c$ for all $b,c \in H(\cdot,s(a))$), and \emph{cancellative} if it is both.
The set of all cancellative elements is denoted by $H^\bullet$, and $H$ is called \emph{cancellative} if $H = H^\bullet$.
The set of isomorphisms of $H$ will also be called the \emph{set of units}, and we denote it by $H^\times$.
A subcategory $D \subset H$ is \emph{wide} if $D_0 = H_0$.

In line with the multiplicative notation, if $H$ and $D$ are two small categories, we call a functor $f\colon H \to D$ a homomorphism (of small categories). Explicitly, a map $f\colon H \to D$ is a homomorphism if $f(H_0) \subset D_0$ and whenever $a,b \in H$ with $t(a)=s(b)$ then also $f(a)\cdot f(b)$ is defined (i.e., $t(f(a)) = s(f(b))$) and $f(ab) = f(a)f(b)$.

\smallskip
If $H$ is a commutative semigroup, and $D \subset H$ is a subsemigroup, then a localization $D^{-1}H$ with an embedding $H \hookrightarrow D^{-1}H$ exists whenever all elements of $D$ are cancellative, and in particular $H$ has a group of fractions if and only if $H$ is cancellative. 
If $H$ is a non-commutative semigroup and $D \subset H$, then a semigroup of right fractions with respect to $D$, $HD^{-1}$, in which every element can be represented as a fraction $ad^{-1}$ with $a \in H$, $d \in D$, together with an embedding $H \hookrightarrow HD^{-1}$, exists if and only if $D$ is cancellative and $D$ satisfies the \emph{right Ore condition}, meaning $aD \cap dH \ne \emptyset$ for all $a \in H$ and $d \in D$. For a semigroup of left fractions, $D^{-1}H$, one gets the analogous \emph{left Ore condition}, and if $D$ satisfies both, the left and the right Ore condition, then every semigroup of right fractions is a semigroup of left fractions and conversely. In this case we write $D^{-1}H=HD^{-1}$. If $H^\bullet$ satisfies the left and right Ore condition, we also write $\quo(H) = H(H^\bullet)^{-1} = ({H^\bullet}^{-1})H$ for the corresponding semigroup of fractions. 

The notion of semigroups of fractions generalizes to categories of fractions with analogous conditions (\cite{gabriel-zisman67}).
Let $H$ be a small category, and $D \subset H^\bullet$ a subset of the cancellative elements.
Then $D$ \emph{admits a calculus of right fractions} if $D$ is a wide subcategory of $H$ and it satisfies the right Ore condition, i.e., $aD \cap dH \ne \emptyset$ for all $a \in H$ and $d \in D$ with $s(a)=s(d)$.
In that case there exists a small category $HD^{-1}$ with $(HD^{-1})_0 = H_0$ and an embedding $j\colon H \to HD^{-1}$ (i.e., $j$ is a faithful functor) with $j\mid H_0 = \id$ and such that every element of $HD^{-1}$ can be represented in the form $j(a) j(d)^{-1}$ with $a \in H$, $d \in D$ and $t(a)=t(d)$, $j(D) \subset H^\times$ and it is universal with respect to that property, i.e., if $f\colon H \to S$ is any homomorphism with $f(D) \subset S^\times$, then there exists a unique $D^{-1}f\colon HD^{-1} \to S$ such that $D^{-1}f \circ j = f$. We can assume $H \subset HD^{-1}$ and take $j$ to be the inclusion map, and we call $HD^{-1}$ the \emph{category of right fractions} of $H$ with respect to $D$. If $D$ also admits a \emph{left calculus of fractions}, then $HD^{-1}$ is also a category of left fractions, and we write $HD^{-1}=D^{-1}H$.

\smallskip
A \emph{monoid} is a cancellative semigroup satisfying the left and right Ore condition (following the convention of \cite{geroldinger13}). Every monoid has a (left and right) group of fractions which is unique up to unique isomorphism. A semigroup $H$ is called \emph{normalizing} if $aH = Ha$ for all $a \in H$. It is easily checked that a normalizing cancellative semigroup is already a normalizing monoid.

\smallskip
Let $\cM$ be a directed multigraph (i.e., a quiver).
For every edge $a$ of $\cM$ we write $s(a)$ for the vertex that is its source and $t(a)$ for the vertex that is its target.
The \emph{path category} on $\cM$, denoted by $\cF(\cM)$, is defined as follows:
It consists of all tuples $y=(e, a_1, \ldots, a_k, f)$ with $k \in \bN_0$, $e,f$ vertices of $\cM$ and $a_1,\ldots,a_k$ edges of $\cM$ with either $k=0$ and $e=f$ or $k > 0$, $s(a_1)=e$, $t(a_i) = s(a_{i+1})$ for all $i \in [1,k-1]$ and $t(a_k) = f$. The set of identities $\cF(\cM)_0$ is the set of all tuples with $k=0$, and given any tuple $y$ as above, $s(y)=(e,e)$ and $t(y)=(f,f)$. Composition is defined in the obvious manner by concatenating tuples and removing the two vertices in the middle. We identify the set of vertices of $\cM$ with $\cF(M)_0$ so that $(e,e) = e$. Every subset $M$ of a small category $H$ will be viewed as a quiver, with vertices $\{\, s(a) \mid a \in M \,\} \cup \{\, t(a) \mid a \in M \,\}$ and for each $a \in M$ a directed edge (again called $a$) from $s(a)$ to $t(a)$.

\subsection{Groupoids} \label{sec:groupoids} A groupoid $G$ is a small category in which every element is a unit (i.e., every morphism is an isomorphism). If $e,f,e',f' \in G_0$ and there exist $a \in G(e,f)$ and $b \in G(e',f')$, then
\begin{equation} \label{eq:groupoid-bij}
  \begin{cases}
    G(e,e')  &\to G(f,f') \\
          x &\mapsto a^{-1} x b
  \end{cases}
\end{equation}
is a bijection. 

For all $e \in G_0$ the set $G(e)$ is a group, called the \emph{vertex group} or \emph{isotropy group} of $G$ at $e$.
If $f \in G_0$ and $a \in G(e,f)$, then, taking $b=a$, the map in \eqref{eq:groupoid-bij} is a group isomorphism from $G(e)$ to $G(f)$.
If $G(e)$ is abelian, it can be easily checked that this isomorphism does not depend on the choice of $a$: If $a,a' \in G(e,f)$, then 
\[
  a'(a^{-1}xa)a'^{-1} = (a'a^{-1})x(aa'^{-1}) = (a'a^{-1})(aa'^{-1})x = x.
\]
In particular, if $G$ is connected (meaning $G(e,e') \ne \emptyset$ for all $e,e' \in G_0$) and one vertex group is abelian, then all vertex groups are abelian, and they are canonically isomorphic.

In this case we define for $e \in G_0$ and $x \in G(e)$ the set $(x) = \{\, a^{-1}xa \mid a \in G(e,\cdot) \,\}$,
and the \emph{universal vertex group} as 
\[
  \bG = \{\, (x) \mid x \in G(e),\, e \in G_0 \,\}.
\]
$\bG$ indeed has a natural abelian group structure: For every $e \in G_0$ there is a bijection $j_e\colon G(e) \to \bG, x \mapsto (x)$ inducing the structure of an abelian group on $\bG$, and because the diagrams
\[
\xy*\xybox{
  \xymatrix@R=3em@C=2em{
    G(e) \ar@{->}[rr]^{x \mapsto a^{-1}xa} \ar@{->}[rd]_{j_e}&             & G(f) \ar@{->}[ld]^{j_f} \\
                      & \bG         &
  }
}\endxy.
\]
commute for every choice of $e,f \in G_0$ and $a \in G(e,f)$, this group structure is independent of the choice of $e$, yielding a canonical group isomorphism $j_e\colon G(e) \to \bG$ for every $e \in G_0$. We will use calligraphic letters to denote elements of $\bG$. If $\cX \in \bG$, then the unique representative of $\cX$ in $G(e)$, $j_e^{-1}(\cX)$, will be denoted by $\cX_e$.

If $G$ is a groupoid, and $H \subset G$ is a subcategory, then $HH^{-1}$ denotes the set of all right fractions of elements of $H$.
Furthermore, $H H^{-1} \subset G$ is a subgroupoid if and only if $H$ satisfies the right Ore condition.

\subsection{Krull monoids and Krull rings}

A monoid $H$ is called a \emph{Krull monoid} if it is completely integrally closed (in other words, a maximal order) and satisfies the ACC on divisorial two-sided ideals. A prime Goldie ring $R$ is a \emph{Krull ring} if it is completely integrally closed and satisfies the ACC on divisorial two-sided ideals (equivalently, its monoid $R^{\bullet}$ of cancellative elements is a Krull monoid; see \cite{geroldinger13}). 
The theory of commutative Krull monoids is presented in \cite{halterkoch98, ghk06}. The simplest examples of non-commutative Krull rings are classical maximal orders in central simple algebras over Dedekind domains (see \autoref{sec:ring-csa}). We discuss monoids of zero-sum sequences.

Let $G = (G,0_G,+)$ be an additively written abelian group, $G_P \subset G$ a subset and let $\cF_{ab}(G_P)$ be the (multiplicatively written) free abelian monoid with basis $G_P$.
Elements $S \in \cF_{ab}(G_P)$ are called \emph{sequences over $G_P$}, and are written in the form $S = g_1\cdot\ldots\cdot g_l$ where $l \in \bN_0$ and $g_1, \ldots, g_l \in G_P$.
We denote by $\length{S}=l$ the \emph{length} of $S$.
Such a sequence $S$ is said to be a \emph{zero-sum sequence} if $\sigma(S) = g_1 + \ldots + g_l = 0_G$.
The submonoid
\[
  \cB(G_P) = \{\, S \in \cF_{ab}(G_P) \mid \sigma(S) = 0 \,\} \;\subset\; \cF_{ab}(G_P)
\]
is called the \emph{monoid of zero-sum sequences} over $G_P$. It is a reduced commutative Krull monoid, which is finitely generated whenever $G_P$ is finite (\cite[Theorem 3.4.2]{ghk06}). 
Moreover, every commutative Krull monoid possesses a transfer homomorphism onto a monoid of zero-sum sequences, and thus $\cB(G_P)$ provides a model for the factorization behavior of commutative Krull monoids (\cite[Section 3.4]{ghk06}).

\section{Arithmetical Invariants} \label{sec:arithinvariants}

In this section we introduce our main arithmetical invariants (rigid factorizations, sets of lengths, sets of distances) and transfer homomorphisms in the setting of cancellative small categories.

\begin{center}
  \emph{Throughout this section, let $H$ be a cancellative small category.}
\end{center}

$H$ is \emph{reduced} if $H^\times = H_0$. An element $u \in H \setminus H^\times$ is an \emph{atom} (or \emph{irreducible}) if $u = bc$ with $b,c \in H$ implies $b \in H^\times$ or $c \in H^\times$. By $\cA(H)$ we denote the set of all atoms of $H$, and call $H$ \emph{atomic} if every $a \in H \setminus H^\times$ can be written as a (finite) product of atoms.
A left ideal of $H$ is a subset $I \subset H$ with $H I \subset I$, and a right ideal of $H$ is defined similarly. A \emph{principal left (right) ideal of $H$} is a set of the form $Ha$ ($aH$) for some $a \in H$.
If $H$ is a commutative monoid, then $p \in H \setminus H^\times$ is a \emph{prime element} if $p \mid ab$ implies $p \mid a$ or $p \mid b$ for all $a, b \in H$. 

\begin{prop} \label{prop:atomic}
  If $H$ satisfies the ACC on principal left and right ideals, then $H$ is atomic.
\end{prop}

\begin{proof}
  We first note that if $a,b \in H$ then $aH = bH$ if and only if $a=b \varepsilon$ with $\varepsilon \in H^\times$, and similarly $Ha = Hb$ if and only if $a = \varepsilon b$ with $\varepsilon \in H^\times$. [We only show the statement for the right ideals. The non-trivial direction is showing that $aH = bH$ implies $a = b \varepsilon$. Since $aH=bH$ implies $a=bx$ and $b=ay$ with $x,y \in H$, we get $a=a(yx)$ and $b=b(xy)$. Since $H$ is cancellative, this implies $xy=t(b)=s(x)$ and $yx=t(a)=s(y)$, hence $y = x^{-1}$ and therefore $x,y \in H^\times$.]

  \begin{enumerate}
    \setlength\itemindent{1.55em}
    \item[Claim \textbf{A}.] If $a \in H \setminus H^\times$, then there exist $u \in \cA(H)$ and $a_0 \in H$ such that $a=u a_0$.
  \end{enumerate}

  \begin{proof}[Proof of Claim \textbf{A}]
    \renewcommand{\qedsymbol}{$\Box$(\textbf{A})}
    Assume the contrary. Then the set
    \[
      \Omega = \{\, a' H \mid \text{$a' \in H \setminus H^\times$ such that there are no $u \in \cA(H)$, $a_0 \in H$ with $a'=ua_0$} \,\} 
    \]
    is non-empty, and hence, using the ascending chain condition on the principal right ideals, possesses a maximal element $aH$ with $a \in H \setminus H^\times$. Then $a \not \in \cA(H)$, and therefore $a = b c$ with $b,c \in H \setminus H^\times$.
    But $a H \subsetneq b H$ since $c \not \in H^\times$, and thus maximality of $aH$ in $\Omega$ implies $b=u b_0$ with $u \in \cA(H)$ and $b_0 \in H$. But then $a = u(b_0 c)$, a contradiction.
  \end{proof}

  We proceed to show that every $a \in H \setminus H^\times$ is a product of atoms. Again, assume that this is not the case.
  Then
  \[
    \Omega' = \{\, H a' \mid \text{$a' \in H\setminus H^\times$ such that $a'$ is not a product of atoms} \,\} 
  \]
  is non-empty, and hence possesses a maximal element $Ha$ with $a \in H \setminus H^\times$ (this time using the ascending chain condition on principal left ideals). Again $a \not \in \cA(H)$ as otherwise it would be a product of atoms.
  By Claim \textbf{A}, $a = u a_0$ with $u \in \cA(H)$ and $a_0 \in H$. Since $a \not \in \cA(H)$, $a_0 \not \in H^\times$. Moreover, $Ha \subsetneq Ha_0$ since $u \not \in H^\times$ and therefore $a_0 = u_1\cdot\ldots\cdot u_l$ with $l \in \bN$ and $u_1,\ldots,u_l \in \cA(H)$. Thus $a = u u_1 \cdot\ldots\cdot u_l$ is a product of atoms, a contradiction.
\end{proof}

The following definition provides a natural notion of an ordered factorization (called a \emph{rigid factorization}) in a cancellative small category. It is modeled after a terminology by Cohn \cite{cohn85,cohn06}. 

Let $\cF(\cA(H))$ denote the path category on atoms of $H$. We define
\[
  H^\times \times_r \cF(\cA(H)) = \{\, (\varepsilon, y) \in H^\times \times \cF(\cA(H)) \mid t(\varepsilon) = s(y) \,\},
\]
and define an associative partial operation on $H^\times \times_r \cF(\cA(H))$ as follows:
If $(\varepsilon,y), (\varepsilon', y') \in H^\times \times_r \cF(\cA(H))$ with $\varepsilon, \varepsilon' \in H^\times$,
\[
  y = (e, u_1, u_2, \ldots, u_k, f) \in \cF(\cA(H)) \;\text{ and }\; y' = (e', v_1, v_2, \ldots, v_l, f') \in \cF(\cA(H)),
\]
then the operation is defined if $t(y) = s(\varepsilon')$, and
\[
  (\varepsilon,y) \cdot (\varepsilon',y') = (\varepsilon, (e, u_1,\ldots, u_k\varepsilon', v_1, v_2, \ldots, v_l, f')) \quad\text{if $k > 0$, }
\]
while $(\varepsilon,y)\cdot(\varepsilon',y') = (\varepsilon\varepsilon', y')$ if $k = 0$.
In this way $H^\times \times_r \cF(\cA(H))$ is again a cancellative small category (with identities $\{\, (e, (e,e)) \mid e \in H_0 \}$ that we identify with $H_0$ again, $s(\varepsilon,y) = s(\varepsilon)$ and $t(\varepsilon,y) = t(y)$).
We define a congruence relation $\sim$ on it as follows: If $(\varepsilon,y), (\varepsilon',y') \in H^\times \times_r \cF(\cA(H))$ with $y,y'$ as before, then
$(\varepsilon,y) \sim (\varepsilon',y')$ if $k = l$, $\varepsilon u_1\cdot\ldots\cdot u_k = \varepsilon' v_1 \cdot\ldots\cdot v_l \in H$ and either $k=0$ or there exist $\delta_2,\ldots,\delta_k \in H^\times$ and $\delta_{k+1} = t(u_k)$ such that
\[
  \varepsilon' v_1 = \varepsilon u_1 \delta_2^{-1} \quad\text{and}\quad v_i = \delta_i u_i \delta_{i+1}^{-1} \text{ for all $i \in [2,k]$}.
\]

\begin{defi}
  The \emph{category of rigid factorizations of $H$} is defined as
  \[
    \sZ^*(H) = (H^\times \times_r \cF(\cA(H))) / \sim,
  \]

  For $z \in \sZ^*(H)$ with $z = [(\varepsilon, (e, u_1, u_2, \ldots, u_k, f))]_\sim$ we write $z=\rf[\varepsilon]{u_1,\ldots,u_k}$ and the operation on $\sZ^*(H)$ is also denoted by $\rfop$. The \emph{length} of $z$ is $\length{z} = k$.
  There is a surjective homomorphism $\pi\colon \sZ^*(H) \to H$, induced by multiplying out the elements of the factorization in $H$, explicitly $\pi(z) = \varepsilon u_1 u_2 \cdot\ldots\cdot u_k \in H$. For $a \in H$, we define $\sZ^*(a) = \sZ_H^*(a) = \pi^{-1}(\{a\})$ to be the \emph{set of rigid factorization of $a$}.
\end{defi}

To simplify the notation, we make the following conventions:
\begin{itemize}
  \item If, for a rigid factorization $z = \rf[\varepsilon]{u_1,\ldots,u_k} \in \sZ^*(H)$, we have $k > 0$ (i.e., $\pi(z) \not \in H^\times$), then the unit $\varepsilon$ can be absorbed into the first factor $u_1$ (replacing it by $\varepsilon u_1$), 
    and we can essentially just work in $\cF(\cA(H)) / \sim$, with $\sim$ defined to match the equivalence relation on $H^\times \times_r \cF(\cA(H))$.
  \item If $H$ is reduced but $\card{H_0} > 1$, we often still write $\rf[s(u_1)]{u_1,\ldots, u_k}$ instead of the shorter $\rf{u_1,\ldots, u_k}$, as $k=0$ is allowed and in the path category there is a different empty path for every $e \in H_0$.
\end{itemize}

\begin{remark}
  \mbox{}
  \begin{enumerate}
    \item If $H$ is reduced, then $\sZ^*(H) = \cF(\cA(H))$.

      If $H$ is not reduced, the $H^\times$ factor allows us to represent trivial factorizations of units, and the equivalence relation $\sim$ allows us to deal with trivial insertion of units.
      In the commutative setting these technicalities can easily be avoided by identifying associated elements and passing to the reduced monoid $H_{\text{red}} = \{\, aH^\times \mid a \in H \,\}$. Unfortunately, associativity (left, right or two-sided) is in general no congruence relation in the non-commutative case.

    \item If $H$ is a commutative monoid, then $\sZ^*(H) \cong H^\times \times \cF(\cA(H_{\text{red}}))$, where $\cF(\cA(H_\text{red}))$ is the free monoid on $\cA(H_{\text{red}})$, while a \emph{factorization} in this setting is usually defined as an element of the free abelian monoid $\sZ(H) = \cF_{\text{ab}}(\cA(H_{\text{red}}))$, implying in particular that factorizations are unordered while rigid factorizations are ordered. The homomorphism $\pi: \sZ^*(H_\text{red}) \to H_{\text{red}}$ obviously factors through the multiplication homomorphism $\sZ(H_{\text{red}}) \to H_{\text{red}}$, and the fibers consist of the different permutations of a factorization.

      In the following we will only be concerned with invariants related to the lengths of factorizations, which may as well be defined using rigid factorizations.
  \end{enumerate}
\end{remark}

\begin{defi} \label{defi:arithinv}
  Let $a \in H$.
  \begin{enumerate}
    \item We call
      \[
        \sL(a) = \sL_H(a) = \{\, \length{z} \in \bN_0 \mid z \in \sZ^*(a) \,\}
      \]
      the \emph{set of lengths of $a$}.

    \item The \emph{system of sets of lengths of $H$} is defined as $\cL(H) = \{\, \sL(a) \subset \bN_0 \mid a \in H \,\}$.

    \item A positive integer $d \in \bN$ is a \emph{distance of $a$} if there exists an $l \in \sL(a)$ such that $\{\, l, l + d \,\} \in \sL(a)$ and $\sL(a) \cap [l+1,l+d-1] = \emptyset$. The \emph{set of distances of $a$} is the set consisting of all such distances and is denoted by $\Delta(a)=\Delta_H(a)$. 
     The \emph{set of distances of $H$} is defined as
     \[
        \Delta(H) = \bigcup_{a \in H} \Delta(a).
     \] 

    \item We define $\cU_k(H) = \bigcup_{\substack{L \in \cL(H) \\ k \in L}} L$ for $k \in \bN_0$.

    \item $H$ is \emph{half-factorial} if $\card{\sL(a)} = 1$ for all $a \in H$ (equivalently, $H$ is atomic and $\Delta(H) = \emptyset$).
  \end{enumerate}
\end{defi}

We write $b \mid_H^r a$ if $a \in Hb$ and similarly $b \mid_H^l a$ if $a \in bH$.

\begin{defilemma} Let $H \subset D$ be subcategories of a groupoid.
  The following are equivalent:
  \begin{enumerate}
    \enumequiv

    \item For all $a,b \in H$, $b \mid_D^r a$ implies $b \mid_H^r a$,
    \item $H H^{-1} \cap D = H$.
  \end{enumerate}
  $H \subset D$ is called \emph{right-saturated} if these equivalent conditions are fulfilled.
\end{defilemma}

\begin{proof}
  (a) $\Rightarrow$ (b): Let $c = ab^{-1}$ with $a,b \in H$, $t(a) = t(b)$ and $c \in D$. Then $cb = a$, i.e., $b \mid_D^r a$ and hence also $b \mid_H^r a$. Since the left factor is uniquely determined as $c = ab^{-1}$, it follows that $c \in H$.

  (b) $\Rightarrow$ (a): Let $b \mid_{D}^r a$. There exists $c \in D$ with $cb = a$, and thus $c = ab^{-1}$. Therefore $c \in HH^{-1} \cap D = H$, hence $b \mid_{H}^r a$.
\end{proof}

\begin{defi}
  Let $B$ be a reduced cancellative small category.
  A homomorphism $\theta\colon H \to B$ is called a \emph{transfer homomorphism} if it has the following properties:
  \begin{enumerate}
      \renewcommand{\theenumi}{T\arabic{enumi}}
      \renewcommand{\labelenumi}{(\theenumi)}

    \item\label{tr:T1} $B = \theta(H)$ and $\theta^{-1}(B_0) = H^\times$.
    \item\label{tr:T2} If $a \in H$, $b_1,b_2 \in B$ and $\theta(a) = b_1 b_2$, then there exist $a_1,a_2 \in H$ such that $a = a_1 a_2$, $\theta(a_1) = b_1$ and $\theta(a_2) = b_2$.
  \end{enumerate}
\end{defi}

The notion of a transfer homomorphism plays a central role in studying sets of lengths. It is easily checked that the following still holds in our generalized setting (cf. \cite[\S3.2]{ghk06} for the commutative case, \cite[Proposition 6.4]{geroldinger13} for the non-commutative monoid case).

\begin{prop} If $\theta\colon H \to B$ is a transfer homomorphism, then $\sL_H(a) = \sL_B(\theta(a))$ for all $a \in H$ and hence
  all invariants defined in terms of lengths coincide for $H$ and $B$. In particular,
  \begin{itemize}
    \item $\cL(H) = \cL(B)$,
    \item $\cU_k(H) = \cU_k(B)$ for all $k \in \bN_0$,
    \item $\Delta_H(a) = \Delta_B(\theta(a))$ for all $a \in H$, and $\Delta(H) = \Delta(B)$.
  \end{itemize}
\end{prop}

\begin{prop} \label{prop:zss}
  Let $H$ be a cancellative small category, $G$ a finite abelian group and $\theta\colon H \to \cB(G)$ a transfer homomorphism.
  Then $H$ is half-factorial if and only if $\card{G} \le 2$.
  If $\card{G} \ge 3$, then we have
  \begin{enumerate}
    \item\label{zss:delta-interval} $\Delta(H)$ is a finite interval, and if it is non-empty, then $\min \Delta(H) = 1$,
    \item\label{zss:uk-interval} for every $k \ge 2$, the set $\cU_k(H)$ is a finite interval,
    \item\label{zss:aamp} there exists an $M \in \bN_0$ such that for every $a \in H$ the set of lengths $\sL(a)$ is an almost arithmetical multiprogression (AAMP) with difference $d \in \Delta(H)$ and bound $M$.
  \end{enumerate}
\end{prop}

\begin{proof}
  By the previous lemma it is sufficient to show these statements for the monoid of zero-sum sequences $\cB(G)$ over a finite abelian group $G$. $\cB(G)$ is half-factorial if and only if $\card{G} \le 2$ by \cite[Proposition 2.5.6]{ghk06}. The first statement is proven in \cite{geroldinger-yuan12}, the second can be found in \cite[Theorem 3.1.3]{geroldinger09}. For the definition of AAMPs and a proof of 3 see \cite[Chapter 4]{ghk06}.
\end{proof}

The description in \ref{zss:aamp}. is sharp by a realization theorem of W.A.~Schmid \cite{schmid09}.

\section{Factorization of integral elements in arithmetical groupoids} \label{sec:fact}

In this section we introduce arithmetical groupoids and study the factorization behavior of integral elements.
In \autoref{sec:ideals} we will see that the divisorial fractional one-sided ideals of suitable semigroups form such groupoids.
Thus in non-commutative semigroups arithmetical groupoids generalize the free abelian group of divisorial fractional two-sided ideals familiar from the commutative setting (see \autoref{prop:agrpd-gen-freeab} and \autoref{rem:after-abstract-transfer}).
This abstract approach to factorizations was first used by Asano and Murata in \cite{asano-murata53}. We follow their ideas and also those of Rehm in \cite{rehm77-1,rehm77-2}, who studies factorizations of ideals in rings in a different abstract framework.
The notation and terminology for lattices follows \cite{graetzer11}, a reference for l-groups is \cite{steinberg10}.
\autoref{prop:fact-freeish-grpd} is the main result on factorizations of integral elements in a lattice-ordered groupoid (due to Asano and Murata). We introduce an abstract norm homomorphism $\eta$, and as the main result in this section, we present a transfer homomorphism to a monoid of zero-sum sequences in \autoref{thm:abstract-transfer}.

\begin{defi} 
  A lattice-ordered groupoid $(G, \le)$ is a groupoid $G$ together with a relation $\le$ on $G$ such that for all $e, f \in G_0$
  \begin{enumerate}
    \item $(G(e,\cdot),\, \le \mid_{G(e,\cdot)})$ is a lattice (we write $\meet'_e$ and $\join'_e$ for the meet and join),
    \item $(G(\cdot,f),\, \le \mid_{G(\cdot,f)})$ is a lattice (we write $\meet''_f$ and $\join''_f$ for the meet and join),
    \item $(G(e,f), \, \le \mid_{G(e,f)})$ is a sublattice of both $G(e,\cdot)$ and $G(\cdot,f)$.
      Explicitly: For all $a,b \in G(e,f)$ it holds that $a \meet'_e b = a \meet''_f b \in G(e,f)$ and $a \join'_e b = a \join''_f b \in G(e,f)$. 
  \end{enumerate}

  If $a,b \in G$ and $s(a) = s(b)$ we write $a \meet b = a \meet'_{s(a)} b$ and $a \join b = a \join'_{s(a)} b$. If $t(a) = t(b)$ we write $a \meet b = a \meet''_{t(a)} b$ and $a \join b = a \join''_{t(a)} b$. By 3 this is unambiguous if $s(a)=s(b)$ and $t(a)=t(b)$ both hold.
  The restriction of $\le$ to any of $G(e,\cdot)$, $G(\cdot,f)$ or $G(e,f)$ will in the following simply be denoted by $\le$ again. (Keep in mind however that $\le$ need not be a partial order on the entire set $G$, and $\meet$ and $\join$ do not represent meet and join operations on the entire set $G$ in the order-theoretic sense.)

  An element $a$ of a lattice-ordered groupoid is called \emph{integral} if $a \le s(a)$ and $a \le t(a)$, and we write $G_+$ for the subset of all integral elements of $G$.
\end{defi}

\begin{defi} \label{defi:arith-grpd}
  A lattice-ordered groupoid $G$ is called an \emph{arithmetical groupoid} if it has the following properties for all $e, f \in G_0$:

\begin{enumerate}
    \renewcommand{\theenumi}{P\textsubscript{\arabic{enumi}}}
    \renewcommand{\labelenumi}{(\theenumi)}

  \item \label{P:int} For $a \in G$, $a \le s(a)$ if and only if $a \le t(a)$.
  \item \label{P:mod} $G(e, \cdot)$ and $G(\cdot, f)$ are modular lattices.
  \item \label{P:ord} If $a \le b$ for $a, b \in G(e, \cdot)$ and $c \in G(\cdot,e)$, then $ca \le cb$. Analogously, if $a,b \in G(\cdot, f)$ and $c \in G(f, \cdot)$, then $ac \le bc$.
  \item \label{P:sup} For every non-empty subset $M \subset G(e,\cdot) \cap G_+$, $\sup(M) \in G(e,\cdot)$ exists, and similarly for $M \subset G(\cdot, f) \cap G_+$. 
    If moreover $M \subset G(e,f)$ then $\sup_{G(e,\cdot)}(M) = \sup_{G(\cdot,f)}(M)$.
  \item \label{P:bdd} $G(e,f)$ contains an integral element.
  \item \label{P:acc} $G(e, \cdot)$ and $G(\cdot, f)$ satisfy the ACC on integral elements.
\end{enumerate}
\end{defi}

\begin{center}
  \emph{For the remainder of this section, let $G$ be an arithmetical groupoid.}
\end{center}

\ref{P:bdd} implies in particular $G(e,f) \ne \emptyset$ for all $e,f \in G_0$, i.e., $G$ is connected.
If $e,e' \in G_0$ and $c \in G(e',e)$, then $G(e,\cdot) \to G(e',\cdot), x \mapsto c x$ is an order isomorphism by \ref{P:ord}, and similarly every $d \in G(f,f')$ induces an order isomorphism from $G(\cdot, f)$ to $G(\cdot, f')$.
\ref{P:mod} could therefore equivalently be required for a single $e$ and a single $f \in G_0$. Moreover, since the map $(G(e, \cdot), \le) \to (G(\cdot, e), \ge), x \mapsto x^{-1}$ is also an order isomorphism (\hyperref[gb:inv]{Lemma \ref*{lemma:gb}.\ref*{gb:inv}}) and the property of being modular is self-dual, it is in fact sufficient that one of $G(e,\cdot)$ and $G(\cdot,e)$ is modular for one $e \in G_0$.

Using \ref{P:bdd} we also observe that it is sufficient to have the ACC on integral elements on one $G(e,\cdot)$ and one $G(\cdot, f)$: If, say, $a_1 \le a_2 \le a_3 \le \ldots$ is an ascending chain of integral elements in $G(e',\cdot)$ and $c \in G(e,e')$ is integral, then $c a_1 \le c a_2 \le c a_3 \le \ldots$ is an ascending chain of integral elements in $G(e,\cdot)$ (\hyperref[gb:integral]{Lemma \ref*{lemma:gb}.\ref*{gb:integral}}), hence becomes stationary, and multiplying by $c^{-1}$ from the left again shows that the original chain also becomes stationary.

We summarize some basic properties that follow immediately from the definitions.

\begin{lemma} \label{lemma:gb}
  Let $e,f \in G_0$.
  \begin{enumerate}
    \item\label{gb:inv} $a \le x \Leftrightarrow a^{-1} \ge x^{-1}$ holds if either $a,x \in G(e,\cdot)$ or $a,x \in G(\cdot,f)$.
      In particular, for $a \in G$ the following are equivalent: (a) $a \le s(a)$; (b) $a \le t(a)$; (c) $a^{-1} \ge s(a)$; (d) $a^{-1} \ge t(a)$.
    \item\label{gb:integral}
      Let $a \in G(e,f)$. If $x \in G(\cdot,e)$ and $y \in G(f,\cdot)$ are integral, then $xa \le a$ and $ay \le a$.
    \item If $a \in G(e,f)$, $x \in G(\cdot,e)$ and $y \in G(f,\cdot)$, then
      \begin{enumerate}
        \item $x(a \join b) = xa \join xb$ and $x(a \meet b) = xa \meet xb$ if $b \in G(e,\cdot)$,
        \item $(a \join b)y = ay \join by$ and $(a \meet b)y = ay \meet by$ if $b \in G(\cdot,f)$.
      \end{enumerate}
    \item \label{gb:sup} Let $\emptyset \ne M \subset G(e,\cdot)$ and $x \in G(\cdot, e)$.
      If $\sup_{G(e,\cdot)}(M)$ exists, then also $\sup_{G(s(x),\cdot)}(xM)$ exists, and $\sup(xM) = x \sup(M)$.
      Moreover, then also $\inf_{G(\cdot,e)}(M^{-1})$ exists and $\inf(M^{-1}) = \sup(M)^{-1}$.
      Analogous statements hold for $\emptyset \ne M \subset G(\cdot, f)$ and $x \in G(f,\cdot)$.
    \item \label{gb:cc} $G(e,\cdot)$, $G(\cdot,f)$, $G(e,f)$ and in particular $G(e)$ are conditionally complete as lattices.
    \item The set $G_+$ of all integral elements forms a reduced wide subcategory of $G$, and $G = \quo(G_+)$
          is the groupoid of (left and right) fractions of this subcategory.
    \item For every $a \in G(e,f)$, there exist $b \in G(e)$ and $c \in G(f)$ with $b \le a$ and $c \le a$.
  \end{enumerate}
\end{lemma}

\begin{proof}
  \mbox{}
  \begin{enumerate}
    \item 
      Assume first $s(x)=s(a)$. By \ref{P:ord}, $a \le x$ if and only if $x^{-1}a \le t(x)$. By \ref{P:int} this is equivalent to $x^{-1}a \le t(a)$. Again by \ref{P:ord} this is equivalent to $x^{-1} \le a^{-1}$. The case $t(x)=t(a)$ is proven similarly.
      
      (a) $\Leftrightarrow$ (b) and (c) $\Leftrightarrow$ (d) by $\ref{P:int}$. 
      For (a) $\Leftrightarrow$ (c) set $x = s(a)$.

    \item Since $x \le t(x) = s(a)$, we have $xa \le s(a)a = a$ by \ref{P:ord}. Similarly, $by \le y$.

    \item 
      We show (i), (ii) is similar.
      Since $a \le a \join b$ and $b \le a \join b$, $\ref{P:ord}$ implies $xa \le x(a \join b)$ and $xb \le x(a \join b)$, thus $xa \join xb \le x(a \join b)$. Therefore
      \[
        a \join b = (x^{-1}xa) \,\join\, (x^{-1}xb) \le x^{-1}(xa \join xb),
      \]
      and multiplying by $x$ from the left gives $x(a \join b) \le xa \join xb$. Dually, $x(a \meet b) = xa \meet xb$.

    \item Let $c = \sup(M)$. Then for all $m \in M$, $xm \le xc$, hence $xc$ is an upper bound for $xM$. If $d \in G(s(x),\cdot)$ is another upper bound for $xM$, then $m \le x^{-1} d$ for all $m \in M$, hence $c \le x^{-1} d$ and thus $xc \le d$. Therefore $xc = \sup(xM)$.

      For $d \in G(e,\cdot)$ we have $m \le d$ for all $m \in M$ if and only if $m^{-1} \ge d^{-1}$ (in $G(\cdot,e)$), and $\inf(M^{-1}) = \sup(M)^{-1}$ follows.

    \item
      We show the claim for $G(e,\cdot)$, for $G(\cdot,f)$ the proof is similar.
      Let $\emptyset \ne M \subset G(e,\cdot)$ be bounded, say $x \le m \le y$ for some $x,y \in G(e,\cdot)$ and all $m \in M$.
      Then $y^{-1} M \subset G(t(y),\cdot)$ is integral, hence $\sup(y^{-1}M)$ exists by \ref{P:sup}, and $\sup(M) = y\sup(y^{-1}M)$ by \ref{gb:sup}. Similarly, $M^{-1} x \subset G(\cdot, t(x))$ is integral, and therefore $\sup(M^{-1} x)$ exists, implying $\inf(M) = \sup(M^{-1})^{-1} = x \sup(M^{-1}x)^{-1}$.

      The proof for $G(e,f)$ is similar but uses in addition $\sup_{G(t(y),\cdot)}(y^{-1}M) = \sup_{G(\cdot,f)}(y^{-1}M)$ (from \ref{P:sup}), to ensure that the supremum lies in $G(e,f)$ again.

    \item
      By 2 and the fact that every $e \in G_0$ is integral by definition, $G_+$ forms a wide subcategory of $G$.
      If $a \in G_+ \setminus G_0$, then $a < s(a)$, thus $a^{-1} > s(a)$ and therefore $a^{-1}$ is not integral. Hence the subcategory of integral elements is reduced.
      Let $x \in G$ and $e = s(x)$. Then $a = x \meet e \le e$, hence $a$ is integral. Since $a \le x$, also $x^{-1} a \le t(x)$ is integral. Set $b = x^{-1} a$. Then $x = ab^{-1}$ with $a,b \in G_+$. Similarly one can find $c,d \in G_+$ with $x = d^{-1} c$.
    \item By \ref{P:bdd} there exist integral $b' \in G(f,e)$ and $c' \in G(e,f)$. Set $b = ab'$ and $c = c'a$. Then $b \le a$, $c \le a$ and $b \in G(e)$, $c \in G(f)$.
      \qedhere
  \end{enumerate}
\end{proof}

For $e,f \in G$ it is immediate from the definitions that $G_+(e,\cdot) = G(e,\cdot) \cap G_+$,\; $G_+(\cdot, f) = G(\cdot,f) \cap G_+$ and $G_+(e,f) = G(e,f) \cap G_+$. Moreover, $G_+(e,\cdot)$ is a sublattice of $G(e,\cdot)$, $G_+(\cdot,f)$ is a sublattice of $G(\cdot,f)$, and $G_+(e,f)$ is a sublattice of $G(e,f)$.

If $a,b \in G_+(e,\cdot)$, then $a \le b$ if and only if $b \mid_{G_+}^l a$ as $a = b(b^{-1}a)$, and $b^{-1}a$ is integral if and only if $a \le b$. Similarly, if $a,b \in G_+(\cdot, f)$, then $a \le b$ if and only if $b \mid_{G_+}^r a$. Correspondingly, for integral elements with the same left (right) identity, we may view the join and meet operations as left (right) gcd and lcm.

\begin{defilemma} \label{defilemma:max-int}
  For $u \in G$ the following are equivalent:
  \begin{enumerate}
    \enumequiv

  \item $u$ is maximal in $G_+(s(u), \cdot) \setminus \{\, s(u) \,\}$,
  \item $u$ is maximal in $G_+(\cdot, t(u)) \setminus \{\, t(u) \,\}$,
    \item $u \in \cA(G_+)$.
  \end{enumerate}
  An element $u \in G$ satisfying these equivalent conditions is called \emph{maximal integral}.
\end{defilemma}

\begin{proof}
  (a) $\Rightarrow$ (b): By definition, $u$ is maximal in $G(s(u),\, \cdot)$ with $u < s(u)$. If $u \le y < t(u)$ with $y \in G(\cdot,\, t(u))$, then $u y^{-1} \le y y^{-1} = s(y)$, hence $u y^{-1} \in G(s(u),\, \cdot)$ is integral, and therefore $u < u y^{-1} \le s(u)$. By maximality of $u$ in the first set, therefore $u y^{-1} = s(u)$, whence $y=u$ and $u$ is maximal in the second set.

  (b) $\Rightarrow$ (c): 
  Assume $u=vw$ with $v,w \in G_+ \setminus G_0$. Then $u < w < t(u)$, contradicting the maximality of $u$ in $G_+(\cdot, t(u))$.

  (c) $\Rightarrow$ (a): Let $v \in G_+(s(u),\cdot)$ with $u \le v < s(u)$. Then $u = v(v^{-1}u)$ with $v$ and $v^{-1}u$ integral, and since $v \not \in G_0$ necessarily $v^{-1}u \in G_0$, i.e., $u=v$.
\end{proof}

\begin{lemma} \label{lemma:lgrp-primes}
  Let $U$ be an l-group.
  For $p \in U$ the following are equivalent:
  \begin{enumerate}
    \enumequiv

    \item $p$ is maximal integral,
    \item $p$ is a prime element in $U_+$,
    \item $p \in \cA(U_+)$.
  \end{enumerate}
\end{lemma}

\begin{proof}
  (a) $\Leftrightarrow$ (c) is shown as in \autoref{defilemma:max-int}. It suffices to show (a) $\Rightarrow$ (b) and (b) $\Rightarrow$ (c).
  Let $e$ be the identity of $U$.

  (a) $\Rightarrow$ (b): Let $p$ be maximal in $U_+$ with $p \ne e$.
  Assume $p \mid ab$ for $a,b \le e$. That means $ab \le p$. Assume $a \not \le p$.
  Then $b \,=\, (a \join p)b \,=\, ab \,\join\, pb \,\le\, p \,\join\, pb \,=\, p$, i.e., $p \mid b$.

  (b) $\Rightarrow$ (c): Let $p$ be a prime, $p = ab$ with $a, b \le e$. Say $p \mid a$, i.e., $a \le p$.
  Then $a \le p \le a$ implies $p=a$ and therefore $b=e$.
\end{proof}

\begin{prop} \label{prop:agrpd-gen-freeab}
  \mbox{}
  \begin{enumerate}
    \item If $G$ is a group (i.e., $\card{G_0} = 1$), then $G$ is the free abelian group with basis $\cA(G_+)$, and $G_+$ is the free abelian monoid with basis $\cA(G_+)$. Moreover $\gcd(a,b) = a \join b$ and $\lcm(a,b) = a \meet b$.

    \item Let $\cM$ be a set, $F$ the free abelian group with basis $\cM$, and $H \subset F$ the free abelian monoid with the same basis.
          A lattice order is defined on $F$ by $a \le b$ if $a = cb$ with $c \in H$, and $(F, \le)$ is an arithmetical groupoid with $F_+ = H$ and $\cM = \cA(F_+)$.

    \item \label{agf:vgrp} For every $e \in G_0$, the group isomorphism $j_e\colon G(e) \to \bG$ induces the structure of an arithmetical groupoid on $\bG$, and the induced structure on $\bG$ is independent of the choice of $e$.

  $\bG$  ($G(e)$) is a free abelian group and $\bG_+$ ($G(e)_+$) is a free abelian monoid with basis $\cA(\bG_+)$ ($\cA(G(e)_+)$). Moreover, $j_e(G(e)_+) = \bG_+$ and $j_e(\cA(G(e)_+)) = \cA(\bG_+)$.
  \end{enumerate}
\end{prop}

\begin{proof}
  \mbox{}
  \begin{enumerate}
    \item  $G$ is an l-group, and by \hyperref[gb:cc]{Lemma \ref*{lemma:gb}.\ref*{gb:cc}} it is conditionally complete. Therefore $G$ is commutative (\cite[Theorems 2.3.1(d) and 2.3.9]{steinberg10}). Since it satisfies the ACC on integral elements, $G_+$ is atomic (\autoref{prop:atomic}). By the previous lemma, every atom of $G_+$ is a prime element, and therefore $G_+$ is factorial.
      Because it is also reduced, $G_+$ is the free abelian monoid with basis $\cA(G_+)$. 
      Now $G=\quo(G_+)$ implies that $G$ is the free abelian group with basis $\cA(G_+)$. 
      Finally, $\gcd(a,b) = a \join b$ and $\gcd(a,b) = a \meet b$ for $a,b \in G_+$ follow because $c \le d$ if and only if $d \mid c$ for all $c,d \in G_+$.

    \item Clearly $\le$ defines a lattice order on $F$, and the properties of an arithmetical groupoid are, except for $\ref{P:mod}$, either trivial, or easily checked. For \ref{P:mod} recall that every l-group is distributive, hence modular, as a lattice (\cite[Theorem 2.1.3(a)]{steinberg10}). Now $F_+ = H$ and $\cA(F_+) = \cM$ are immediate from the definitions.

  \item For every $e \in G_0$ the vertex group $G(e)$ is an arithmetical groupoid, as is easily checked. Via the group isomorphism $j_e: G(e) \to \bG$ therefore $\bG$ gains the structure of an arithmetical groupoid. If $f \in G_0$ and $c \in G(e,f)$, then for all $x,y \in G(e)$ we have $x \le y \Leftrightarrow c^{-1}xc \le c^{-1}yc$, and since $j_f^{-1} \circ j_e(x) = c^{-1} x c$, the induced order on $\bG$ is independent of the choice of $e$.

  By 1., applied to $\bG$, respectively $G(e)$, the remaining claims follow (for $j_e(\cA(\bG_+))=\cA(G(e)_+)$ use the characterization of atoms as maximal integral elements from \autoref{lemma:lgrp-primes}).
  \qedhere
  \end{enumerate}
\end{proof}

Let $[a,b]$, $[c,d]$ be intervals in a lattice. Recall that $[a,b]$ is \emph{down-perspective} to $[c,d]$ if $c = a \meet d$ and $b = a \join d$. Moreover, $[a,b]$ is \emph{perspective} to $[c,d]$ if either $[a,b]$ is down-perspective to $[c,d]$, or $[c,d]$ is down-perspective to $[a,b]$. The intervals $[a,b]$, $[c,d]$ are \emph{projective} if there exists a finite sequence of intervals $[a,b]=[a_0,b_0], [a_1,b_1], \ldots, [a_k,b_k]=[c,d]$ such that $[a_{i-1},b_{i-1}]$ is perspective to $[a_i,b_i]$ for all $i \in [1,k]$. (See \cite[Chapter I.3.5]{graetzer11}.)

\begin{defi}
  \mbox{}
  \begin{enumerate}
    \item An element $a \in G_+(e,f)$ is \emph{transposable} to an element $b \in G_+(e',f')$ if there exists an element $c \in G_+(e,e')$ such that $[a,e]$ is down-perspective to $[cb,c]$.
      Explicitly: $b = c^{-1}(c \meet a)$ and $c \join a = e$.
    \item An element $a \in G_+$ is \emph{projective} to an element $b \in G_+$ if there exists a sequence of integral elements $a = c_0, c_1, \ldots, c_n, c_{n+1} = b$, such that for any pair of successive elements $(c_i, c_{i+1})$ either $c_i$ is transposable to $c_{i+1}$ or $c_{i+1}$ is transposable to $c_i$.
  \end{enumerate}
\end{defi}

It is easily checked that being transposable is a transitive and reflexive relation (but not symmetric), and projectivity is an equivalence relation. Note that in a modular lattice perspective intervals are isomorphic (\cite[p.308, Theorem 348]{graetzer11}), and therefore in particular the lengths of $[a, s(a)]$ and $[b,s(b)]$ coincide if $a$ is projective to $b$.

\begin{lemma}
  If $a, b, c$ are as in the definition of transposability, then $cb = c \meet a = ad$ for some $d \in G_+$ and $c$ is transposable to $d$.
\end{lemma}

\begin{proof}
  Since $c \meet a \le a$, there exists an integral $d$ with $ad = c \meet a$, and $cb=c \meet a$ by definition of transposability.
  The claim follows from $d = a^{-1}(c \meet a)$ and $c \join a = e$.
\end{proof}

\begin{lemma} \label{lemma:proj-is-proj}
  If two lattice intervals $[a,b]$ and $[c,d]$ of $G(e,\cdot)$ are projective, then the integral elements $b^{-1}a$ and $d^{-1}c$ are projective to each other in the sense of the previous definition.
\end{lemma}

\begin{proof}
  It suffices to show that if the lattice interval $[a,b]$ is down-perspective to $[c,d]$, then $b^{-1}a$ is transposable to $d^{-1}c$.
  Then $a = bx$, $d = by$ and $c = dz = byz$ with $x, y, z \in G_+$.
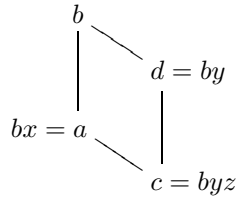
  \[
    \xy*\xybox{
      \xymatrix@R=0.7em@C=-0.5em{
        b \ar@{-}[rd] \ar@{-}[dd] & \\
                                  & *+![r]{d = by} \ar@{-}[dd] \\
       *+![l]{bx = a} \ar@{-}[rd] & \\
                                  & *+![r]{c = byz}
    } }\endxy
  \]
  Since $[a,b]$ is down-perspective to $[c,d]$, we get
  \begin{align*}
    b &= a \join d = bx \join by = b(x \join y)        &\text{ and therefore }& & x \join y &= s(x) = t(b), \text{ and } \\
    c &= a \meet d = bx \meet by = b(x \meet y)  &\text{ and therefore }& & x \meet y &= b^{-1}c = yz.
  \end{align*}
  Thus $d^{-1} c = z = y^{-1} (x \meet y)$ with $y \in G(t(b), t(d))$, and hence $x=b^{-1}a$ is transposable to $d^{-1}c$.
\end{proof}

\begin{defilemma}
  For every $a \in G$,
  \[
    \{\, \cX \in \bG \mid \cX_{s(a)} \le a \,\} = \{\, \cX \in \bG \mid \cX_{t(a)} \le a \,\},
  \]
  and we write $\bG_{\le a}$ for this set.
  The \emph{lower bound} $\Phi\colon G \to \bG$ is defined by $\Phi(a) = \sup(\bG_{\le a})$.
\end{defilemma}

\begin{proof}
  Let $\cX \in \bG$. Recall from \autoref{sec:groupoids} that $\cX_{s(a)} = j_{s(a)}^{-1}(\cX)$ denotes the unique representative of $\cX$ in $G(s(a))$, and that $\cX_{t(a)} = a^{-1} \cX_{s(a)} a$.
  We have to show that $\cX_{s(a)} \le a$ if and only if $\cX_{t(a)} \le a$, but this follows from $\cX_{s(a)} \le a \;\;\Leftrightarrow\;\; a^{-1} \cX_{s(a)} a \le a^{-1} a a = a$.
\end{proof}

With the definition of $\Phi$ and the notation of \autoref{sec:groupoids} we have: If $a \in G(e,f)$, then $\Phi(a)_e = \sup\{\, x \in G(e) \mid x \le a \,\} \in G(e)$,\; $\Phi(a)_f = \sup\{\, x \in G(f) \mid x \le a \,\} \in G(f)$,\; $\Phi(a)_f = a^{-1} \Phi(a)_e a = b^{-1} \Phi(a)_e b$ for all $b \in G(e,f)$ and $\Phi(a) = j_e(\Phi(a)_e) = j_f(\Phi(a)_f)$.

\begin{lemma} \label{lemma:max-int}
  Let $e,f \in G_0$.
  \begin{enumerate}
    \item If $a,b \in G(e,\cdot)$ or $a,b \in G(\cdot,f)$ with $a \le b$, then $\Phi(a) \le \Phi(b)$. In particular, if $a \in G_+$, then $\Phi(a) \in \bG_+$.
    \item \label{mi:phi-prod} If $a \in G(e,f)$, $b \in G(f, \cdot)$ then $\Phi(a)\Phi(b) \le \Phi(ab)$.
      If moreover $a, b \in G_+$, then $\Phi(ab) \le \Phi(a) \meet \Phi(b)$, and if furthermore $\Phi(a)$ and $\Phi(b)$ are coprime, then $\Phi(ab) = \Phi(a)\Phi(b)$.
    \item \label{max-int:prime} If $u \in \cA(G_+)$, then $\Phi(u) \in \bG_+$ is prime.
    \item Let $u, v \in G_+$ be projective elements.
      If $u \in \cA(G_+)$, then $v \in \cA(G_+)$, and $\Phi(u) = \Phi(v)$.
  \end{enumerate}
\end{lemma}

\begin{proof}
  \mbox{}
  \begin{enumerate}
    \item Immediate from the definition of $\Phi$.

    \item 
      Observe that $c = a^{-1} \Phi(a)_e \in G(f,e)$ is integral. Therefore
      \[
        \Phi(a)_e \Phi(b)_e = ac \Phi(b)_e = acc^{-1}\Phi(b)_f c = a \Phi(b)_f c \le a \Phi(b)_f \le ab,
      \]
      and hence $\Phi(a)\Phi(b) \le \Phi(ab)$.

      Let now $a,b$ be integral. Then 1. implies $\Phi(ab) \le \Phi(a)$ and $\Phi(ab) \le \Phi(b)$, so $\Phi(ab) \le \Phi(a) \meet \Phi(b)$.
      The last statement follows because $\Phi(a) \meet \Phi(b) = \lcm(\Phi(a),\Phi(b))$ in $\bG_+$.

    \item 
      By \autoref{lemma:lgrp-primes} it suffices to show $\Phi(u) \in \cA(\bG_+)$.
  If $e = s(u)$, then it suffices to prove $\Phi(u)_e \in \cA(G(e)_+)$ (by \hyperref[agf:vgrp]{Proposition \ref*{prop:agrpd-gen-freeab}.\ref*{agf:vgrp}}).
      Assume that $\Phi(u)_e = ab$ with $a,b \in G(e)$ such that $a < e$ and $b < e$.
      Then $b \join u = e$, since $b > ab = \Phi(u)_e$, and therefore
      \[
        u \,\ge\, ab \join au \,=\, a(b \join u) \,=\, a,
      \]
      a contradiction to $a > ab = \Phi(u)_e$.

    \item
      We first show that $v$ is maximal integral, and may assume that either $u$ is transposable to $v$ or $v$ is transposable to $u$.
      Let $e = s(u)$ and $f = s(v)$. Assume first that $u$ is transposable to $v$ via $c \in G_+(e,f)$.
      Then $[u,e]$ is down-perspective to $[cv,c]$, and since $G(e,\cdot)$ is modular, the intervals are isomorphic, hence have the same length (namely $1$). Multiplying from the left by $c^{-1}$ therefore also $[v,f]$ has length $1$, and thus $v$ is maximal integral. If $v$ is transposable to $u$, one argues along similar lines. 

      For the remainder of the claim we may now assume that $u$ is transposable to $v$ (since we already know that $v$ is also maximal integral). Let again $c \in G_+(e,f)$ be such that $cv = c \meet u$ and $e = c \join u$. If $p = \Phi(u)_e$, then $c^{-1}pc = \Phi(u)_f$. 
      Since $pc \le c \meet p \le c \meet u = cv$, we get $c^{-1}pc \le v$ and therefore $\Phi(v) \ge \Phi(c^{-1}pc) = \Phi(\Phi(u)_f) = \Phi(u)$. By \ref*{max-int:prime}., $\Phi(u)$ is prime and thus maximal integral in $\bG_+$, which implies $\Phi(u) = \Phi(v)$.
      \qedhere
  \end{enumerate}
\end{proof}

The converse of \hyperref[lemma:max-int]{Lemma \ref*{lemma:max-int}.\ref*{max-int:prime}} is false in general: A non-maximal integral element can have a prime lower bound. 

\begin{prop} \label{prop:fact-freeish-grpd}
  \mbox{}
  \begin{enumerate}
    \item \label{ffg:hf}
      The category $G_+$ is half-factorial.
      Explicitly: Every $a \in G_+$ possesses a rigid factorization
      \[
        \rf[s(a)]{u_1, \ldots, u_k} \in \sZ^*(a)
      \]
      with $k \in \bN_0$ and $u_1,\ldots,u_k \in \cA(G_+)$ and the number of factors, $k \in \bN_0$, is uniquely determined by $a$. 
      Moreover, if $\rf[s(a)]{v_1, \ldots, v_k} \in \sZ^*(a)$ is another rigid factorization with $v_1,\ldots,v_k \in \cA(G_+)$, then there exists a permutation $\tau \in \fS_k$ such that $u_{\tau(i)}$ is projective to $v_{i}$ for all $i \in [1,k]$. In particular, $\Phi(u_{\tau(i)}) = \Phi(v_{i})$ for all $i \in [1,k]$.

    \item \label{ffg:two}
      Any two rigid factorizations of $a \in G_+$ can be transformed into each other by a number of steps, each of which only involves replacing two successive elements by two new ones.

    \item \label{ffg:transposition}(Transposition.) 
      If $a = uv$ with $u,v \in \cA(G_+)$ and $\Phi(u) = \cP$, $\Phi(v)=\cQ$, $\cQ \ne \cP$, then there exist uniquely determined $v', u' \in \cA(G_+)$ such that $\Phi(v') = \cQ$, $\Phi(u') = \cP$ and $uv = v'u'$.

      Explicitly,
      \begin{align*}
        u' &= a \join \cP_{t(a)},  &  u' \meet v &= a,  &  u' \join v &= t(a),  \\
        v' &= a \join \cQ_{s(a)},  &  u  \meet v'&= a,  &  u \join v' &= s(a).
      \end{align*}
      So $u$ is transposable to $u'$ and $v'$ is transposable to $v$.

    \item \label{ffg:realize}
      Given any permutation $\tau' \in \fS_k$, there exist $w_1,\ldots,w_k \in \cA(G_+)$, such that
      \[
        \rf[s(a)]{w_1, \ldots, w_k} \in \sZ^*(a)
      \]
      and $\Phi(w_i) = \Phi(u_{\tau'(i)})$ for all $i \in [1,k]$.
  \end{enumerate}
\end{prop}

\begin{proof}
  \mbox{}
  \begin{enumerate}
    \item[1, 2.]
      We observe that rigid factorizations of $a$ correspond bijectively to maximal chains of the sublattice $[a,s(a)]$ of $G(s(a),\cdot)$:
      If $\rf[s(a)]{u_1, \ldots, u_k} \in \sZ^*(a)$, then $s(a) > u_1 > u_1 u_2 > \ldots > u_1\cdot\ldots\cdot u_k = a$ is a chain in $[a,s(a)]$ and since $u_1, \ldots, u_k$ are maximal integral, it is in fact a maximal chain of $[a, s(a)]$.
      Conversely, if $s(a)=x_0 > x_1 > x_2 > \ldots > x_k=a$ is a maximal chain of $[a,s(a)]$ then we set $u_i = x_{i-1}^{-1} x_i$ for all $i \in [1,k]$. These elements are maximal integral, i.e., atoms of $G_+$, and $a = x_k = s(a) x_1 (x_1^{-1} x_2)\cdot \ldots\cdot (x_{k-2}^{-1}x_{k-1}) (x_{k-1}^{-1} x_k) = s(a) u_1\cdot\ldots\cdot u_k$.

      By \ref{P:acc}, $[a,s(a)]$ satisfies the ACC, but also the DCC because if $s(a) = x_0 \ge x_1 \ge \ldots \ge a$ is a descending chain in $[a,s(a)]$, then $x_0^{-1} a \le x_1^{-1} a \le \ldots \le a^{-1}a = t(a)$ is an ascending chain in $G_+(\cdot, t(a))$ and therefore becomes stationary again by \ref{P:acc}. Being a modular lattice, $[a,s(a)]$ is therefore of finite length.

      The claims now follow from the Jordan-Hölder Theorem for modular lattices (see e.g., \cite[p.333, Theorem 377]{graetzer11}).
      The existence of maximal chains implies that $G_+$ is atomic (alternatively, use \autoref{prop:atomic} together with the ACC on integral elements). For half-factoriality, and projectivity of the factors, assume that $s(a)=x_0 > x_1 > x_2 > \ldots > x_k=a$ and $s(a)=y_0 > y_1 > y_2 > \ldots > y_l=a$ are two maximal chains from which rigid factorizations with factors $u_i = x_{i-1}^{-1}x_i$ for $i \in [1,k]$ and $v_i = y_{i-1}^{-1} y_i$ for $i \in [1,l]$ are derived.
      Then the uniqueness part of the Jordan-Hölder Theorem implies $k=l$ and that there exists a permutation $\tau \in \fS_k$ such that $[x_{\tau(i)},x_{\tau(i)-1}]$ is projective to $[y_i,y_{i - 1}]$ for all $i \in [1,k]$. 
      By \autoref{lemma:proj-is-proj}, this implies that $u_{\tau(i)}$ is projective to $v_i$ for all $i \in [1,k]$.

      Finally, \ref{ffg:two}. follows in a similar manner by induction on the length of $a$. 
      Fix a composition series of $[u_1 \meet v_1, a]$. This gives rise to refinements of $s(a) > u_1 > u_1 \meet v_1 > a$ and $s(a) > v_1 > u_1 \meet v_1 > a$ to composition series of $[a,s(a)]$. Applying the induction hypothesis to $\rf{u_2,\ldots,u_k}$ (respectively $\rf{v_2,\ldots,v_k}$), and the rigid factorization derived from the refined chain $t(u_1) > u_1^{-1}(u_1 \meet v_1) > \ldots > u_1^{-1}a$ (respectively $t(v_1) > v_1^{-1}(u_1 \meet v_1) > \ldots > v_1^{-1}a$) one proves the claim.

    \item[3.]
      Let $e = s(u)$, $q = \cQ_e$ and set $v' = uv \join q$.
      We first show:
      \begin{enumerate}
        \setlength\itemindent{3em}
        \item[Claim \textbf{A}.] $q \not \le u$,
        \item[Claim \textbf{B}.] $v' \meet u = uv$,
        \item[Claim \textbf{C}.] $v'$ is maximal integral.
      \end{enumerate}

      \emph{Proof of Claim \textbf{A}.} Suppose $q \le u$. Then $\cQ \le \Phi(u) = \cP$, a contradiction to $\cP$ and $\cQ$ being distinct prime elements of $\bG_+$ (\hyperref[max-int:prime]{Lemma \ref*{lemma:max-int}.\ref*{max-int:prime}}).

      \emph{Proof of Claim \textbf{B}.}
      Since $G(e,\cdot)$ is modular and $uv \le u$,
      \[
        v' \meet u = (uv \join q) \meet u = uv \join (q \meet u).
      \]
      Because $q \not \le u$ (Claim \textbf{A}), we have $q > q \meet u \ge qu$ and thus, by maximality of $u$, $qu = q \meet u$. Therefore
      \[
        uv \join (q \meet u) = uv \join qu = uv \join u\cQ_{t(u)} = u(v \join \cQ_{t(u)}) = uv.
      \]

      \emph{Proof of Claim \textbf{C}.} 
      Since $uv$ is a product of two atoms and $G_+$ is half-factorial, it suffices to show $uv < v' < e$.
      Suppose first $uv = v'$. Then $q \le uv \le u$, contradicting Claim \textbf{A}. Assume now $v' = e$. Then $u = e \meet u = v' \meet u$, and by Claim \textbf{B} therefore $u=uv$, contradicting $v < s(v)$.

      \emph{Existence.}
      We have $uv = v'u'$ with $v' \in \cA(G_+)$ (by Claim \textbf{C}) and $u' = v'^{-1}uv \in G_+$.
      Since $G_+$ is half-factorial, this necessarily implies $u' \in \cA(G_+)$.
      By definition of $v'$, $\cQ \le \Phi(v') < 1_{\bG}$, where the latter inequality is strict because $v' < e$.
      Thus $\Phi(v') = \cQ$ and \ref{ffg:hf}. implies $\Phi(u') = \cP$.

      \emph{Uniqueness.}
      If $v''u'' = uv$ with $\Phi(v'') = \cQ$, then $v'' < e$ and $v'' \ge uv \join q=v'$. By Claim \textbf{C}, $v'$ is maximal integral and thus $v'' = v'$, and then also $u''=u'$.

      \smallskip
      \emph{Explicit formulas.}
      Since $e \ge u \join v' \ge \Phi(u)_e \join \Phi(v')_e = \cP_e \join \cQ_e = (\cP \join \cQ)_e = (1_\bG)_e = e$, it follows that $u \join v' = e$.
      By Claim \textbf{B}, $u \meet v' = uv = a$.

      The equalities $u' = a \join \cP_{t(a)}$, $u' \meet v = a$ and $u' \join v = t(a)$ are shown similarly.

    \item[4.] Write $\tau'$ as a product of transpositions and use \ref{ffg:transposition}.
      \qedhere
  \end{enumerate}
\end{proof}

\hyperref[ffg:transposition]{Proposition \ref*{prop:fact-freeish-grpd}.\ref*{ffg:transposition}} gives an explicit and complete description of the possible relations between two maximal integral elements with coprime lower bound. The case where the lower bounds coincide is more complicated (there can be no relations, or many), but in the case where we will need it, it is quite simple (see \autoref{sec:dist-prelim}).

\begin{cor} \label{cor:bf-ff}
  If $H \subset G_+$ is a subcategory, then $\sL_H(a)$ is finite and non-empty for all $a \in H$.
  If for every prime $\cP \in \bG_+$ and all (equivalently, one) $e \in G_0$ the set $\{\, u \in \cA(G_+) \mid \Phi(u) = \cP \text{ and } s(u)=e\,\}$ is finite, then $\sZ^*_H(a)$ is finite for all $a \in H$.
\end{cor}

\begin{proof}
  Using that $G_+$ is reduced, it follows from \ref{P:acc} that $H$ satisfies the ACC on principal left and right ideals, and hence $\sZ_H^*(a) \ne \emptyset$.
  If $\rf[s(a)]{u_1,\ldots,u_k} \in \sZ_H^*(a)$ with $k \in \bN_0$ and $u_1,\ldots,u_k \in \cA(H)$, then in particular $u_i < s(u_i)$ for all $i \in [1,k]$, and therefore $k$ is bounded by the length of the factorization of $a$ in $G_+$.
  A similar argument shows the second claim.
\end{proof}

The properties that all sets of lengths, respectively that all sets of factorizations, are finite have been studied a lot in the commutative setting. Note, if $H$ is a commutative monoid and $a \in H$, then $\sZ_H(a)$ is finite if and only if $\sZ^*_H(a)$ is finite.

\begin{defilemma} \label{deflemma:abstract-norm}
  There exists a unique groupoid epimorphism $\eta\colon G \to \bG$ such that $\eta(u) = \Phi(u)$ for all $u \in \cA(G_+)$.
  We call $\eta$ the \emph{abstract norm} of $G$.
\end{defilemma}

\begin{proof}
  We need to show existence and uniqueness of such a homomorphism.
  Let $a \in G_+$, and let $\rf[s(a)]{u_1,\ldots,u_k} \in \sZ^*(a)$ with $u_1,\ldots,u_k \in \cA(G_+)$.
  Since the sequence of $\Phi(u_1), \ldots, \Phi(u_k)$ is, up to order, uniquely determined by $a$ (\autoref{prop:fact-freeish-grpd}), it follows that we can define $\eta(a) = \Phi(u_1) \cdot\ldots\cdot \Phi(u_k)$, and this is a homomorphism $G_+ \to \bG_+$ with $\eta(u) = \Phi(u)$ for all $u \in \cA(G_+)$.
  $G$ is the category of (left and right) fractions of $G_+$, and hence $\eta$ extends to a unique groupoid homomorphism $\eta\colon G \to \bG$.

 To verify that $\eta$ is surjective, let first $\cP \in \bG$ be a prime element of $\bG_+$, and let $e \in G_0$. Let $u \in G_+(e,\cdot)$ be a maximal integral element with $\cP_u \le u$. Then $\Phi(u) = \cP$, and therefore $\eta(u) = \cP$. The claim follows since $\bG$ is the free abelian group with basis $\cA(\bG_+)$.
\end{proof}

In general $\eta \ne \Phi$, since $\Phi$ need not be a homomorphism, but from \hyperref[mi:phi-prod]{Lemma \ref*{lemma:max-int}.\ref*{mi:phi-prod}} it follows that for integral $a$ the prime factorizations of $\Phi(a)$ and $\eta(a)$ have the same support and $\val_{\cP}(\eta(a)) \ge \val_{\cP}(\Phi(a))$ for all primes $\cP$ of $\bG_+$.

\begin{thm} \label{thm:abstract-transfer}
  Let $G$ be an arithmetical groupoid, $\eta\colon G \to \bG$ the abstract norm, $H$ a right-saturated subcategory of $G_+$, and $\cgrp = \bG / \quo(\eta(H))$. For $\cG \in \bG$ set $[\cG] = \cG \quo(\eta(H)) \in \cgrp$, and $\cgrp_M = \{\, [\eta(u)] \in \cgrp \mid \text{$u \in \cA(G_+)$} \,\}$.
  Assume that
  \begin{enumerate}
    \item for $a \in G$ with $s(a) \in H_0$, $a \in HH^{-1}$ if and only if $\eta(a) \in \quo(\eta(H))$.
    \item for every $e \in G_0$ and $g \in \cgrp_M$, there exists an element $u \in \cA(G_+)$ such that $s(u) = e$ and $[\eta(u)] = g$.
  \end{enumerate}

  Then there exists a transfer homomorphism $\theta\colon H \to \cB(\cgrp_M)$.
\end{thm}

\begin{proof}
  Let $\theta\colon H \to \cB(\cgrp_M)$ be defined as follows:
  For $a \in H$ and $\rf[s(a)]{u_1, \ldots, u_k} \in \sZ^*_{G_+}(a)$ with $u_1,\ldots,u_k \in \cA(G_+)$, set $\theta(a) = [\eta(u_1)] \cdot\ldots\cdot [\eta(u_k)] \in \cB(\cgrp_M)$ (in particular, identities are mapped to the empty sequence).
  We have to show that this definition depends only on $a$, and not on the particular rigid factorization into maximal integral elements chosen. Let $\rf[s(a)]{v_1, \ldots, v_k} \in \sZ^*_{G_+}(a)$ be another such rigid factorization. Then there exists a permutation $\tau \in \fS_k$ with $\eta(u_i) = \Phi(u_i) = \Phi(v_{\tau(i)}) = \eta(v_{\tau(i)})$ (due to \hyperref[ffg:hf]{Proposition \ref*{prop:fact-freeish-grpd}.\ref*{ffg:hf}} and by definition of $\eta$). Therefore $[\eta(u_1)] \cdot\ldots\cdot [\eta(u_k)] = [\eta(v_1)] \cdot\ldots\cdot [\eta(v_k)]$.

  With this definition $\theta$ is a homomorphism: Obviously $\theta(e) = \vec 1_{\cB(\cgrp_M)}$ for all $e \in H_0$, and if $b \in H$ with $t(a) = s(b)$ and $\rf[s(b)]{w_1, \ldots, w_l} \in \sZ^*(b)$ is a rigid factorization of $b$ into maximal integral elements, then $\rf[s(a)]{u_1,\ldots, u_k, w_1, \ldots, w_l}$ is a rigid factorization of $ab$. Thus
  \[
    \theta(ab) = [\eta(u_1)]\cdot\ldots\cdot[\eta(u_k)] [\eta(w_1)] \cdot\ldots\cdot [\eta(w_l)] = \theta(a) \theta(b).
  \]

  We still have to check that $\theta$ has properties \ref{tr:T1} and \ref{tr:T2}.

  If $\theta(a) = \vec 1_{\cB(\cgrp_M)}$, then $a$ possesses an empty factorization into maximal elements, hence $a \in G_0 \cap H = H_0 = H^\times$.

  $\theta$ is surjective: $\theta(e) = \vec 1_{\cB(\cgrp_M)}$ for any $e \in H_0$. Let $k \in \bN$ and $g_1 \cdot\ldots\cdot g_k \in \cB(\cgrp_M)$. By definition of $\cgrp_M$ and our second assumption, there exists an element $u_1 \in \cA(G_+)$ with $[\eta(u_1)] = g_1$ and $s(u_1) \in H_0$. Again by our second assumption, for all $i \in [2,k]$, there exist $u_i \in \cA(G_+)$ with $s(u_i) = t(u_{i-1})$ and $[\eta(u_i)] = g_i$. With $a=u_1\cdot\ldots\cdot u_k \in G$ we get $[\eta(a)] = [\eta(u_1)\cdot\ldots\cdot\eta(u_k)] = [\eta(u_1)] + \ldots + [\eta(u_k)] = \vec 0 \in \cgrp$ and $s(a) \in H_0$, and hence $\eta(a) \in \quo(\eta(H))$.
  By our first assumption, therefore $a \in HH^{-1}$, and since moreover $a$ is integral in $G$ and $H$ is right-saturated in $G_+$, we get $a \in H$ and $\theta(a) = g_1 \cdot\ldots\cdot g_k$.

  $\theta$ satisfies \ref{tr:T2}: Let $a \in H$, $\theta(a) = S T$ with $S,T \in \cB(\cgrp_M)$ and $S=g_1\cdot\ldots\cdot g_k$, $T = g_{k+1} \cdot\ldots\cdot g_l$, where $k \in \bN_0$ and $l \in \bN_{\ge k}$. By \hyperref[ffg:realize]{Proposition \ref*{prop:fact-freeish-grpd}.\ref*{ffg:realize}}, we can find a rigid factorization $\rf[s(a)]{u_1,\ldots,u_l} \in \sZ^*(a)$ with $u_i \in \cA(G_+)$ and $[\eta(u_i)] = g_i$ for all $i \in [1,l]$. Let $b = s(a) u_1 \cdot\ldots\cdot u_k$ and $c = t(b) u_{k+1}\cdot\ldots\cdot u_l$. Then $a=bc$. Since $s(b)\in H_0$ and $[\eta(b)] = \sigma(S) = \vec 0$, the first assumption implies $b \in HH^{-1} \cap G_+ = H$. Then $s(c) \in H_0$ and $c \in H$ follows similarly. Finally, $\theta(b) = S$ and $\theta(c) = T$.
\end{proof}

The theorem remains true if $H$ is a left-saturated subcategory of $G_+$, and in the first condition the set $HH^{-1}$ is replaced by $H^{-1} H$, and the condition $s(a) \in H_0$ is replaced by $t(a) \in H_0$. Similarly, one can replace the second condition by a symmetrical one, requiring $t(u) = e$ instead of $s(u)=e$ (in the proof of the surjectivity of $\theta$ one then first chooses $u_k$, followed by $u_{k-1}$ and so on).

\begin{remark} \label{rem:after-abstract-transfer}
  If $G$ is a group, then $G_+$ is the free abelian monoid with basis $\cA(G_+)$ (\autoref{prop:agrpd-gen-freeab}). As a saturated submonoid of this free abelian monoid, $H$ is therefore a reduced commutative Krull monoid \cite[Theorem 2.4.8]{ghk06}. Since $\eta = \id_G$ and $HH^{-1} = \eta(H) \eta(H)^{-1}$ the first condition is trivially satisfied, and because of $G_0 = \{\, 1 \,\}$, the second condition is also trivially satisfied.

  Conversely, let $H$ be a normalizing Krull monoid. Then $H_{\text{red}} = \{\, a H^\times \mid a \in H \,\}$ is a reduced commutative Krull monoid, isomorphic to the monoid of its non-zero principal ideals (\cite[Corollary 4.14]{geroldinger13}). The latter is a submonoid of the divisorial fractional ideals of $H$, which form the free abelian monoid of integral elements in the free abelian group of divisorial ideals of $H$. In this way we recover the well-known transfer homomorphism for Krull monoids as given for example in \cite[Proposition 3.4.8]{ghk06} for commutative Krull monoids, and in \cite[Theorem 6.5]{geroldinger13} for normalizing Krull monoids.

  We continue the discussion of normalizing Krull monoids in Remarks \hyperref[rem:ideal-struct:normalizing]{\ref*{rem:ideal-struct}.\ref*{rem:ideal-struct:normalizing}} and \hyperref[rtr:normalizing]{\ref*{rem:rtr}.\ref*{rtr:normalizing}}, where the divisorial two-sided ideal theory appears as a special case of the divisorial one-sided ideal theory.
\end{remark}

\section{Divisorial ideal theory in semigroups} \label{sec:ideals}

In this section we develop a divisorial one-sided ideal theory in semigroups. This follows again original ideas of Asano and Murata and generalizes the corresponding theory in rings and the theory of divisorial two-sided ideals in cancellative semigroups (see \cite{asano39,jacobson43,asano49,asano50, asano-ukegawa52, asano-murata53, deuring68} for classical treatments, and \cite{mcconnell-robson01, halterkoch98, halterkoch10, geroldinger13, jespers-okninski07} for more modern treatments in this area). In particular, the one-sided ideal theory of classical maximal orders over Dedekind domains is a special case of the theory presented here. 

The divisorial fractional one-sided ideals with left- and right-orders maximal in a fixed equivalence class of orders form a groupoid as studied in the previous section (this was in fact the motivation for Brandt to introduce the notion of a groupoid, see \cite{brandt27, brandt28}). We connect the factorization theory of elements of a maximal order $H$ with the one for the cancellative small category of integral principal ideals with left- and right-order conjugate to $H$, and apply results from the previous section to the factorization of elements in $H^\bullet$. The main result in this section is \autoref{thm:transfer-semigroup}. After having derived it we discuss in detail the case of rings, and of classical maximal orders (\autoref{sec:ideals-rings} and \autoref{sec:ring-csa}).

\smallskip
A semigroup $Q$ is called a \emph{quotient semigroup} if every cancellative element is invertible in $Q$, in short $Q^\bullet = Q^\times$.
A subsemigroup $H \subset Q$ is a \emph{right order} in $Q$ if $H (H \cap Q^\bullet)^{-1} = Q$, and $H$ is a \emph{left order} in $Q$ if $(H \cap Q^\bullet)^{-1} H = Q$. $H$ is an \emph{order} in $Q$ if it is a left and a right order.
We summarize the connection between a subsemigroup $H \subset Q$ being an order, and $Q$ being a semigroup of (left and right) fractions of $H$.

\begin{lemma}
  Let $Q$ be a quotient semigroup, and $H \subset Q$ a subsemigroup.
  \begin{enumerate}
    \item If $H$ is an order in $Q$, then $H^\bullet = H \cap Q^\bullet$ and $Q=\quo(H)$.
    \item If $\quo(H) = Q$, then $H^\bullet = H \cap Q^\bullet$ and $H$ is an order in $Q$.
    \item If $H$ is an order in $Q$, $H'$ is a subsemigroup of $Q$ and there exist $a,b \in Q^\bullet$ with $aHb \subset H'$, then $H'$ is an order in $Q$.
  \end{enumerate}
\end{lemma}

\begin{proof}
  \mbox{}
  \begin{enumerate}
    \item It suffices to show $H^\bullet = H \cap Q^\bullet$, and the inclusion $H \cap Q^\bullet \subset H^\bullet$ is clear.
      Let $a \in H^\bullet$, and $q,q' \in Q$ with $aq = aq'$. 
      Since $H$ is a right order in $Q$, there exist $c,d \in H$ and $s \in H \cap Q^\bullet$ with $q=cs^{-1}$ and $q'=ds^{-1}$, where we can choose a common denominator because $H \cap Q^\bullet$ satisfies the right Ore condition in $H$. Then $ac = ad$, and, because $a \in H^\bullet$, also $c=d$, showing $q=q'$. Since $H$ is a left order it follows in the same way that $a$ is right-cancellative in $Q^\bullet$, and hence $a \in H \cap Q^\bullet$.

    \item It again suffices to show $H^\bullet = H \cap Q^\bullet$, and this follows in the same way as in 1.

    \item It suffices to show that every $q \in Q$ has representations of the form $q = cs^{-1} = t^{-1}d$ with $c,d \in H'$ and $s,t \in H' \cap Q^\bullet$. Since $H$ is an order in $Q$, there exist $c',d' \in H$ and $s',t' \in H \cap Q^\bullet$ with $a^{-1}qa=c's'^{-1}$ and $bqb^{-1}=t'^{-1}d'$. Setting $c=ac'b$, $d=ad'b$, $s=as'b$ and $t=at'b$, the claim follows.
      \qedhere
  \end{enumerate}
\end{proof}

\begin{center}
  \emph{For the remainder of this section, let $Q$ be a quotient semigroup.}
\end{center}

If $H$ and $H'$ are orders in $Q$, then $H$ is (Asano-)equivalent to $H'$, written $H \sim H'$, if there exist $a,b,c,d \in Q^\bullet$ with $a H b \subset H'$ and $c H' d \subset H$. 
This is an equivalence relation on the set of orders in $Q$. 
An order $H$ is \emph{maximal} if it is maximal within its equivalence class with respect to set inclusion.

A feature of the non-commutative theory is that often there is no unique maximal order in a given equivalence class, and in fact in the most important cases we study there are usually infinitely many, but only finitely many conjugacy classes of them. In studying the divisorial one-sided ideal theory of a maximal order $H$, one has to study the ideal theory of all maximal orders in its equivalence class at the same time.

Let $H, H' \subset Q$ be subsemigroups (not necessarily orders), and let $X, Y \subset Q$. As in the previous sections $XY = \{\, xy \mid x \in X, y \in Y \,\}$.
$X$ is a \emph{left $H$-module} if $HX \subset X$, and a \emph{right $H'$-module} if $XH' \subset X$. It is an \emph{$(H,H')$-module} if it is a left $H$- and a right $H'$-module, i.e., $HXH' \subset X$. We define 
\[
  \rc{Y}{X} = \{\, q \in Q \mid Xq \subset Y \,\} \quad\text{and}\quad \lc{Y}{X} = \{\, q \in Q \mid qX \subset Y \,\}.
\]
Every left $H$-module is an $(H, \{1\})$-module, and similarly every right $H'$-module is a $(\{1\}, H')$-module.
We set $\cO_l(X) = \lc{X}{X}$ and $\cO_r(X) = \rc{X}{X}$.

\begin{lemma} \label{lemma:setfrac}
  Let $H, H'$ be subsemigroups of $Q$ and let $X$ be an $(H,H')$-module.
  \begin{enumerate}
    \item $\rc{H}{X}$ and $\lc{H'}{X}$ are $(H', H)$-modules.
    \item \label{setfrac:closure} $X \subset \lc{H}{\rc{H}{X}}$ and $X \subset \rc{H'}{\lc{H'}{X}}$.
    \item $\cO_l(X)$ and $\cO_r(X)$ are subsemigroups of $Q$.
    \item $\rc{\cO_l(X)}{X} = \lc{\cO_r(X)}{X} = \{\, q \in Q \mid XqX \subset X \,\}$.
    \item \label{setfrac:int} $X \subset \cO_l(X)$ if and only if $X \subset \cO_r(X)$ if and only if $X^2 \subset X$.
  \end{enumerate}
\end{lemma}

\begin{proof}
  \mbox{}
  \begin{enumerate}
    \item $XH'\rc{H}{X}H \,\subset\, X\rc{H}{X}H \,\subset HH = H$ and similarly for $\lc{H'}{X}$.

    \item $X\rc{H}{X} \subset H$ by definition of $\rc{H}{X}$ and thus $X \subset \lc{H}{\rc{H}{X}}$. The other identity is proven analogously.

    \item Clearly $1 \in \cO_l(X)$ and $\cO_l(X)\cO_l(X) X \subset \cO_l(X) X \subset X$ implies $\cO_l(X)\cO_l(X) \subset \cO_l(X)$. The claim for $\cO_r(X)$ is shown similarly.

    \item We show $\rc{\cO_l(X)}{X} = \{\, q \in Q \mid XqX \subset X \,\}$. Let $q \in Q$. Then $XqX \subset X$ is equivalent to $Xq \subset \cO_l(X)$, which in turn is equivalent to $q \in \rc{\cO_l(X)}{X}$.

    \item Immediate from the definitions of $\cO_l(X)$ and $\cO_r(X)$.
    \qedhere
  \end{enumerate}
\end{proof}

\begin{defi}
  For $X \subset Q$ as in \autoref{lemma:setfrac}, we define 
  \[
    X^{-1} = \rc{\cO_l(X)}{X} = \lc{\cO_r(X)}{X} = \{\, q \in Q \mid XqX \subset X \,\} \quad\text{and}\quad X_v = (X^{-1})^{-1}.
  \]
\end{defi}

\begin{defi} \label{defi:ideals}
  Let $H$ and $H'$ be orders in $Q$.
  \begin{enumerate}
    \item A \emph{fractional left $H$-ideal} is a left $H$-module $I$ such that $I \cap Q^\bullet \ne \emptyset$ and $\rc{H}{I} \cap Q^\bullet \ne \emptyset$. 
    \item A \emph{fractional right $H'$-ideal} is a right $H'$-module $I$ such that $I \cap Q^\bullet \ne \emptyset$ and $\lc{H'}{I} \cap Q^\bullet \ne \emptyset$.
    \item If $I$ is a fractional left $H$-ideal and a fractional right $H'$-ideal, then $I$ is a \emph{fractional $(H,H')$-ideal}.
    \item A \emph{fractional $H$-ideal} is a fractional $(H,H)$-ideal.
    \item $I$ is a \emph{left $H$-ideal} if it is a fractional left $H$-ideal and $I \subset H$. A \emph{right $H'$-ideal} is defined analogously.
    \item If $I$ is a left $H$-ideal and a right $H'$-ideal, then $I$ is an \emph{$(H,H')$-ideal}.
    \item An \emph{$H$-ideal} is an $(H,H)$-ideal.
    \item A fractional left $H$-ideal $I$ is \emph{integral} if $I \subset \cO_l(I)$ (equivalently, $I \subset \cO_r(I)$).
          The same definition is made for fractional right $H'$-ideals.
  \end{enumerate}
\end{defi}

If $H$ is a maximal order, then the notions of a left $H$-ideal and that of an integral fractional left $H$-ideal coincide (this will follow from Lemma \hyperref[frac:equiv]{\ref*{lemma:frac}.\ref*{frac:equiv}} and \hyperref[frac:max]{\ref*{lemma:frac}.\ref*{frac:max}}). We will sometimes call a fractional left (right) $H$-ideal \emph{one-sided} to emphasize that it need not be a fractional right (left) $H$-ideal, or \emph{two-sided} to emphasize that it is indeed a fractional $H$-ideal.

We recall some properties of fractional left $H$-ideals and first observe the following.

\begin{lemma} \label{lemma:pre-frac}
  If $H$ is an order in $Q$ and $I$ is a fractional left $H$-ideal, then $\cO_l(I)$ and $\cO_r(I)$ are orders.
  $I$ is a fractional $(\cO_l(I),\cO_r(I))$-ideal.
\end{lemma}

\begin{proof}
  Let $a \in I \cap Q^\bullet$, and let $b \in \rc{H}{I} \cap Q^\bullet$.
  By definition, $H \subset \cO_l(I)$ and since $H$ is an order and $\cO_l(I)$ a semigroup, it is also an order.
  $\cO_l(I) I \subset I$ and $b \in \rc{\cO_l(I)}{I}$ imply that $I$ is a fractional left $\cO_l(I)$-ideal. 
  Since $b \in \rc{\cO_l(I)}{I} = \lc{\cO_r(I)}{I}$, it holds that $b \cO_l(I) a \subset bI \subset \cO_r(I)$, and since $\cO_r(I)$ is a semigroup and $\cO_l(I)$ an order, $\cO_r(I)$ is also an order. Therefore $I$ is also a fractional right $\cO_r(I)$-ideal.
\end{proof}

The previous lemma implies that it is no restriction to require $I$ to be a fractional $(H,H')$-ideal over it, say, being a fractional left $H$-ideal (set $H'=\cO_r(I)$).

\begin{lemma} \label{lemma:frac}
  Let $H$ and $H'$ be orders in $Q$, and let $I$ be a fractional $(H,H')$-ideal.
  \begin{enumerate}
    \item \label{frac:equiv} The orders $H$, $H'$, $\cO_l(I)$ and $\cO_r(I)$ are all equivalent.
    \item \label{frac:max} If $H$ is maximal, then $\cO_l(I) = H$, and similarly, if $H'$ is maximal, then $\cO_r(I)=H'$.
    \item \label{frac:inv} $\rc{H}{I}$ is a fractional right $H$-ideal, and $\lc{H'}{I}$ is a fractional left $H'$-ideal.
    \item If $J$ is a fractional left $H$-ideal, then \label{frac:isec-union} $I \cap J$ and $I \cup J$ are fractional left $H$-ideals.
    \item If $(I_m)_{m \in M}$ is a non-empty family of left $H$-ideals for some index set $M$, then $\bigcup_{m \in M} I_m$ is a left $H$-ideal.
    \item If $H''$ is an order, and $K$ is a fractional $(H',H'')$-ideal, then $IK$ is a fractional $(H,H'')$-ideal.
  \end{enumerate}
\end{lemma}

\begin{proof}
    \mbox{}
  \begin{enumerate}
    \item By definition of the right and left order, $H \subset \cO_l(I)$ and $H' \subset \cO_r(I)$. Let $a \in I \cap Q^\bullet$, $b \in \rc{H}{I} \cap Q^\bullet$ and $c \in \lc{H'}{I} \cap Q^\bullet$. Then $\cO_l(I)ab \subset Ib \subset H$ and $ca\cO_r(I) \subset H'$, so $H \sim \cO_l(I)$ and $H' \sim \cO_r(I)$. Finally, $cHa \subset cI \subset H'$ and $aH'b \subset H$ imply $H \sim H'$.
    \item By \ref{frac:equiv}., $\cO_l(I) \sim H$ and by definition of the left order $H \subset \cO_l(I)$. Maximality of $H$ implies $H = \cO_l(I)$. Analogously, $H'= \cO_r(I)$ if $H'$ is maximal.
    \item $\rc{H}{I}$ is an $(H',H)$-module, and $\rc{H}{I} \cap Q^\bullet \ne \emptyset$ because $I$ is a fractional left $H$-ideal. Since $I \subset \lc{H}{\rc{H}{I}}$, also $\lc{H}{\rc{H}{I}} \cap Q^\bullet \ne \emptyset$, and thus $\rc{H}{I}$ is a fractional right $H$-ideal. Similarly one shows that $\lc{H'}{I}$ is a fractional left $H'$-ideal.
    \item Clearly $H(I\cap J) \subset I \cap J$ and if $c \in \rc{H}{I} \cap Q^\bullet$, then $(I\cap J)c \subset Ic \subset H$. It remains to show $I \cap J \cap Q^\bullet \ne \emptyset$. Let $a \in I \cap Q^\bullet$ and $b \in J \cap Q^\bullet$. Then $a = a's^{-1}$ and $b = b' s^{-1}$ with $a',b',s \in H \cap Q^\bullet$ (we can choose a common denominator using the right Ore condition). By the left Ore condition there exist $a'' \in H \cap Q^\bullet$ and $b'' \in H$ with $a''a' = b''b'$. Then $a'' a's^{-1} = b''b's^{-1} \in I \cap J \cap Q^\bullet$.

     For the union, again $H(I\cup J) \subset I \cup J$, and there exists $a \in (I \cup J) \cap Q^\bullet$. It remains to show $\rc{H}{(I \cup J)} \cap Q^\bullet \ne \emptyset$. But $\rc{H}{(I \cup J)} = \rc{H}{I} \cap \rc{H}{J}$, and we are done by applying our previous statement about the intersection to the fractional right $H$-ideals $\rc{H}{I}$ and $\rc{H}{J}$.
    \item Set $I = \bigcup_{m \in M} I_m$. Then $HI \subset I$ and $I \cap Q^\bullet \ne \emptyset$ are clear, and $I_m \subset H$ for all $m \in M$ implies $1 \in \rc{H}{I}$.
    \item Certainly $HIKH'' \subset IK$. If $a \in I \cap Q^\bullet$ and $b \in K \cap Q^\bullet$, then $ab \in IK \cap Q^\bullet$.
      Let $c \in \rc{H}{I} \cap Q^\bullet$ and $d \in \rc{H'}{K} \cap Q^\bullet$. Then $IKdc \subset IH'c \subset Ic \subset H$, i.e., $dc \in \rc{H}{IK} \cap Q^\bullet$. 
      If $c' \in \lc{H'}{I} \cap Q^\bullet$ and $d' \in \lc{H''}{K} \cap Q^\bullet$, then $d'c'IK \subset d'H'K \subset d'K \subset H''$, i.e., $d'c' \in \lc{H''}{IK} \cap Q^\bullet$. 
    \qedhere
  \end{enumerate}
\end{proof}

\begin{lemma} \label{lemma:ord}
  Let $H$ and $H'$ be orders in $Q$.
  \mbox{}
  \begin{enumerate}
    \item \label{ord:denom} If $H \sim H'$, then there exist $a,b \in H'^\bullet$ with $aH'b \subset H$. 
                            If moreover $H \subset H'$, then we can even take $a,b \in H^\bullet$.
    \item \label{ord:intermediate} If $H \sim H'$ and $H \subset H'$, then there exists an order $H''$ and $a,b \in H^\bullet$ such that $H \subset H'' \subset H'$ and $H''b \subset H$ and $aH' \subset H''$.
    \item \label{ord:conn} The following statements are equivalent:
      \begin{enumerate}
        \renewcommand{\theenumii}{(\alph{enumii})}
        \renewcommand{\labelenumii}{\theenumii}

        \item $H \sim H'$.
        \item There exists a fractional $(H,H')$-ideal.
        \item There exists a fractional $(H',H)$-ideal.
      \end{enumerate} 
  \end{enumerate}
\end{lemma}

\begin{proof}
  \mbox{}
  \begin{enumerate}
    \item There exist $x, y \in Q^\bullet$ with $xH'y \subset H$. Since $H'$ is an order in $Q$, $x=ac^{-1}$ and $y=d^{-1}b$ with $a,b \in H'$ and $c,d \in H' \cap Q^\bullet$. If $H \subset H'$ we even take $a,b,c,d \in H$.
      Then $aH'b \subset ac^{-1}H'd^{-1}b \subset H$.

    \item Using \ref{ord:denom}. choose $a,b \in H^\bullet$ with $aH'b \subset H$. Let $H'' = H \cup aH' \cup HaH'$. 
      Then it is easily checked that $H'' H'' \subset H''$ and obviously $H \subset H''$, thus $H''$ is an order. Moreover, $H'' \subset H'$,\, $H''b \subset H$ and $aH' \subset H''$, as claimed.

    \item
      (a) $\Rightarrow$ (b): By \ref{ord:denom}. there exist $a,b \in H^\bullet$ with $aHb \subset H'$ and $c,d \in H'^\bullet$ with $cH'd \subset H$.
            Define $I = HbcH'$. Clearly $I$ is an $(H,H')$-module with $bc \in I \cap Q^\bullet$.
            Since $aI = aHbcH' \subset H'cH' \subset H'$ and $Id=HbcH'd \subset HbH \subset H$, $I$ is a fractional $(H,H')$-ideal.

      (b) $\Rightarrow$ (a): By \hyperref[lemma:frac]{Lemma \ref*{lemma:frac}.\ref*{frac:equiv}}.

      (a) $\Leftrightarrow$ (c) follows by symmetry, swapping the roles of $H$ and $H'$.
    \qedhere
  \end{enumerate}
\end{proof}

\begin{lemma} \label{lemma:max-ord}
  Let $H$ be an order in $Q$. The following statements are equivalent.
  \mbox{}
  \begin{enumerate}
    \enumequiv

    \item $H$ is a maximal order.
    \item If $I$ is a fractional left $H$-ideal, then $\cO_l(I)=H$ and if $J$ is a fractional right $H$-ideal, then $\cO_r(I)=H$.
    \item If $I$ is a fractional $H$-ideal, then $\cO_r(I) = \cO_l(I) = H$.
    \item\label{max-ord:d} If $I$ is an $H$-ideal, then $\cO_r(I) = \cO_l(I) = H$.
  \end{enumerate}
\end{lemma}

\begin{proof}
  (a) $\Rightarrow$ (b): By \hyperref[frac:max]{Lemma \ref*{lemma:frac}.\ref*{frac:max}}, $\cO_l(I) = H$ and $\cO_r(J) = H$.

  (b) $\Rightarrow$ (c) $\Rightarrow$ (d): Trivial.

  (d) $\Rightarrow$ (a): Assume $H'$ is an order equivalent to $H$ and $H \subset H'$. Applying \hyperref[ord:intermediate]{Lemma \ref*{lemma:ord}.\ref*{ord:intermediate}} we find an equivalent order $H''$ with $H \subset H'' \subset H'$ and $a,b \in Q^\bullet$ with $aH' \subset H''$ and $H''b \subset H$. Let $I = \{\, x \in Q \mid H''x \subset H \,\}$. Then $I$ is an $H$-ideal, and $H'' \subset \cO_l(I)$, implying $H'' = H$ by \ref{max-ord:d}.
  Set $J = \{\, x \in Q \mid xH' \subset H \,\}$.
  Then $J$ is again an $H$-ideal (we use $a \in J$ since $H''=H$), and $H' \subset \cO_r(J)$ implies $H' = H$.
\end{proof}

\begin{lemma} \label{lemma:mid}
  Let $H$ be a maximal order in $Q$, let $I$ and $J$ be fractional left $H$-ideals, and let $K$ be a fractional left $\cO_r(I)$-ideal.
  \begin{enumerate}
    \item \label{mid:v} $\cO_l(I)=H$,\, $I^{-1}$ is a fractional right $H$-ideal with $\cO_r(I^{-1})=H$ and $I_v$ is a fractional left $H$-ideal with $\cO_l(I_v) = H$ and $I \subset I_v$. Moreover, $\cO_l(I)_v = H_v = H$.
    \item \label{mid:principal} If $a \in Q^\bullet$, then $Ha$ is a fractional left $H$-ideal with $\cO_r(Ha)=a^{-1}Ha$, $(Ha)^{-1} = a^{-1}H$ and $(Ha)_v = Ha$.
      $Ha$ is integral (equivalently, a left $H$-ideal), if and only if $a \in H^\bullet$.
    \item \label{mid:incl} If $I \subset J$ then $J^{-1} \subset I^{-1}$ and $I_v \subset J_v$.
    \item \label{mid:v-inv} $I \subset I_v = (I_v)_v$,\, $I_v^{-1} = (I^{-1})_v = I^{-1}$ and $\cO_l(I_v) = \cO_l(I) = \cO_l(I)_v = H$.
    \item \label{mid:isec} $(I_v \cap J_v)_v = I_v \cap J_v$.
    \item \label{mid:union} $(I \cup J)_v = (I_v \cup J)_v = (I \cup J_v)_v = (I_v \cup J_v)_v$.
    \item \label{mid:assoc} If $\cO_r(I)$ and $\cO_r(K)$ are also maximal, then $(IK)_v = (I_v K)_v = (I K_v)_v = (I_v K_v)_v$.
  \end{enumerate}
\end{lemma}

\begin{proof}
  \mbox{}
  \begin{enumerate}
    \item 
      By \hyperref[frac:max]{Lemma \ref*{lemma:frac}.\ref*{frac:max}}, $\cO_l(I) = H$, and thus \hyperref[frac:inv]{Lemma \ref*{lemma:frac}.\ref*{frac:inv}} implies that $I^{-1}= \rc{H}{I}$ is a right $H$-ideal. By the symmetric statement of what we just showed for fractional right $H$-ideals, therefore $\cO_r(I^{-1}) = H$ and $I_v$ is a fractional left $H$-ideal. Applying the first part of the statement to $I_v$ yields $\cO_l(I_v) = H$.
      Now $I \subset I_v$ follows from \hyperref[setfrac:closure]{Lemma \ref*{lemma:setfrac}.\ref*{setfrac:closure}}, and $H_v = H$ from $H^{-1} = \rc{H}{H} = H$.

    \item Since $a \in Ha \cap Q^\bullet$ and $a^{-1} \in \rc{H}{Ha}$, $Ha$ is a fractional left $H$-ideal.
      Certainly $a^{-1}Ha \subset \cO_r(Ha)$. Conversely, if $x \in \cO_r(Ha)$, then $ax \in Ha$ and thus $x \in a^{-1}Ha$, so that altogether $\cO_r(Ha) = a^{-1}Ha$.
      Moreover, $(Ha)a^{-1}H \subset H$ and if $Hax \subset H$ for $x \in Q^\bullet$, then $ax \in H$ and hence $x \in a^{-1}H$, implying $(Ha)^{-1} = a^{-1}H$. Finally, $(Ha)_v = Ha$ because $(Ha)_v = ((Ha)^{-1})^{-1} = Ha$.
      $Ha$ is a left $H$-ideal if and only if it is integral due to maximality of $H$, and $Ha \subset H$ if and only if $a \in H \cap Q^\bullet = H^\bullet$.

    \item If $x \in Q$ with $Jx \subset H$, then $Ix \subset Jx \subset H$, and hence $J^{-1} \subset I^{-1}$.
      By \ref{mid:v}., $J^{-1}$ and $I^{-1}$ are fractional right $H$-ideals with $\cO_r(I^{-1}) = \cO_r(J^{-1}) = H$, and we apply the symmetric statement for right fractional $H$-ideals to obtain $I_v \subset J_v$.

    \item 
      By \ref{mid:v}., $I^{-1}$ is a fractional right $H$-ideal with $\cO_r(I^{-1}) = H$, and $I_v$ is a fractional left $H$-ideal with $\cO_l(I_v) = H$. Moreover, also by 1, $I \subset I_v$ and $I^{-1} \subset (I^{-1})_v = [(I^{-1})^{-1}]^{-1} = I_v^{-1}$.
      It follows from \ref{mid:incl}., that $I_v \subset (I_v)_v$ and $I^{-1}=I_v^{-1}$. 
      Therefore $(I_v)_v = (I_v^{-1})^{-1} \subset (I^{-1})^{-1} = I_v$, whence $I_v = (I_v)_v$.

    \item $I_v \cap J_v \subset (I_v \cap J_v)_v \subset (I_v)_v \cap (J_v)_v = I_v \cap J_v$.

    \item $I \cup J \subset I_v \cup J \subset I_v \cup J_v \subset (I \cup J)_v$ and by taking divisorial closures, and \ref{mid:v-inv}., the claim follows.

    \item
      We use $\cO_r(I_v) = \cO_r(I)$ and $\cO_l(K) = \cO_l(K_v)$ (from \ref{mid:v}.).
      We have $IK \subset I_v K \subset I_v K_v$, and similarly $IK \subset I K_v \subset I_v K_v$ (by \ref{mid:v}.).
      By \ref{mid:incl}., this implies $(IK)_v \subset (IK_v)_v \subset (I_vK_v)_v$ and $(IK)_v \subset (I_vK)_v \subset (I_vK_v)_v$.

      To prove the claim it suffices to show $(I_v K_v)_v \subset (IK)_v$, which will follow from \ref{mid:incl}. and \ref{mid:v-inv}. if we show $I_v K_v \subset (IK)_v$.
      Since $IK (IK)^{-1} \subset \cO_l(I)$, we have $K(IK)^{-1} \subset \rc{\cO_l(I)}{I}=I^{-1} = I_v^{-1}$, where the last equality is due to \ref{mid:v-inv}. Multiplying by $I_v$ from the right gives $K(IK)^{-1} I_v \subset I_v^{-1} I_v$. By definition, $I_v^{-1} I_v \subset \cO_r(I_v)$. Since $\cO_r(I_v)=\cO_r(I)=\cO_l(K)$, therefore $K(IK)^{-1} I_v \subset \cO_l(K)$. Now $(IK)^{-1} I_v \subset K^{-1} = K_v^{-1}$ (using \ref{mid:v-inv}. again). Multiplying by $K_v$ from the right and using $\cO_r(K_v)=\cO_r(K)$, we obtain $(IK)^{-1} I_v K_v \subset \cO_r(K)$. Since $\cO_l((IK)^{-1})=\cO_r(K)$, this implies $I_v K_v \subset ((IK)^{-1})^{-1} = (IK)_v$.
    \qedhere
  \end{enumerate}
\end{proof}

\begin{defi} \label{defi:divisorial}
  Let $H$ be an order in $Q$. A fractional left or right $H$-ideal $I$ is called \emph{divisorial} if $I = I_v$.
\end{defi}

If $I$ is a fractional left $H$-ideal for a maximal order $H$, then it is not necessarily true that $\cO_r(I)$ is again a maximal order.
The next proposition shows that for divisorial fractional left or right $H$-ideals with $H$ maximal, already both, $\cO_l(I)$ and $\cO_r(I)$, are maximal. We can define an associative partial operation, the \emph{$v$-product}, by $I \cdot_v J = (IJ)_v$ when $J$ is a divisorial fractional left $\cO_r(I)$-ideal. Moreover it shows that every divisorial fractional left or right ideal is \emph{$v$-invertible}, i.e., invertible with respect to this operation.

\begin{prop} \label{prop:div-ideals}
  Let $H$ be a maximal order in $Q$.
  Let $I$ be a fractional left $H$-ideal. Then
  \begin{enumerate}
    \item\label{di:max} $\cO_l(I^{-1})$ is a maximal order. In particular, $\cO_r(I_v)$ is a maximal order.
    \item\label{di:inv} $(II^{-1})_v = \cO_l(I)$ and if $\cO_r(I)$ is also maximal, then $(I^{-1}I)_v = \cO_r(I)$.
    \item\label{di:assoc} If $I$ is a divisorial fractional left $H$-ideal, $J$ a divisorial fractional left $\cO_r(I)$-ideal and $K$ a divisorial fractional left $\cO_r(J)$-ideal, then
      \[
        (I \cdot_v J) \cdot_v K = I \cdot_v (J \cdot_v K).
      \]
  \end{enumerate}
\end{prop}

\begin{proof}
  \mbox{}
  \begin{enumerate}
    \item
      Because $H$ is maximal, $\cO_l(I) = H$.
      Trivially, $\cO_r(I) \subset \cO_l(I^{-1})$.
      Let $H' \supset \cO_l(I^{-1})$ be an order with $H' \sim \cO_l(I^{-1})$.
      Then $J = I H' I^{-1}$ is an $(H,H)$-module and if $a \in I \cap Q^\bullet$ and $b \in I^{-1} \cap Q^\bullet$,
      then $ab \in J$ and moreover
      \[
        J^2 = I H' I^{-1} I H' I^{-1} \subset I H' \cO_l(I^{-1}) H' I^{-1} = I H' I^{-1} = J,
      \]
      showing that $J$ is an integral left $\cO_l(J)$-ideal.

      We claim $H = \cO_l(J)$.
      Since $H = \cO_l(I) \subset \cO_l(J)$ and $H$ is maximal, it suffices to show $\cO_l(J) \sim H$.
      To this end we first show $b J a \subset H'$:
      \[
        bJa = bIH'I^{-1}a \subset I^{-1}IH'I^{-1}I \subset \cO_r(I)H'\cO_r(I) = H'.
      \]
      Since $H' \sim H$, there exist $c,d \in Q^\bullet$ with $cH'd \subset H$. Since $ab \in J \cap Q^\bullet$ therefore $cb(\cO_l(J)ab)ad \subset cbJad \subset H$, proving the claim.

      Therefore, from the definition of $J$, $H' I^{-1} \subset \rc{J}{I} \subset \rc{\cO_l(J)}{I} = \rc{H}{I} = I^{-1}$ and thus $H' \subset \cO_l(I^{-1})$, and, because we started out with the converse inclusion, also $H' = \cO_l(I^{-1})$.

      Now $\cO_r(I_v) = \cO_l(I^{-1})$ implies the ``in particular'' statement.

    \item 
      We have to show $(I I^{-1})_v = \cO_l(I)$ and $(I^{-1} I)_v = \cO_r(I)$, and we check the first equality as the second one then follows analogously. 
      The inclusion $I I^{-1} \subset \cO_l(I)$ implies $(I I^{-1})_v \subset \cO_l(I)_v = \cO_l(I)$. 
      It remains to prove $\cO_l(I) \subset (II^{-1})_v$.
      Due to maximality of $\cO_l(I)$, it holds that $\cO_l(II^{-1}) = \cO_l(I)$, and therefore $II^{-1} (II^{-1})^{-1} \subset \cO_l(I)$.
      Thus $I^{-1}(II^{-1})^{-1} \subset I^{-1}$, and $(II^{-1})^{-1} \subset \cO_r(I^{-1})=\cO_l(I)$. By \hyperref[mid:incl]{Lemma \ref*{lemma:mid}.\ref*{mid:incl}}, therefore $\cO_l(I) \subset (II^{-1})_v$.

    \item 
      Using \hyperref[mid:assoc]{Lemma \ref*{lemma:mid}.\ref*{mid:assoc}}, which can be applied due to 1., $((IJ)_v K)_v = (IJK)_v = (I(JK)_v)_v$.
      \qedhere
  \end{enumerate}
\end{proof}

\begin{cor}
  If $H$ is a maximal order, then every order $H'$ with $H' \sim H$ is contained in a maximal order equivalent to $H$.
\end{cor}

\begin{proof}
  By \hyperref[ord:conn]{Lemma \ref*{lemma:ord}.\ref*{ord:conn}}, there exists a fractional $(H,H')$-ideal $I$. Then $I_v$ is divisorial, $\cO_r(I_v) \sim \cO_l(I_v) = H$, $H' \subset \cO_r(I_v)$ and by \hyperref[di:max]{Proposition \ref*{prop:div-ideals}.\ref*{di:max}} $\cO_r(I_v)$ is maximal.
\end{proof}

\begin{cor} \label{cor:ideals-grpd}
  Let $\alpha$ denote an equivalence class of maximal orders of $Q$.
  Let
  \[
    \cF_v(\alpha) = \{\, I \mid \text{$I$ is a divisorial fractional left (or right) $H$-ideal with $H \in \alpha$} \,\}
  \]
  and
  \[
    \cI_v(\alpha) = \{\, I \mid \text{$I$ is a divisorial left (or right) $H$-ideal with $H \in \alpha$} \,\}.
  \]
  Then $(\cF_v(\alpha), \cdot_v, \subset)$ is a lattice-ordered groupoid, with identity elements the maximal orders in $\alpha$.
  If $I, J$ are in $\cF_v(\alpha)$ with $\cO_l(I)=\cO_l(J)$ or $\cO_r(I) = \cO_r(J)$, then $I \meet J = I \cap J$ and $I \join J = (I \cup J)_v$. Moreover, $\cI_v(\alpha)$ is the subcategory of integral elements.
\end{cor}

\begin{proof}
  For $I \in \cF_v(\alpha)$ we have $\cO_l(I) \cdot_v I = I = I \cdot_v \cO_r(I)$. If $J \in \cF_v(\alpha)$, then the $v$-product $I \cdot_v J$ is defined whenever $\cO_r(I) = \cO_l(J)$, and then $I \cdot_v J$ is a divisorial fractional $(\cO_l(I), \cO_r(J))$-ideal. The $v$-product is associative when it is defined (\hyperref[di:assoc]{Proposition \ref*{prop:div-ideals}.\ref*{di:assoc}}).
  Therefore $\cF_v(\alpha)$ with $\cdot_v$ as composition is a category where the set of identities is the set of maximal orders, $\alpha$, and for $I \in \cF_v(\alpha)$ we have $s(I) = \cO_l(I)$ and $t(I) = \cO_r(I)$. This category is a groupoid due to \hyperref[di:inv]{Proposition \ref*{prop:div-ideals}.\ref*{di:inv}}.

  On $\cF_v(\alpha)$ set inclusion defines a partial order, and obviously also the restrictions to $\{\, I \in \cF_v(\alpha) \mid \cO_l(I) = H \,\}$ and $\{\, I \in \cF_v(\alpha) \mid \cO_r(I) = H \,\}$ for $H \in \alpha$, given by set inclusion in these subsets, are partial orders.
  Let $I, J \in \cF_v(\alpha)$ with $\cO_l(I) = \cO_l(J)$. Then $I \cap J \in \cF_v(\alpha)$ and $(I \cup J)_v \in \cF_v(\alpha)$ (by Lemmas \ref{lemma:frac} and \ref{lemma:mid}), and clearly they are the infimum respectively supremum of $\{\, I, J \,\}$ in $\{\, I \in \cF_v(\alpha) \mid \cO_l(I) = H \,\}$, making this set lattice-ordered. Symmetric statements hold if $\cO_r(I) = \cO_r(J)$. If $\cO_l(I) = \cO_l(J)$ and $\cO_r(I) = \cO_r(J)$ both hold, then also $\cO_l(I \cap J) = \cO_l((I \cup J)_v) = \cO_l(I)$ and $\cO_r(I \cap J) = \cO_r((I \cup J)_v) = \cO_r(I)$ both hold. Therefore $(\cF_v(\alpha), \subset)$ is a lattice-ordered groupoid with the claimed meet and join.
  It is immediate from the definitions that $\cI_v(\alpha)$ is the subcategory of integral elements of this lattice-ordered groupoid.
\end{proof}

\begin{defilemma} \label{defilemma:bounded}
  An order $H$ is \emph{bounded} if it satisfies the following equivalent conditions:
  \begin{enumerate}
    \enumequiv

    \item Every fractional left $H$-ideal and every fractional right $H$-ideal contains a fractional (two-sided) $H$-ideal.
    \item \label{bdd:b} Every left $H$-ideal and every right $H$-ideal contains a (two-sided) $H$-ideal.
    \item \label{bdd:c} For all $a \in Q^\bullet$ there exist $b,c \in Q^\bullet$ such that $bH \subset Ha$ and $Hc \subset aH$.
    \item \label{bdd:d} For all $a \in Q^\bullet$ there exist $b,c \in Q^\bullet$ such that $Ha \subset bH$ and $aH \subset Hc$.
    \item For all $a \in Q^\bullet$, $HaH$ is a fractional (two-sided) $H$-ideal.
    \item For all $a \in Q^\bullet$ there exists a fractional (two-sided) $H$-ideal $I$ with $a \in I$.
    \item \label{bdd:g} If $M \subset Q$ and $a,b \in Q^\bullet$ with $aMb \subset H$, then there exist $c,d \in Q^\bullet$ with $cM \subset H$ and $Md \subset H$.
  \end{enumerate}
\end{defilemma}

\begin{proof}
  (a) $\Rightarrow$ (b): Trivial.

  (b) $\Rightarrow$ (c):
  Let $a \in Q^\bullet$. Then $a = d^{-1}c$ with $c,d \in H \cap Q^\bullet$, and $Ha = Hd^{-1}c \supset Hc$. By \ref{bdd:b}, $Hc$ contains an $H$-ideal $J$. If $b \in J \cap Q^\bullet$, then $bH \subset J \subset Ha$. The symmetric claim follows similarly.

  (c) $\Rightarrow$ (d): By \ref*{bdd:c} applied to $a^{-1}$, there exist $b,c \in Q^\bullet$ with $b^{-1}H \subset Ha^{-1}$ and $Hc^{-1} \subset a^{-1}H$. Then $Ha \subset bH$ and $aH \subset Hc$.

  (d) $\Rightarrow$ (e): $HaH$ is an $(H,H)$-module and contains the element $a \in Q^\bullet$. Let $b,c \in Q^\bullet$ with $Ha \subset bH$ and $aH \subset Hc$. Then $HaH \subset bH$ and $HaH \subset Hc$, hence $b^{-1} \in \lc{H}{HaH}$ and $c^{-1} \in \rc{H}{HaH}$.

  (e) $\Rightarrow$ (f): Trivial.

  (f) $\Rightarrow$ (g): $aM \subset Hb^{-1} \subset Hb^{-1}H$, and the latter being contained in a fractional $H$-ideal, there exists $a' \in Q^\bullet \cap \lc{H}{Hb^{-1}H}$ and thus $a'aM \subset H$. Similarly, $Mb \subset a^{-1}H \subset Ha^{-1}H$, and there exists a $b' \in Q^\bullet \cap \rc{H}{Ha^{-1}H}$. Thus $Mbb' \subset H$.

  (g) $\Rightarrow$ (a): Let $I$ be a fractional left $H$-ideal and $a \in I \cap Q^\bullet$. Then $(Ha^{-1})a \subset H$ and so there exists a $b \in Q^\bullet$ such that $b(Ha^{-1}) \subset H$, and thus $bH \subset Ha$. Therefore $HbH \subset Ha \subset I$ and $HbH$ is a fractional $H$-ideal contained in $I$ (as $b \in HbH$, $a^{-1} \in \rc{H}{HbH}$ and, by \ref*{bdd:g} again, there exists a $c \in Q^\bullet$ with $c(Hb) \subset H$, whence $c \in \lc{H}{HbH}$). The case where $I$ is a fractional right $H$-ideal is similar.
\end{proof}

\begin{lemma} \label{lemma:ord2}
  \mbox{}
  \begin{enumerate}
    \item Let $H$ and $H'$ be orders in $Q$. If $H$ is bounded and $H \sim H'$, then $H'$ is also bounded.
    \item \label{ord2:int-conn} Let $H$ and $H'$ be bounded equivalent maximal orders of $Q$. Then there exists an $(H,H')$-ideal $I$.
  \end{enumerate}
\end{lemma}

\begin{proof}
  \mbox{}
  \begin{enumerate}
    \item
      Because $H \sim H'$ and $H$ is bounded, there exist $c,d \in Q^\bullet$ with $cH \subset H' \subset dH$ (using \ref{bdd:c} and \ref{bdd:d} of the equivalent characterizations of boundedness). We verify condition \ref{bdd:c} for $H'$.
    Let $a \in Q^\bullet$.
    Then $cHa \subset H'a$, and there exists an $x \in Q^\bullet$ with $xH \subset Ha$. Then $cx H \subset H'a$ and finally $cxd^{-1}H' \subset cxH \subset H'a$. Similarly, one finds $z \in Q^\bullet$ with $H'z \subset aH'$.

    \item 
      We show: If $H$ and $H'$ are bounded equivalent maximal orders, then $H'H$ is a fractional $(H',H)$-ideal. 
      Then $(H'H)^{-1}$ is an $(H,H')$-ideal (since $\cO_r(H'H)=H$ by maximality of $H$ and $H \subset H'H$, \hyperref[mid:incl]{Lemma \ref*{lemma:mid}.\ref*{mid:incl}} implies $(H'H)^{-1} \subset H$; similarly, one shows $(H'H)^{-1} \subset H'$).

      Clearly $H'H$ is an $(H',H)$-module and $1 \in H'H$.
      We need to show that there exist $a, b \in Q^\bullet$ with $H'Ha \subset H'$ and $bH'H \subset H$.
      Since $H$ and $H'$ are bounded and equivalent there exist $a,b \in Q^\bullet$ with $Ha \subset H'$ and $bH' \subset H$, and the claim follows.
    \qedhere
  \end{enumerate}
\end{proof}

\begin{prop} \label{prop:ideals-are-lgrpd}
  Let $\alpha$ be an equivalence class of maximal orders in $Q$.
  $(\cF_v(\alpha), \cdot_v, \subset)$ is an arithmetical groupoid if and only if all $H \in \alpha$ (equivalently, one $H \in \alpha$) satisfy the following three conditions:
  \begin{enumerate}
    \renewcommand{\theenumi}{A\textsubscript{\arabic{enumi}}}
    \renewcommand{\labelenumi}{(\theenumi)}

  \item \label{A:acc} \label{A:first} $H$ satisfies the ACC on divisorial left $H$-ideals and the ACC on divisorial right $H$-ideals,
  \item \label{A:bdd} $H$ is bounded,
  \item \label{A:mod} \label{A:last} the lattice of divisorial fractional left $H$-ideals is modular, and the lattice of divisorial right $H$-ideals is modular.
  \end{enumerate}
\end{prop}

\begin{proof}
  From \autoref{cor:ideals-grpd} we already know that $\cF_v(\alpha)$ is a lattice-ordered groupoid.
  As in the discussion after \autoref{defi:arith-grpd} and from \autoref{lemma:ord2}, we see that if one representative $H \in \alpha$ satisfies \ref{A:first}--\ref{A:last}, then the same is true for all $H' \in \alpha$.

  Assume first that \ref{A:first}--\ref{A:last} hold. Then \ref{P:int} holds due to \hyperref[setfrac:int]{Lemma \ref*{lemma:setfrac}.\ref*{setfrac:int}}, property \ref{P:mod} is just \ref{A:mod} in the present setting.
  \ref{P:ord} follows easily: If $I \subset J$ are divisorial fractional left $H$-ideals, and $K$ is a divisorial fractional right $H$-ideal, then $KI \subset KJ$ and therefore $K \cdot_v I = (KI)_v \subset (KJ)_v = K \cdot_v J$, and similarly for the symmetric statement.
  \ref{P:sup} also holds: Let $(I_m)_{m \in M}$ be a non-empty family of divisorial left $H$-ideals. Then $(\bigcup_{m \in M} I_m)_v \subset H$ is also a divisorial left $H$-ideal, and if $(I_m)_{m \in M}$ is a family of $(H,H')$-ideals with $H' \in \alpha$, then the divisorial closure of the union is again an $(H,H')$-ideal.
  \ref{P:acc} is just \ref{A:acc}.
  \ref{A:bdd} implies \ref{P:bdd}:
  If $H,H' \in \alpha$, then there exists an $(H,H')$-ideal $I$ by \hyperref[ord2:int-conn]{Lemma \ref*{lemma:ord2}.\ref*{ord2:int-conn}}. Then $I_v$ is as required in \ref{P:bdd}.

  Assume now that $(\cF_v(\alpha), \cdot_v, \subset)$ is an arithmetical groupoid, and $H \in \alpha$. Then \ref{P:mod} implies \ref{A:mod}, and \ref{P:acc} implies \ref{A:mod}.
  From \ref{P:bdd} we can derive \ref{A:bdd}:
  Let $I$ be a fractional left $H$-ideal, and $x \in I \cap Q^\bullet$. Then $Hx \in \cF_v(\alpha)$. By \ref{P:bdd}, there exists $J \in \cI_v(\alpha)$ with $\cO_l(J) = \cO_r(Hx)$ and $\cO_r(J) = H$. Thus $HxJ$ is a fractional $H$-ideal, and $HxJ \subset I$. We proceed similarly if $I$ is a fractional right $H$-ideal.
\end{proof}

\begin{remark} \label{rem:ideal-struct}
  \mbox{}
  \begin{enumerate}
    \item From the discussion after \autoref{defi:arith-grpd}, we also see that we can equivalently formulate \ref{A:mod} as ``the lattice of divisorial fractional left (right) $H$-ideals is modular'', as the property for the other side then holds automatically.
      
    \item \label{rem:ideal-struct:normalizing}
    Let $H$ be a normalizing monoid.
    By definition of a monoid, $H$ satisfies the left and right Ore condition, hence it is an order in its quotient group. \hyperref[mid:principal]{Lemma \ref*{lemma:mid}.\ref*{mid:principal}} shows that every fractional left or right $H$-ideal is in fact already a two-sided $H$-ideal, and thus $H$ is bounded.

    Assume that $H$ is a normalizing Krull monoid.
    Then $\alpha = \{\, H \,\}$, and the lattice-ordered groupoid $\cF_v(\alpha)$ is in fact a group.
    The lattice of divisorial fractional $H$-ideals is then modular, even distributive \cite[Theorem 2.1.3(a)]{steinberg10}, and hence by the previous theorem an arithmetical groupoid.
  \end{enumerate}
\end{remark}

\begin{defi}
  We call a maximal order $H$ satisfying \ref{A:first}--\ref{A:last} an \emph{arithmetical maximal order}. If $\alpha$ is its equivalence class of arithmetical maximal orders, then we denote by $\cM_v(\alpha) \subset \cI_v(\alpha)$ the (quiver of) maximal integral elements.
\end{defi}

\begin{center}
  \emph{Let from here on $H$ be an arithmetical maximal order in $Q$, and let $\alpha$ be its equivalence class of arithmetical maximal orders.}
\end{center}

By \hyperref[mid:principal]{Lemma \ref*{lemma:mid}.\ref*{mid:principal}}, every principal left ideal $Ha$ with $a \in H^\bullet$ is a divisorial left $H$-ideal with inverse $a^{-1}H \in \cF_v(\alpha)$.
Let
\[
  \cH(\alpha) = \{\, H'a \in \cI_v(\alpha) \mid H' \in \alpha,\, a \in H'^\bullet \,\}.
\]
The $v$-product coincides with the usual proper product on $\cH(\alpha)$.
Thus $(\cH(\alpha),\cdot) \subset (\cI_v(\alpha), \cdot_v)$ is a wide subcategory, with the product $I J = I \cdot J = I \cdot_v J$ for $I,J \in \cH(\alpha)$ defined whenever $\cO_r(I) = \cO_l(J)$, and then $\cO_l(IJ)=\cO_l(I)$ and $\cO_r(IJ) = \cO_r(J)$. 
The inclusion $(\cH(\alpha),\cdot) \subset (\cI_v(\alpha), \cdot_v)$ is left- and right-saturated.
By $\cH_H(\alpha)$ (or shorter, $\cH_H$, since $H$ determines $\alpha$) we denote the subcategory of $\cH(\alpha)$ where the left and right orders of every element are not only equivalent but in fact conjugate to $H$. Explicitly,
\[
  \cH_H = \cH_H(\alpha) = \{\, d(Ha)d^{-1} \in \cI_v(\alpha) \mid a \in H^\bullet,\, d \in Q^\bullet \,\}.
\]
If $H' \in \alpha$, then $\cH_H = \cH_{H'}$ if and only if $H'$ and $H$ are conjugate.
Again the inclusion $(\cH_H,\cdot) \subset (\cH(\alpha), \cdot)$ is left- and right-saturated, and thus so is the inclusion $(\cH_H, \cdot) \subset (\cI_v(\alpha), \cdot_v)$.

The following simple lemma gives a correspondence between $H$ and $\cH_H$.

\begin{lemma} \label{lemma:fact-bij-pre}
  Let $d \in Q^\bullet$.
  \begin{enumerate}
    \item If $a, a_1,a_2 \in H^\bullet$ with $a = a_1 a_2$, then $d^{-1}(Ha)d = d^{-1}(Ha_2)d \cdot d^{-1}a_2^{-1}(Ha_1)a_2d \in \cH_H$ with $d^{-1}(Ha_2)d,\, d^{-1}a_2^{-1}(Ha_1)a_2d \in \cH_H$.
     
    \item If $a \in H^\bullet$ and $d^{-1}(Ha)d = I_2 \cdot I_1$ with $I_1, I_2 \in \cH_H$, then there exist $a_1,a_2 \in H^\bullet$ with $I_2 = d^{-1}(Ha_2)d$, $I_1 = d^{-1}a_2^{-1}(Ha_1)a_2d$ and $a = a_1 a_2$.

    \item If $a_1,a_2,b_1,b_2 \in H^\bullet$ with $Ha_2 = Hb_2$ and $a_2^{-1}(Ha_1)a_2 = b_2^{-1}(Hb_1)b_2$, then there exist $\varepsilon_1, \varepsilon_2 \in H^\times$ with $b_1 = \varepsilon_1 a_1 \varepsilon_2^{-1}$ and $b_2 = \varepsilon_2 a_2$.
  \end{enumerate}

  In particular, for $a \in H^\bullet$ we have $a \in \cA(H^\bullet)$ if and only if $d^{-1}(Ha)d \in \cA(\cH_H)$.
\end{lemma}

\begin{proof}
  \mbox{}
  \begin{enumerate}
    \item The multiplication is defined because $\cO_r(d^{-1}(Ha_2)d) = d^{-1}a_2^{-1}H a_2 d = \cO_l(d^{-1}a_2^{-1}(Ha_1)a_2d)$. The remaining statements are then clear.
    \item Since $\cO_l(I_2) = d^{-1}Hd$ we have $I_2 = d^{-1}Hd a_2'$ with $a_2' \in (d^{-1}Hd)^\bullet$, and hence, with $a_2 = d a_2' d^{-1} \in H^\bullet$, $I_2 = d^{-1}(Ha_2)d$. Then $\cO_l(I_1) = \cO_r(I_2) = d^{-1}a_2^{-1}Ha_2d$, and therefore similarly $I_1 = d^{-1}a_2^{-1}(Ha_1')a_2d$ with $a_1' \in H^\bullet$. Hence $d^{-1}(Ha)d = d^{-1}(Ha_1'a_2)d$, and thus $a = \varepsilon a_1' a_2$ with $\varepsilon \in H^\times$. Taking $a_1 = \varepsilon a_1'$ the claim follows.

    \item Since $Ha_2 = Hb_2$, there exists an $\varepsilon_2 \in H^\times$ with $b_2 = \varepsilon_2 a_2$. Then
      \[
        a_2^{-1}(Ha_1)a_2 = b_2^{-1}(Hb_1)b_2 = a_2^{-1}\varepsilon_2^{-1} (H b_1) \varepsilon_2 a_2 = a_2^{-1}(Hb_1 \varepsilon_2)a_2,
      \]
      and thus there exists $\varepsilon_1 \in H^\times$ with $\varepsilon_1 a_1 = b_1 \varepsilon_2$, i.e., $b_1 = \varepsilon_1 a_1 \varepsilon_2^{-1}$.
    \qedhere
  \end{enumerate}
\end{proof}

Observe that we may view a rigid factorization $\rf{Ha_2, a_2^{-1}(Ha_1)a_2} \in \sZ^*(\cH_H)$ as a multiplicative way of writing the chain $H \supset Ha_2 \supset H a_1 a_2$.

\begin{prop} \label{prop:fact-bij}
  Let $a \in H^\bullet$.
  For every $d \in Q^\bullet$ there is a bijection $\sZ_H^*(a) \to \sZ_{\cH_H}^*(d^{-1}(Ha)d)$, given by
  \[
    \rf{u_1, u_2, \ldots, u_k} \,\mapsto\, \rf{d^{-1} (H u_k) d, (d^{-1} u_k^{-1} (H u_{k-1}) u_k d), \ldots, (d^{-1} u_k^{-1} \cdot\ldots\cdot u_2^{-1} (H u_1) u_2 \cdot\ldots\cdot u_k d)}.
  \]

  If $\lbar \theta\colon \cH_H \to B$ is a transfer homomorphism to a reduced cancellative small category $B$ and having the additional property that $\lbar \theta(d^{-1}(Ha)d) = \lbar \theta(Ha)$ for all $a \in H^\bullet$ and $d \in Q^\bullet$, then it induces a transfer homomorphism $\theta\colon H^\bullet \to B^{\text{op}}$ given by $\theta(a) = \lbar\theta(Ha)$.
\end{prop}

\begin{proof}
  The claimed bijection follows by iterating the previous lemma.

  We need to verify that $\theta$ is a transfer homomorphism and first check that $\theta$ is a homomorphism: For $a,b \in H^\bullet$
  \[
    \theta(ab) = \lbar\theta(Hab) = \lbar\theta( Hb \cdot b^{-1}(Ha)b ) = \lbar\theta(Hb) \cdot \lbar\theta(b^{-1}(Ha)b) = \lbar\theta(Hb) \cdot \lbar\theta(Ha) = \theta(a) \cdot^{\text{op}} \theta(b),
  \]
  and if $a \in H^\times$ then $Ha = H$, whence $\theta(a) = \lbar\theta(H) \in B_0$. We verify \ref{tr:T1}: Let $b \in B$. Then there exist $d \in Q^\bullet$ and $a \in H^\bullet$ with $\lbar\theta(d^{-1}(Ha)d) = b$, hence $\theta(a) = b$. If $a \in H^\bullet$ with $\theta(a) \in B_0$, then $\lbar \theta(Ha) \in B_0$, hence $Ha \in (\cH_H)_0$, i.e., $Ha = H$ and $a \in H^\times$. It remains to check \ref{tr:T2}: Let $a \in H^\bullet$ and $b_1,b_2 \in B$ with $\theta(a) = b_1 \cdot^\text{op} b_2$. Then $\lbar\theta(Ha) = b_2 b_1$, hence there exist $a_1,a_2 \in H^\bullet$ with $Ha = Ha_2 \cdot a_2^{-1}(Ha_1) a_2$ and $\lbar\theta(Ha_2) = b_2$, $\lbar\theta(a_2^{-1}(Ha_1)a_2) = b_1$. This implies $a = \varepsilon a_1 a_2$ with $\varepsilon \in H^\times$, and $\theta(\varepsilon a_1) = b_1$, $\theta(a_2) = b_2$.
\end{proof}

\begin{remark}
  The condition $\lbar \theta(d^{-1}(Ha)d) = \lbar \theta(Ha)$ implies in particular $\card{\lbar\theta(\cH_H)_0} = 1$.
  Thus in fact $B$ is necessarily a semigroup.
\end{remark}

Let $\bG$ be the universal vertex group of $\cF_v(\alpha)$, and let $\eta\colon \cF_v(\alpha) \to \bG$ be the abstract norm, as defined in the previous section.

\begin{lemma}
  \mbox{}
  \begin{enumerate}
    \item If $I \in \cF_v(\alpha)$ and $d \in Q^\bullet$, then $\eta(d^{-1}Id) = \eta(I)$.
    \item $\quo( \eta(\cH_H) ) = \quo( \{\, \eta(Ha) \mid a \in H^\bullet \,\} ) = \{\, \eta(Hq) \mid q \in Q^\bullet\,\}$.
  \end{enumerate}
\end{lemma}

\begin{proof}
  \mbox{}
  \begin{enumerate}
    \item It suffices to verify the claim for maximal integral $I \in \cI_v(\alpha)$.
      If $P \in \cI_v(\alpha)$ is the maximal divisorial two-sided $\cO_l(I)$-ideal contained in $I$, then $d^{-1} P d$ is the maximal divisorial two-sided ideal contained in $d^{-1}Id$, and since $d^{-1} P d = (\cO_l(d^{-1}Id) d^{-1} \cO_l(I)) \cdot_v P \cdot_v (\cO_l(I) d \cO_l(d^{-1}Id))$ we have $\eta(I) = (P)=(d^{-1}Pd) = \eta(d^{-1}Id) \in \bG$.
    \item The first equality is immediate from 1. For the second equality, note that if $q = ab^{-1}$ with $a,b \in H^\bullet$, then (using 1. multiple times and the fact that $\eta$ is a homomorphism)
      \[
        \begin{split}
          \eta(Hq) &= \eta(Hab^{-1}) = \eta(Hb^{-1} \cdot b(Ha)b^{-1}) = \eta(Hb^{-1}) \eta(b(Ha)b^{-1}) = \eta(bH)^{-1} \eta(Ha) \\
                   &= \eta(b^{-1}(bH)b)^{-1} \eta(Ha) = \eta(Ha) \eta(Hb)^{-1}. \qedhere
        \end{split}
      \]
  \end{enumerate}
\end{proof}

Applying \autoref{thm:abstract-transfer} to the present situation, we obtain a transfer homomorphism $\cH_H \to \cB(\cgrp_M)$ if we impose some additional crucial conditions on $H$.

\begin{thm} \label{thm:transfer-semigroup}
  Let $Q$ be a quotient semigroup, $H$ an arithmetical maximal order in $Q$, and $\alpha$ its equivalence class of arithmetical maximal orders.
  \begin{enumerate}
    \item 
      For all $a \in H^\bullet$, $\sL_{H^\bullet}(a)$ is finite and non-empty.
      If, for every maximal divisorial $H$-ideal $P$, the number of maximal divisorial left $H$-ideals $I$ with $P \subset I$ is finite, then $\sZ^*_{H^\bullet}(a)$ is finite for all $a \in H^\bullet$.

    \item Let $P_{H^\bullet} = \{\, \eta(Hq) \mid q \in Q^\bullet \,\} \subset \bG$, $\cgrp = \bG / P_{H^\bullet}$, and $\cgrp_M = \{\, [\eta(I)] \in \cgrp \mid I \in \cI_v(\alpha) \text{ maximal integral} \,\}$.
      Assume:
      \begin{enumerate}

        \item A divisorial fractional left $H$-ideal $I$ is principal if and only if $\eta(I) \in P_{H^\bullet}$.
        \item For all $H' \in \alpha$ and all $g \in \cgrp_M$ there exists a maximal divisorial left $H'$-ideal with $[\eta(I)] = g$.
      \end{enumerate}
      Then there exists a transfer homomorphism $\theta\colon H^\bullet \to \cB(\cgrp_M)$.
  \end{enumerate}
\end{thm}

\begin{proof}
  By \autoref{prop:ideals-are-lgrpd}, $(\cF_v(\alpha), \cdot_v, \subset)$ is an arithmetical groupoid, and $\cI_v(\alpha)$ is its subcategory of integral elements. $(\cH_H,\cdot)$ is a left- and right-saturated subcategory of $(\cI_v, \cdot_v)$.
  
  \begin{enumerate}
    \item This follows immediately from \autoref{cor:bf-ff} and \autoref{prop:fact-bij}.

    \item 
      Let $I$ be a fractional left $H'$-ideal with $H'=dHd^{-1}$. 
      Then $I$ is principal if and only if the fractional left $H$-ideal $d^{-1}Id$ is, and this is the case if and only if $\eta(I) = \eta(d^{-1}Id) \in P_{H^\bullet} = \quo(\eta(\cH_H))$ (where the last equality is due to the previous lemma). Therefore the first condition of \autoref{thm:abstract-transfer} is satisfied. Condition (ii) of the present theorem is equivalent to the second condition of \autoref{thm:abstract-transfer}. Thus there exists a transfer homomorphism $\lbar \theta\colon \cH_H \to \cB(\cgrp_M)$ as in \autoref{thm:abstract-transfer}. By \autoref{prop:fact-bij}, there exists a transfer homomorphism $\theta\colon H^\bullet \to \cB(\cgrp_M)$.
      \qedhere
  \end{enumerate}
\end{proof}

\begin{remark} \label{rem:rtr}
  \mbox{}
  \begin{enumerate}
    \item \label{rtr:normalizing}
    We continue our discussion from \autoref{rem:ideal-struct}. Let $H$ be a normalizing Krull monoid.
    Then $\alpha = \{\, H \,\}$, $Ha = HaH = aH$ for all $a \in Q^\bullet$ and associativity is a congruence relation \cite[Lemma 4.4.1]{geroldinger13}, thus $H_{\text{red}} = \{\, H^\times a \mid a \in H \,\}$ with the induced operation is also a monoid.
    Therefore $\cH = \cH_H = \{\, HaH \mid a \in H \,\} \cong H_{\text{red}}$ and $G = \cF_v(\alpha)$ is the free abelian group on the maximal divisorial (two-sided) $H$-ideals, while $\cI_v(\alpha)$ is the free abelian monoid on the same basis.

    In the previous theorem we therefore have $\bG = G$, $\eta = \id$, $P_{H^\bullet} = \{\, Hq \mid q \in Q^\bullet \,\}$, and hence $\cgrp$ is the divisorial class group of $H$, and $\cgrp_M$ is the set of divisorial ideal classes that contain a maximal divisorial $H$-ideal. The second condition of the theorem is trivially true by virtue of $\card{G_0} = 1$ and the definition of $\cgrp_M$, and the first condition is trivially true because $\eta = \id$. We thus get a transfer homomorphism $H \to \cB(\cgrp_M)$ (induced from the transfer homomorphism $H_{\text{red}} \cong \cH_H \to \cB(\cgrp_M)$), which is the same one as in \cite[Theorem 6.5]{geroldinger13}.

  \item \label{rtr:non-bounded}
    If $H$ is a maximal order satisfying only \ref{A:acc} and \ref{A:mod}, then $\sL_{H^\bullet}(a)$ is finite and non-empty for all $a \in H^\bullet$.
    In \autoref{sec:fact} one may drop \ref{P:bdd} and \ref{P:acc}, and still obtain \hyperref[ffg:hf]{Proposition \ref*{prop:fact-freeish-grpd}.\ref*{ffg:hf}} in the weaker form that, for each $a \in G_+$, either $\sZ^*_{G_+}(a) = \emptyset$ or $\card{\sL_{G_+}(a)} = 1$ (and of course without any statement about $\Phi$, which can only be defined in the presence of \ref{P:bdd}). This is possible because \ref{P:acc} is only used to show existence of a rigid factorization of $a$. A sufficient condition for $\sZ^*_{G_+}(a) \ne \emptyset$ is that $G_+(s(a),\cdot)$ and $G_+(\cdot,t(a))$ satisfy the ACC. If $H$ satisfies \ref{A:acc}, then $G_+(e,\cdot)$ and $G_+(\cdot,e)$ with $e \in (\cH_H)_0$ (corresponding to conjugate orders of $H$) satisfy the ACC, and as in \autoref{cor:bf-ff} one shows that $\sL_{\cH_H}(a)$ is finite and non-empty for all $a \in \cH_H$. Hence the same is true for $H^\bullet$.
  \end{enumerate}
\end{remark}

\subsection{Rings} \label{sec:ideals-rings}

Suppose that $Q = (Q,+,\cdot)$ is a quotient ring in the sense of \cite[Chapter 3]{mcconnell-robson01} (but recall that we in addition require it to be unital, as we do for all rings).
Then $(Q,\cdot)$ is a quotient semigroup.
In the remainder of this section we show that the ring-theoretic divisorial one-sided ideal theory for maximal orders in $(Q,+,\cdot)$ coincides with the semigroup-theoretic one. \footnote{In \cite[Chapter 5]{mcconnell-robson01} the terminology ``reflexive'' is used in place of ``divisorial''.}
If $R$ is a ring-theoretic order in $Q$, then a fractional left $R$-ideal $I$ in the semigroup-theoretic sense is a fractional left $R$-ideal in the ring-theoretic sense if and only if $I - I \subset I$ (see \cite[\S3.1.11]{mcconnell-robson01} for the usual definition).

\begin{center}
  \emph{Let for the remainder of this subsection $Q=(Q,+,\cdot)$ be a quotient ring.}
\end{center}

\begin{lemma} \label{lemma:ring-ideals}
  Let $H$ be an order in the multiplicative semigroup $(Q,\cdot)$ and $I$ a fractional left $H$-ideal (in the semigroup-theoretic sense). 
  Consider the following statements:
  \begin{enumerate} 
    \enumequiv

    \item $I - I \subset I$.
    \item $\cO_l(I)$ is a subring of $Q$.
    \item $\cO_r(I)$ is a subring of $Q$.
  \end{enumerate}

  Then (a) $\Rightarrow$ (b) and (a) $\Rightarrow$ (c).
  If $H$ is a maximal order and $I$ is divisorial, then (a) $\Leftrightarrow$ (b) $\Leftrightarrow$ (c).
\end{lemma}

\begin{proof}
  Assume that (a) holds. We show (b): Let $a,b \in \cO_l(I)$. Then $aI \subset I$ and $bI \subset I$ and hence $(a-b)I \subset aI - bI \subset I - I \subset I$, thus $a-b \in \cO_l(I)$.

  Assume now that $H$ is maximal, $I=I_v$ and (b) holds. We show (a). Let $a, b \in I = I_v = (I^{-1})^{-1}$. Then $a I^{-1} \subset \cO_l(I)$, and $bI^{-1} \subset \cO_l(I)$, whence $(a - b)I^{-1} \subset aI^{-1} - bI^{-1} \subset \cO_l(I) - \cO_l(I) = \cO_l(I)$ and thus $a-b \in (I^{-1})^{-1}=I$.
\end{proof}

\begin{lemma} \label{lemma:ring-max}
  A ring-theoretic order $R$ in $Q$ is maximal in the ring-theoretic sense if and only if it is maximal in the semigroup-theoretic sense.
\end{lemma}

\begin{proof}
  We show that if $R$ is maximal in the ring-theoretic sense, then it is maximal in the semigroup-theoretic sense, as the other direction is trivial.
  Let $I$ be a fractional left $R$-ideal in the semigroup-theoretic sense. Then ${}_R \langle I \rangle$ is a fractional left $R$-ideal in the ring-theoretic sense, and using $R \subset \cO_l(I)$, it follows that $\cO_l(I) \subset \cO_l({}_R \langle I \rangle)$. Maximality of $R$ in the ring-theoretic sense implies $R = \cO_l({}_R \langle I \rangle)$, hence also $R = \cO_l(I)$. Similarly, if $J$ is a fractional right $R$-ideal in the ring-theoretic sense then $\cO_r(J) = R$.
  Therefore \autoref{lemma:max-ord} implies that $R$ is maximal in the semigroup-theoretic sense.
\end{proof}

As before let $\alpha$ be an equivalence class of maximal orders of $(Q,\cdot)$ in the semigroup-theoretic sense.

\begin{lemma} \label{lemma:rip}
  Let $H \in \alpha$ and assume that $H$ is a subring of $Q$ (i.e., an order in $Q$ in the ring-theoretic sense).
  \begin{enumerate}
    \item \label{rip:max} Every $H' \in \alpha$ is a subring of $Q$ (and therefore an order in $Q$ in the ring-theoretic sense).
    \item \label{rip:prop} If $I$ is a divisorial fractional left $H$-ideal and $J$ is a divisorial fractional left $\cO_r(I)$-ideal, then
      \[
      \{\, ab \mid a \in I, b \in J \,\}_v = \big( {}_H \langle\{\, ab \mid a \in I, b \in J \,\}\rangle \big)_v,
      \]
      i.e., the semigroup-theoretic $v$-product coincides with the ring-theoretic one.

    \item \label{rip:join} If $I$ and $J$ are divisorial fractional left $H$-ideals, then $(I \cup J)_v = (I + J)_v$.
  \end{enumerate}
\end{lemma}

\begin{proof}
  \mbox{}
  \begin{enumerate}
    \item By \hyperref[ord:conn]{Lemma \ref*{lemma:ord}.\ref*{ord:conn}} there exists a fractional $(H,H')$-ideal $I$. By maximality of $H$ and $H'$, also $\cO_l(I_v) = H$ and $\cO_r(I_v)=H'$ and the claim follows from \autoref{lemma:ring-ideals} applied to $I_v$.

    \item Write $I \cdot_S J = \{\, ab \mid a \in I, b \in J \,\}$ for the semigroup-theoretic ideal product and $I \cdot_R J = {}_H \langle\{\, ab \mid a \in I, b \in J \,\}\rangle$ for the ring-theoretic one. Then $I \cdot_S J \subset I \cdot_R J$, and both of these sets are fractional left $H$-ideals (in the semigroup-theoretic sense). Therefore $(I \cdot_S J)_v \subset (I \cdot_R J)_v$. For the converse inclusion, it suffices to show $I \cdot_R J \subset (I\cdot_S J)_v$, but this is true because by \autoref{lemma:ring-ideals} $(I\cdot_S J)_v$ is additively closed.

    \item Clearly $I \cup J \subset I + J$ and both sets are fractional left $H$-ideals (for $I+J$ proceed as in the proof of \hyperref[frac:isec-union]{Lemma \ref*{lemma:frac}.\ref*{frac:isec-union}}; in particular observe $\rc{H}{I \cup J} \subset \rc{H}{I+J}$). As before it therefore suffices to show $I + J \subset (I \cup J)_v$. This again holds due to \autoref{lemma:ring-ideals}.
      \qedhere
  \end{enumerate}
\end{proof}

Altogether, if $R$ is a maximal order in $Q$ in the ring-theoretic sense, then it does not matter whether we form $\cF_v(\alpha)$ by using the ring-theoretic or the semigroup-theoretic notions.
We use the same notion of boundedness for ring-theoretic orders as in \autoref{defilemma:bounded}; for semiprime Goldie rings this coincides with the notion in \cite{mcconnell-robson01}.

\begin{thm} \label{thm:transfer-ring}
  Let $R$ be a maximal order in a quotient ring $Q$, $\alpha$ its equivalence class of maximal orders in the semigroup-theoretic sense, and $\beta$ its equivalence class of maximal orders in the ring-theoretic sense. Then $\alpha = \beta$ and $\cF_v(\alpha) = \cF_v(\beta)$, where the latter is the ring-theoretic analogue of $\cF_v(\alpha)$.

  If $R$ is bounded, satisfies the ACC on divisorial left $R$-ideals and on divisorial right $R$-ideals, and the lattice of divisorial fractional left (right) $R$-modules is modular, then $(R,\cdot)$ is an arithmetical maximal order in $(Q,\cdot)$ in the semigroup-theoretic sense.
  In particular, the conclusions of \autoref{thm:transfer-semigroup} hold for $R$.
\end{thm}

\begin{proof}
  By \hyperref[rip:max]{Lemma \ref*{lemma:rip}.\ref*{rip:max}}, $\alpha = \beta$, and by \autoref{lemma:ring-ideals}, $\cF_v(\alpha) = \cF_v(\beta)$ as sets. By \autoref{lemma:rip}, the $v$-product, meet and join coincide, and hence $\cF_v(\alpha) = \cF_v(\beta)$ as lattice-ordered groupoids. The remaining claims follow from this.
\end{proof}

In \cite[\S5(d)]{rehm77-2}, Rehm gives examples for bounded maximal orders $E$, that are prime and satisfy the ACC on divisorial two-sided $E$-ideals, but do not satisfy the ACC on divisorial left $E$-ideals or the ACC on divisorial right $E$-ideals.
In fact (unless one takes the special case where $E$ itself is a quotient ring), the orders $E$ are not even atomic.
However, these orders are not Goldie, as they are not of finite left or right uniform dimension, and do not satisfy the ACC on left or right annihilator ideals.

Before going to maximal orders in central simple algebras, we discuss principal ideal rings.

\begin{exm}[Principal ideal rings]
  Let $R$ be a bounded order in a quotient ring $Q$. 
  Assume that every left $R$-ideal and every right $R$-ideal is principal.
  By the characterization in \autoref{lemma:max-ord}, $R$ is then already a maximal order, and it satisfies \ref{A:first}--\ref{A:last}.
  Thus $\cH_{R} = \cI_v(\alpha)$, and facts about the rigid factorizations in $\cI_v(\alpha)$ trivially descend to facts about rigid factorizations of $R^\bullet$.
  Examples we have in mind include bounded skew polynomial rings $D[X,\sigma]$, where $D$ is a division ring and $\sigma\colon D \to D$ is an automorphism, and the Hurwitz quaternions $\bZ[1,i,j,\f{1+i+j+ij}{2}]$ with $i^2=-1$, $j^2=-1$ and $ij=-ji$.
  Both of these examples are left- and right-euclidean domains, and hence principal ideal rings.
  In this way we can for example rediscover Theorem 2 in \cite[\S5]{conway-smith03}.

  Let $Q$ be a quaternion algebra over a field $K$ with $\chr(K) \ne 2$, and $a \in Q^\bullet \setminus K^\times$. Then $\nr(a) = a\lbar{a} \in K^\times$ and $\tr(a) = a + \lbar{a} \in K$. For the polynomial ring $Q[X]$ in the central variable $X$, therefore
  \[
    f = X^2 - \tr(a) X + \nr(a) = (X - cac^{-1})(X - c\lbar{a}c^{-1}) \quad\text{for all $c \in Q^\bullet$},
  \]
  and thus $\card{\sZ^*_{Q[X]}(f)} = \infty$ if $K$ is infinite.
  (But these rigid factorizations are usually considered to be identical factorizations, and $Q[X]$, being left- and right-euclidean, is even a UFD with suitable definitions, see for example \cite[Chapter 3.2]{berrick-keating00} and \cite[Chapter 3]{cohn85}.)
  In terms of ideal theory, every element $X - cac^{-1}$ with $c \in Q^\bullet$ generates a maximal left $Q[X]$-ideal lying above the maximal two-sided $Q[X]$-ideal $Q[X]f$. If also $d \in Q^\bullet$, then $Q[X](X - cac^{-1}) = Q[X](X - dad^{-1})$ if and only if $cac^{-1} = dad^{-1}$, i.e., $d^{-1}c \in K(a)$.
\end{exm}

\subsection{Classical maximal orders over Dedekind domains in CSAs} \label{sec:ring-csa}

Let $\cO$ be a commutative domain with quotient field $K$. 
By a central simple algebra $A$ over $K$, we mean a $K$-algebra with $\dim_K(A) < \infty$, which is simple as a ring, and has center $K$. Then $A$ is artinian because it is a finite-dimensional $K$-algebra, and hence it is a quotient ring (in an artinian ring, every non-zero-divisor is invertible \cite[\S3.1.1]{mcconnell-robson01}, hence it is a quotient ring and an element is left-cancellative if and only if it is right-cancellative if and only if it is cancellative).
By Posner's Theorem (\cite[\S13.6.6]{mcconnell-robson01}), a ring $R$ is a prime PI ring if and only if it is an order in a central simple algebra, and hence in particular, prime PI rings are bounded Goldie rings. Furthermore,
PI Krull rings are characterized as those maximal orders in central simple algebras whose center is a commutative Krull domain (\cite[Theorem 2.4]{jespers86}). We start with a simple corollary of \autoref{thm:transfer-ring}.

\begin{cor} \label{cor:pi-krull-ring}
  If $R$ is a PI Krull ring, then $\sL_{R^\bullet}(a)$ is finite and non-empty for all $a \in R^\bullet$.
\end{cor}

\begin{proof}
  We only have to verify the conditions of \autoref{thm:transfer-ring}.
  By \cite[Theorem 2.4]{jespers86} the various notions of Krull rings coincide for prime PI rings.
  Thus $R$ is a bounded Chamarie-Krull ring.
  The ACC on divisorial left $R$-ideals and divisorial right $R$-ideals follows from \cite{chamarie81} (or \cite[Corollary 3.11]{halimi12}).
  Moreover, for every divisorial prime $R$-ideal $P$, the set of regular elements modulo $P$, denoted $\cC(P)$, is cancellative, satisfies the left and right Ore condition, and for the localization $R_{\cC(P)}={}_{\cC(P)} R \subset Q$ every left (right) $R_{\cC(P)}$-ideal is principal (\cite[Proposition 2.5]{chamarie81}).
  The lattice of divisorial fractional left (right) $R_{\cC(P)}$-ideals is hence modular.
  Using the ACC on divisorial left and right $R$-ideals, one checks as in the commutative case that $I_v R_{\cC(P)} = (I R_{\cC(P)})_v$ for a fractional right $R$-ideal $I$.

  Suppose now $I,J,K$ are divisorial fractional right $R$-ideals, and $K \subset I$. We have to check $I \cap (J+K)_v = ((I \cap J) + K)_v$.
  But $(I \cap (J + K)_v) R_{\cC(P)} = IR_{\cC(P)} \cap (JR_{\cC(P)} + KR_{\cC(P)})_v$ and $((I \cap J) + K)_v R_{\cC(P)} = ((I R_{\cC(P)} \cap J R_{\cC(P)}) + K R_{\cC(P)})_v$, and thus, by modularity in the localizations, they are equal for every divisorial prime $R$-ideal $P$.
  The claim now follows from \cite[Lemme 2.7]{chamarie81}, by which the global divisorial fractional right $R$-ideals can be recovered as intersections from the local ones.
\end{proof}

Using \hyperref[rtr:non-bounded]{Remark \ref*{rem:rtr}.\ref*{rtr:non-bounded}}, we get the above result even for more general classes of rings, namely for Dedekind prime rings and bounded Chamarie-Krull rings (cf. \cite{chamarie81}).

\smallskip
But the aim of this subsection is to restrict to the situation where the base ring $\cO$ is a Dedekind domain, as a preparation for the structural results on sets of lengths in the setting of holomorphy rings.
Suppose that $\cO$ is a Dedekind domain.
A ring $R$ is a \emph{classical $\cO$-order} of $A$ if $\cO \subset R$, $R$ is finitely generated as $\cO$-module and $K R = A$. $R$ is a \emph{classical maximal $\cO$-order} if it is maximal with respect to set inclusion within the set of all classical $\cO$-orders.
Such classical maximal $\cO$-orders as well as their ideal theory are well-studied, in particular Reiner's book \cite{reiner75} provides a thorough description of them.
If $R$ is a classical $\cO$-order, then it is a ring-theoretic order in $A$ in the sense we discussed, and it is a maximal order if and only if it is a classical maximal $\cO$-order (for this see \cite[\S5.3]{mcconnell-robson01}).
The set of all classical maximal $\cO$-orders forms an equivalence class of (ring-theoretic) maximal orders, call it $\beta$ for a moment.
If we write $\alpha$ for the same semigroup-theoretic equivalence class of maximal orders (i.e., $\alpha=\beta$ as sets, but we view the elements of $\beta$ as rings and those of $\alpha$ just as semigroups), then $\cF_v(\alpha) = \cF_v(\beta)$ by \autoref{thm:transfer-ring}.
Next, we recall that our notion of ideals coincides with that of \cite{reiner75} and \cite{vigneras80} in the case of maximal orders, thereby seeing how the one-sided ideal theory of classical maximal $\cO$-orders is a special case of the semigroup-theoretic divisorial one-sided ideal theory developed in this section. We also recognize the abstract norm homomorphism $\eta$ of \autoref{sec:fact} as a generalization of the reduced norm of ideals (in the sense of \cite[\S24]{reiner75}).

\begin{lemma}
  Let $I \subset A$ and let $T$ be a classical $\cO$-order in $A$. The following are equivalent:
  \begin{enumerate}
    \enumequiv 

    \item $I$ is a fractional left $T$-ideal in the ring-theoretic sense (i.e., as in \cite[\S3.1.11]{mcconnell-robson01}).
    \item $I$ is a finitely generated $\cO$-module with $K I = A$ and $TI \subset I$ . \footnote{Here $K I = \{\, \lambda a \mid \lambda \in K, a \in I \,\} = \{\, \sum_{i=1}^n \lambda_i a_i \mid \text{$\lambda_i \in K, a_i \in I$} \,\} = K \otimes_{\cO} I$.} 
  \end{enumerate}
  If $T$ is maximal, then in addition the following statements are equivalent to the previous ones:
  \begin{enumerate}
    \enumequiv
    \setcounter{enumi}{2}
    \item $I$ is a divisorial fractional left $T$-ideal in the semigroup-theoretic sense (Definitions \ref{defi:ideals} and \ref{defi:divisorial}).
    \item $I$ is a divisorial fractional left $T$-ideal in the ring-theoretic sense (i.e., a reflexive fractional left $T$-ideal as in \cite[\S5.1]{mcconnell-robson01}).
  \end{enumerate}
\end{lemma}

\begin{proof}

  (a) $\Rightarrow$ (b): Recall that $I$ is a fractional left $T$-ideal in the ring-theoretic sense if $TI \subset I$, $I + I \subset I$ and there exist $x,y \in A^\times$ with $x \in I$ and $Iy \subset T$.
  $\cO$ is the center of $T$, and $T$ is finitely generated over the noetherian ring $\cO$. 
  Since $Iy \subset T$, therefore also $I$ is a finitely generated $\cO$-module.
  Writing $x^{-1} = rc^{-1}$ with $r \in T^\bullet$ and $c \in \cO^\bullet$ we see that $c = rx \in I \cap \cO^\bullet$.
  If $a \in A$ is arbitrary, then $a = r'd^{-1}$ with $r' \in T$, $d \in \cO^\bullet$ and therefore $a = (r'c)(c^{-1}d^{-1}) \in KI$.

  (b) $\Rightarrow$ (a): Certainly $TI \subset I$ and $I + I \subset I$. We have to find $x,y \in A^\times$ with $x \in I$ and $Iy \subset T$.
  Since $KI = A$, there exist $\lambda \in K^\times$ and $x \in I \cap A^\times$ with $1=\lambda x$ (in fact even $x \in K^\times$).
  If $I = {}_{\cO} \langle y_1,\ldots,y_l \rangle$ with $y_1,\ldots,y_l \in I$, then due to $KT = A$ there exists a common denominator $y \in \cO^\bullet$ with $y_i y \in T$, hence $Iy \subset T$.

  \smallskip
  Let now $T$ be maximal.
  (d) $\Rightarrow$ (a) is trivial, and (a) $\Rightarrow$ (d) follows because $T$ is a Dedekind prime ring, and hence every fractional left $T$-ideal (in the ring-theoretic sense) is invertible (see \cite[\S5.2.14]{mcconnell-robson01} or \cite[\S22,\S23]{reiner75} for the more specific case where $R$ is a maximal order in a CSA), and therefore divisorial.

  (c) $\Leftrightarrow$ (d) follows from \autoref{lemma:ring-ideals}.
\end{proof}

A subset $I \subset A$ satisfying the second condition of the previous lemma and additionally $\cO_l(I) = T$ is considered to be a left $T$-ideal in \cite{reiner75} and \cite{vigneras80}. Thus, a left $T$-ideal in the sense of \cite{reiner75,vigneras80} is (in our terms) a fractional left $T$-ideal in the ring-theoretic sense with $\cO_l(I) = T$. If $T$ is maximal, then the extra condition $\cO_l(I) = T$ is trivially satisfied, and the definitions are equivalent, but for a non-maximal order the definitions do not entirely agree (we will only need to work with ideals of maximal orders). 

Since all $I \in \cF_v(\alpha)$ are invertible (i.e., $II^{-1} = \cO_l(I)$ and $I^{-1}I=\cO_r(I)$ for the ring-theoretic products), the $v$-product coincides with the usual proper product of ideals: $I \cdot_v J = I \cdot J$ whenever $I, J \in \cF_v(\alpha)$ with $\cO_r(I) = \cO_l(J)$.
Therefore, $\cF_v(\alpha)$ is the groupoid of all normal ideals of $A$ in Reiner's terminology ($\cO$ is fixed implicitly).

To be able to apply our abstract results we still have to check that \ref{A:first} through \ref{A:last} are true for $\alpha$:
\ref{A:acc} follows because every $R \in \alpha$ is noetherian, while \ref{A:bdd} is true because every fractional left $R$-ideal with $R \in \alpha$ in fact even contains a non-zero element of the center (cf. \cite[Prop. 5.3.8(i) and (ii)]{mcconnell-robson01} or see ``(b) $\Rightarrow$ (a)'' of the last proof). Since every fractional left (right) $R$-ideal is divisorial, $\ref{A:mod}$ follows from the modularity of the lattice of left (right) $R$-modules.

Writing $\cF^\times(\cO)$ for the non-zero fractional ideals of the commutative Dedekind domain $\cO$, and $\bG$ for the universal vertex group of $\cF_v(\alpha)$, we have the following.

\begin{lemma} \label{lemma:eta-nr}
  If $R, R' \in \alpha$ and $\cP \in \bG$, then $\cP_R \cap \cO = \cP_{R'} \cap \cO \in \max(\cO)$ and there is a canonical bijection
  \[
    \{\, \cP \mid \cP \in \bG \text{ maximal integral \}} \to \max(\cO), 
  \]
  inducing an isomorphism of free abelian groups $r \colon \bG \isomto \cF^\times(\cO)$.
  The inverse map is given by $\fp \mapsto (\fP)$ where $\fP$ is the unique maximal (two-sided) $R$-ideal lying over $\fp$.
  If $R$ is unramified at $\fp$, then $\fP = R \fp$.

  If all residue fields of $\cO$ are finite, and $\eta\colon \cF_v(\alpha) \to \bG$ is the abstract norm homomorphism, then $r \circ \eta = \nr_{A/K}$.
\end{lemma}

\begin{proof}
  All but the last statement follow from \cite[Theorem 22.4]{reiner75}.
  Since $r \circ \eta$ and $\nr_{A/K}$ are both homomorphisms $\cF_v(\alpha) \to \cF^\times(\cO)$, it suffices to verify equality for $M$ a maximal integral left $R'$-ideal with $R' \in \alpha$, where it holds due to \cite[Theorem 24.13]{reiner75}.
\end{proof}

\section{Proof of \autoref{thm:main-transfer}} \label{sec:maxord-transfer}

\begin{center}
  \emph{Throughout this section, let $K$ be a global field and $\cO$ be a holomorphy ring in $K$. \footnote{For us, $\cO$ is a holomorphy ring if it is an intersection of all but finitely many of the valuation domains associated to valuations of $K$.} Furthermore, let $A$ be a central simple algebra over $K$, and $R$ a classical maximal $\cO$-order.
  }
\end{center}

Setting $P_A = \{\, aO \mid a \in K^\times, \text{ $a_v > 0$ for all archimedean places $v$ of $K$ where $A$ is ramified} \,\}$, and denoting by $\Cl_A(\cO) = \cF^\times(\cO) / P_A$ the corresponding ray class group, we have the following.

\begin{lemma} \label{lemma:clgrp-iso}
  Let $r$ be as in \autoref{lemma:eta-nr}. Then $r$ induces an isomorphism
  \[
    \bG / P_{R^\bullet} \,\cong\, \Cl_A(\cO),
  \]
  where $P_{R^\bullet} = \{\, \eta(Rx) \mid x \in A^\times \,\} \subset \bG$.
\end{lemma}

\begin{proof}
  By \autoref{lemma:eta-nr}, $r \circ \eta = \nr_{A/K}$.
  The isomorphism follows because $\nr(Rx) = \cO \nr(x)$ for all $x \in A^\times$, and $\nr(A^\times) = \{\, a \in K^\times \mid a_v > 0 \text{ for all archimedean places $v$ of $K$ where $A$ is ramified } \,\}$ by the Hasse-Schilling-Mass theorem on norms (\cite[Theorem 33.15]{reiner75}).
\end{proof}

\begin{lemma} \label{lemma:primes-in-every-class}
  For all classical maximal $\cO$-orders $R'$, and all $g \in \Cl_A(\cO)$, there exist infinitely many maximal left $R'$-ideals $I$ with $[\nr(I)] = g$.
\end{lemma}

\begin{proof}
  Let $g \in \Cl_A(\cO)$. Then there exist infinitely many distinct maximal ideals $\fp$ of $\cO$ with $[\fp] = g$: 
  The number field case for $\cO=\cO_K$, the ring of algebraic integers, can be found in \cite[Corollary 2.11.16]{ghk06} or \cite[Corollary 7 to Proposition 7.9]{narkiewicz04}. The general case then follows because $\cO$ is obtained from $\cO_K$ by localizing at finitely many maximal ideals, hence the induced epimorphism $\Cl_A(\cO_K) \to \Cl_A(\cO)$ yields the statement.
  For the function field case see \cite[Proposition 8.9.7]{ghk06}.

  For each $\fp \in \max(\cO)$ with $[\fp]=g$ and every maximal left $R'$-ideal $M$ with $\fp \subset M$, we have $[\nr(M)] = [\fp] = g$ (\cite[Theorem 24.13]{reiner75}, or use \autoref{lemma:eta-nr}).
\end{proof}

In the following equivalent characterizations of the first condition of \autoref{thm:main-transfer}, ``left'' may be replaced by ``right'' in each statement; this follows easily from the first statement.
We write $\LCl(R)$ for the finite set of isomorphism classes of fractional left $R$-ideals, i.e., $[I] = [J]$ in $\LCl(R)$ if and only if $J = Ix$ with $x \in A^\times$. 
The reduced norm induces a surjective map of finite sets $\mu_R\colon \LCl(R) \to \Cl_A(\cO)$, given by $[I] \mapsto [\nr(I)]$.

\begin{lemma} \label{lemma:sff-equiv} The following are equivalent.
  \begin{enumerate}
    \enumequiv

    \item A fractional left $R'$-ideal with $R'$ conjugate to $R$ is principal if and only if $\nr(I) \in P_A$.
    \item A fractional left $R$-ideal is principal if and only if $\nr(I) \in P_A$.
    \item Every fractional left $R$-ideal $I$ with $[\nr(I)] = \vec 0$ is principal.
    \item For the map of finite sets $\mu_R\colon \LCl(R) \to \Cl_A(\cO)$ it holds that $\card{\mu_R^{-1}(\vec 0)} = 1$.
    \item \label{sffe:e} Every stably free left $R$-ideal is free.
    \item Every finitely generated projective $R$-module that is stably free is free.
  \end{enumerate}
\end{lemma}

\begin{proof}
  The equivalence of (a), (b), (c) and (d) is trivial. 
  The remaining equivalences follow from standard literature: 
  (f) $\Rightarrow$ (e) is true because $R$ is hereditary noetherian.
  
  (e) $\Rightarrow$ (f): 
  Let $M \ne \vec 0$ be a stably free finitely generated projective $R$-module. 
  Then $M \cong R^n \oplus I$ for some left $R$-ideal $I$ and $n \in \bN_0$ (\cite[Theorem 27.8]{reiner75} or \cite[\S5.7.8]{mcconnell-robson01}). $I$ is stably free and hence free by \ref*{sffe:e}, but then so is $M$.

  To see (d) $\Leftrightarrow$ (e) it suffices to recall that $\Cl_A(\cO)$ is isomorphic to the projective class group $\Cl(R)$ (see e.g. \cite[Corollary 9.5]{swan80}) and that $\LCl(R)$ is just the set of isomorphism classes of locally free $R$-modules of rank one, i.e., the map $\mu_R$ corresponds to $\mathsf{LF}_1 \to \Cl(R), [I] \mapsto [I] - [R]$ in the notation of \cite{swan80}. (See also \cite[Theorem 35.14]{reiner75} or \cite{froehlich75} for the number field case.)
\end{proof}

\begin{proof}[Proof of \autoref{thm:main-transfer}]
  By \autoref{thm:transfer-ring}, $R$ is an arithmetical maximal order in $Q$. We verify conditions (i) and (ii) of \autoref{thm:transfer-semigroup}.
  Let $I$ be a fractional left $R$-ideal.
  By \autoref{lemma:clgrp-iso}, $\eta(I) \in P_{R^\bullet}$ if and only if $\nr(I) \in P_A$.
  By \autoref{lemma:sff-equiv}, and the fact that every stably free left $R$-ideal is free, this is the case if and only if $I$ is principal, thus condition (i) holds. Condition (ii) holds due to \autoref{lemma:primes-in-every-class}. By \autoref{lemma:clgrp-iso}, $C \cong \Cl_A(\cO)$, and by \autoref{lemma:primes-in-every-class}, therefore $C = C_M$. Hence there exists a transfer homomorphism $\theta\colon R^\bullet \to \cB(\Cl_A(\cO))$.
  The remaining claims in the theorem follow from this by \autoref{prop:zss}.
\end{proof}

\begin{remark}
  If more generally $\cO'$ is an arbitrary Dedekind domain with quotient field the global field $K$, then there is a transfer homomorphism to either $\cB(\Cl_A(\cO'))$ or $\cB(\Cl_A(\cO') \setminus \{ \vec 0 \})$, depending on whether or not $\cO'$ contains prime elements.
  Only \autoref{lemma:primes-in-every-class} has to be adapted: $\cO'$ is a localization of a holomorphy ring $\cO$, and hence there is an epimorphism $\Cl_A(\cO) \to \Cl_A(\cO')$. This implies that every class $g \in \Cl_A(\cO') \setminus \{ \vec 0 \}$ contains a maximal ideal  (see \cite{claborn65} for details). Therefore, for all classical maximal $\cO'$-orders $R'$ and all $g \in \Cl_A(\cO') \setminus \{ \vec 0 \}$, there exists a maximal left $R'$-ideal $I$ with $[\nr(I)]=g$. The trivial class however may or may not contain a maximal ideal. 
  In either case, the statements 1--3 of \autoref{thm:main-transfer} hold true. 
  Thanks to Kainrath for pointing this out.
\end{remark}

\section{Proof of \autoref{thm:main-distances}} \label{sec:maxord-distances}

\emph{Throughout this section, let $\cO_K$ be the ring of algebraic integers in a number field $K$, $A$ a central simple algebra over $K$, and $R$ a classical maximal $\cO_K$-order in $A$ having a stably free left $R$-ideal that is not free. Furthermore, the discriminant of $A$ is denoted by
\[
  \fD = \prod_{\substack{\fp \in \max(\cO_K) \\ \textnormal{$A$ is ramified at $\fp$}}} \fp \ideal \cO_K .
\]}

The aim of this section is to prove Theorem \ref{thm:main-distances}.
The existence of a stably free left $R$-ideal that is not free implies that $A$ is a totally definite quaternion algebra and that $K$ is totally real (Note, that conversely, for all but finitely many isomorphism classes of such classical maximal $\cO_K$-orders in totally definite quaternion algebras there exist stably free left $R$-ideals that are non-free). We proceed in three subsections.

\subsection{Reduction.}

We state two propositions and show how they imply Theorem \ref{thm:main-distances}. The proofs of these two propositions will then be given in \autoref{subsec:proof-totdef}.

\begin{prop} \label{prop:delta-tot-def}
  There exists a totally positive prime element $p \in \cO_K$, a non-empty subset $E \subset \{\, 2,3,4 \,\}$ and for every $l \in \bN_0$ an atom $y_l \in \cA(R^\bullet)$ such that
  \[
    \sL_{R^\bullet}(y_l p) = \{\, 3 \,\} \,\cup\, (l + E).
  \]
  (We emphasize that $E$ does not depend on $l$.)
\end{prop}

\begin{prop} \label{prop:shift}
  If $L \in \cL(R^\bullet)$ and $n \in \bN$, then $n + L = \{\, n + l \,\mid\, l \in L \,\} \in \cL(R^\bullet)$.
\end{prop}

\begin{proof}[Proof of \autoref{thm:main-distances} (based on \autoref{prop:delta-tot-def} and \autoref{prop:shift})]
  We first show that there is no transfer homomorphism $R^\bullet \to \cB(G_P)$ for any subset $G_P$ of an abelian group.
  Assume to the contrary that $\theta\colon R^\bullet \to \cB(G_P)$ is such a transfer homomorphism.

  \begin{enumerate}
   \setlength\itemindent{1.55em}
   \item[Claim \textbf{A}.] If $S \in \cB(G_P)$ and $U \in \cA(\cB(G_P))$, then
      $
        \max \sL_{\cB(G_P)}(SU) \le \card{S} + 1.
      $
  \end{enumerate}

  \begin{proof}[Proof of \textbf{A}]
  \renewcommand{\qedsymbol}{$\Box$(\textbf{A})}
  Let $S=g_1\cdot\ldots\cdot g_l$, with $l = \card{S}$ and $g_1,\ldots,g_l \in G_P$, and suppose that $SU = T_1\cdot\ldots\cdot T_k$ with $k \in \bN$ and $T_1,\ldots,T_k \in \cA(\cB(G_P))$. Then for every $i \in [1,k]$ either $T_i \mid U$, but then already $T_i = U$, or $g_j \mid T_i$ for some $j \in [1,l]$. This shows $k \le \card{S} + 1$.
  \end{proof}

  By \autoref{prop:delta-tot-def}, there exists a totally positive prime element $p \in \cO_K$, and for every $l \in \bN_0$ an atom $y_l \in \cA(R^\bullet)$ with $\max \sL_{R^\bullet}(y_l p) \ge l+2$. But, if $l \ge \card{\theta(p)}$,
  then
  \[
    l + 2 \,\le\, \max \sL_{R^\bullet}(y_l p) = \max \sL_{\cB(G_P)}(\theta(y_l) \theta(p)) \,\le\, \card{\theta(p)}+1 \,\le\, l+1,
  \]
  a contradiction.

  \smallskip
  In order to show $\Delta(R^\bullet) = \bN$, we choose $d \in \bN$.
  Let $p$ and $E$ be as in \autoref{prop:delta-tot-def} and set $\epsilon = \min E$. If $l = d + 3 - \epsilon$ and $y_l$ as in \autoref{prop:delta-tot-def}, then we find $d = (l + \epsilon) - 3 \in \Delta_{R^\bullet}(y_l p)$.

  Let $k \in \bN_{\ge 3}$. By definition, we have $\cU_k(R^\bullet) \subset \bN_{\ge 2}$. Thus it remains to show that for every $k' \ge 3$ there exists an element $a \in R^\bullet$ with $\{\, k, k' \,\} \subset \sL(a)$. Assume without restriction that $k \le k'$ and let $k = 3 + n$ with $n \in \bN_0$. Using \autoref{prop:delta-tot-def}, we find an element $a' \in R^\bullet$ with $\{\,3 = k - n,\, k' - n \,\} \in \sL(a')$, and hence by \autoref{prop:shift} there exists an element $a \in R^\bullet$ with $\{\, k,k' \,\} \in \sL(a)$.
\end{proof}

\subsection{Preliminaries.} \label{sec:dist-prelim}

\emph{Algebraic number theory:}
Our notation mainly follows Narkiewicz \cite{narkiewicz04}. Let $L/K$ be an extension of number fields. Then $D_{L/K}$ is the relative different, $\Norm_{L/K}$ the relative field norm, $d_{L/K} = \Norm_{L/K}(D_{L/K})$ is the relative discriminant and $d_K = d_{K/\bQ}$ the absolute discriminant (we tacitly identify ideals of $\bZ$ with their positive generators for the absolute discriminant and norm). If $\cO \subset \cO_K$ is an order, then $\ff_{\cO}$ is the conductor of $\cO$ in $\cO_K$ and $h(\cO) = \card{\Pic(\cO)}$ is the class number of $\cO$. Given $a \in L$ with minimal polynomial $f \in K[X]$ over $K$, $\delta_{L/K}(a) = f'(a)$ is the different of $a$. Completion at a prime $\fp \in \max(\cO_K)$ is denoted by a subscript $\fp$, e.g., $\cO_{K,\fp}$, $K_\fp$, and so on.
If $\fm \ideal \cO_K$ is a squarefree ideal, then
\[
  \Cl^+_{\fm}(\cO_K) = \{\, \fa \in \cF^\times(\cO_K) \mid (\fa, \fm) = \cO_K \,\} \,\big/\, \{\, a\cO_K \mid \text{$a \in K^\times$ is totally positive, $a \equiv 1 \mod \fm$} \,\}
\]
denotes the corresponding ray class group.
We will repeatedly make use of the fact that every class in $\Cl^+_{\fm}(\cO_K)$ contains infinitely many maximal ideals of $\cO_K$ (\cite[Corollary 7 to Proposition 7.9]{narkiewicz04}).

\smallskip
\emph{Quaternion algebras:} We follow \cite{vigneras80, maclachlan-reid03}, and \cite{kirschmer-voight10,kirschmer-voight10cor} for computational aspects.
Denote by $\lbar{\,\cdot\,}: A \isomto A^{\text{op}}$ the anti-involution given by conjugation of elements.
Then
\[
  \nr_{A/K}(x)=\nr(x) = x \lbar{x} = \lbar{x} x \qquad\text{and}\qquad \tr_{A/K}(x) = \tr(x) = x + \lbar{x} \qquad\text{for all $x \in A$.}
\]
Every element $x \in A$ satisfies an equation of the form
\[
  x^2 - \tr(x) x + \nr(x) = 0,
\]
 and if $x \in A \setminus K$, then $K(x) / K$ is a quadratic field extension. From the equation above we see that $\Norm_{K(x)/K} = \nr_{A/K}|K(x)$ and $\Trace_{K(x)/K} = \tr_{A/K}|K(x)$.

\smallskip
A classical $\cO_K$-order $T$ of $A$ is called a \emph{classical Eichler ($\cO_K$)-order} if it is the intersection of two classical maximal $\cO_K$-orders. \footnote{Though unconventional, we keep the qualifier ``classical'' for consistency with the earlier sections.} The reduced discriminant of a classical $\cO_K$-order $T$ takes the form $\fD \fN$ where $\fN \ideal \cO_K$ is the level of $T$. Furthermore, because $A$ is totally definite, we have $[T^\times: \cO_K^\times] < \infty$ for the unit group, and $\Cl_A(\cO_K) = \Cl^+(\cO_K)$ is the narrow class group of $\cO_K$.

\smallskip
As in the previous section, $\LCl(R)$ is the set of isomorphism classes of left $R$-ideals, and $\mu_R\colon \LCl(R) \to \Cl^+(\cO_K), [I] \mapsto [\nr(I)]$.
\begin{prop} \label{prop:exist-max}
  Let $C \in \bN$.
  Let $\fp \in \max(\cO_K)$ with $\fp \nmid d_K \fD$, and such that
  \[
    \f{h^+}{M w^2} (\Norm_{K/\bQ}(\fp) + 1 ) - \f{2}{w} \sqrt{\Norm_{K/\bQ}(\fp)} \ge C.
  \]
  Then, for every $c \in \LCl(R)$ with $\mu_R(c) = [\fp]$, there exist at least $C$ maximal left (right) $R$-ideals of reduced norm $\fp$ and class $c$. Here $h^+ = \card{\Cl^+(\cO_K)}$ is the narrow class number, and $w$ and $M$ are constants depending on $\fD$ (see \cite{kirschmer-voight10,kirschmer-voight10cor}).
\end{prop}

\begin{proof}
  Although not explicitly stated in this way, this is proved by Kirschmer and Voight in \cite{kirschmer-voight10, kirschmer-voight10cor}:
  In the proof of \cite[Proposition 7.7]{kirschmer-voight10}, a lower bound on the entries of a matrix $T'$ (in their notation) is derived, which immediately gives a lower bound on the entries of a matrix $T(\fp)$ (in their notation). This is exactly what we need, as is clear from their definition of $T(\fp)$.
\end{proof}

\smallskip
\emph{Optimal embeddings:} Let $L/K$ be a quadratic field extension, and $T$ a classical Eichler $\cO_K$-order in $A$ of squarefree level $\fN \ideal \cO_K$. If $\cO$ is an order in $L$, every embedding $\iota\colon \cO \to T$ gives rise to a unique embedding $\iota\colon L \to A$, and $\iota$ is an \emph{optimal embedding} if $\iota(L) \cap T = \cO$. For $a \in T^\times$, $\cO \to T, x \mapsto a\iota(x) a^{-1}$ is then another such embedding. The number of optimal embeddings up to conjugation by units is bounded above by a constant (depending only on $\fD$ and $\fN$) times $h(\cO)$ (see \cite[Corollaire III.5.12]{vigneras80}). Since $[T^\times\colon \cO_K^\times]$ is finite, the total number of optimal embeddings of $\cO$ into $T$ is still bounded by a constant times $h(\cO)$.

\emph{Quadratic forms:}
We use a theorem about representation numbers of totally positive definite quadratic forms over totally real fields.
Let $V$ be an $n$-dimensional $K$-vector space.
An $\cO_K$-lattice $L$ of rank $n$ is a finitely generated $\cO_K$-submodule of $V$ that generates $V$ (over $K$).
Together with a quadratic form $q\colon V \to K$ with $q(L) \subset \cO_K$, $(L,q)$ it is a quadratic lattice.
For $a \in \cO_K$ we set
\[
  r(L, a) = \card{ \{ x \in L \mid q(x) = a \} }.
\]
An element $a \in \cO_K$ is locally represented everywhere by $(L,q)$ if it is represented by the completion $L_v = L \otimes_{\cO_K} \cO_{K,v}$ for all places $v$ of $K$.
The following result is a special case of Theorem 5.1 in \cite{schulze-pillot04}.

\begin{prop} \label{prop:pos-def-quad}
  Let $(L,q)$ be a quadratic $\cO_K$-lattice of rank four and suppose that $q$ is totally positive definite.
  Then, for every $\eta > 0$ and $s \in \bN_0$, there exists a constant $C_{\eta,s} > 0$, such that for all $a \in \cO_K$ that are locally represented everywhere by $L$, with $\abs{\Norm_{K/\bQ}(a)}$ sufficiently large and $\fp^s \nmid a\cO_K$ if $L_\fp$ is anisotropic, the asymptotic formula
  \[
    r(L, a) \,=\, r(\gen L, a) + O(\abs{\Norm_{K/\bQ}(a)}^{\f{11}{18} + \varepsilon})
  \]
  holds with
  \[
    r(\gen L, a) \,\ge\, C_{\eta,s} \cdot \Norm_{K/\bQ}\big(a\, ({}_{\cO_K} \langle q(L) \rangle)^{-1}\big)^{1-\eta}.
  \]
\end{prop}

In particular, $r(L,a)$ is of order of magnitude $\abs{\Norm_{K/\bQ}(a)}^{1 - \eta}$.
If $T$ is any classical $\cO_K$-order and $I$ any left $T$-ideal (in particular if $I=T$), then the restriction of the reduced norm to $I$ makes $(I, \nr\mid I)$ into a quadratic $\cO_K$-lattice of rank four and this is the situation that we will apply this result to.

\medskip
\emph{Ideal theory in $R$:}
Let $\alpha$ be the set of all classical maximal $\cO_K$-orders in $A$ (i.e., the equivalence class of the maximal order $R$).
 Conjugation extends to ideals: For $I \in \cF_v(\alpha)$ define $\lbar{I} = \{\, \lbar{x} \mid x \in I \,\}$. Then $\lbar{I}$ is a fractional $(\cO_r(I), \cO_l(I))$-ideal, $I \cdot \lbar{I} = \cO_l(I) \nr(I)$, and $\lbar{I} \cdot I = \cO_r(I) \nr(I)$, and hence $I^{-1} = \lbar{I} \cdot (\cO_l(I)\nr(I))^{-1} = (\nr(I)\cO_r(I))^{-1} \cdot \lbar{I}$.

$\cF_v(\alpha)$ takes a particularly simple form: If $\fp \mid \fD$ then there exists a maximal two-sided $R$-ideal $\fP$ with $\fP^2 = \fp$, $\nr(\fP) = \fp$ and if $I$ is a left or right $R$-ideal with $\nr(I)=\fp^k$, then $I = \fP^k$.

If $\fp \nmid \fD$, then $\fP = \fp R$ is the maximal two-sided $R$-ideal lying above $\fp$, and $R_{\fp} / \fP_{\fp} \cong M_2(\cO_{K,\fp} / \fp_\fp ) \cong M_2(\bF_{\Norm_{K/\bQ}(\fp)})$. In particular there are $\Norm_{K/\bQ}(\fp) + 1$ maximal left $R$-ideals (respectively maximal right $R$-ideals) with reduced norm $\fp$. If $M,N$ are two distinct maximal left $R$-ideals with $\nr(M) = \nr(N) = \fp$, then $M \cap N = \fP$ (since the composition length of $M_2(\bF_{\Norm_{K/\bQ}(\fp)})$ is two). This implies that if $M \cdot M' = N \cdot N'$ with maximal integral $M',N' \in \cF_v(\alpha)$, then $M\cdot M' = N \cdot N' = \fP$, and thus necessarily $M' = \lbar{M}$, $N' = \lbar{N}$.

We therefore explicitly know all relations between maximal integral elements of $\cF_v(\alpha)$: From \autoref{prop:fact-freeish-grpd} we know that all relations are generated from those between pairs of products of two elements and it also characterizes the only relation between maximal integral elements of coprime reduced norm. With the discussion above we now also know the relations between two maximal integral elements of the same reduced norm: Either there are none, or the product is $\fP$.

A left $R$-ideal $I$ is \emph{primitive} if it is not contained in an ideal of the form $R \fa$ with $\fa \ideal \cO_K$. If $\nr(I) = \fp^k$ with $\fp \in \max(\cO_K)$ and $I$ is primitive, then it has a unique rigid factorization in $\cI_v(\alpha)$.

\subsection{Proofs of \autoref{prop:delta-tot-def} and \autoref{prop:shift}} \label{subsec:proof-totdef}
We start with some lemmas.

\begin{lemma} \label{lemma:eps-in-center}
  Let $T$ be a classical $\cO_K$-order in $A$.
  For all but finitely many associativity classes of totally positive prime elements $q \in \cO_K$ we have:
  If $x \in T$ with $\nr(x) = q$ and $x^2 = \varepsilon q$ for some $\varepsilon \in T^\times$, then $\varepsilon = -1$.
\end{lemma}

\begin{proof}
  $x$ satisfies the polynomial equation $x^2 - \tr(x) x + \nr(x) = 0$.
  Substituting $x^2 = \varepsilon q$ and $\nr(x) = q$ yields
  \begin{equation} \label{eq:tr-eps}
    \tr(x) x = (1 + \varepsilon) q.
  \end{equation}
  It will thus suffice to show that for all but finitely many $\cO_K q$, we have $\tr(x) = 0$.

  Assume that $q \in \cO_K$ is a totally positive prime element, $x \in T$ with $x^2 = \varepsilon q$ and $\tr(x) \ne 0$.
  Then $K(x) = K(\varepsilon)$ by \eqref{eq:tr-eps}. Let $L=K(\varepsilon)$. Since $\varepsilon \in L \cap T^\times \subset \cO_L^\times$,
  \[
    \cO_L q = \cO_L \varepsilon q = (\cO_L x)^2,
  \]
  and therefore $q$ ramifies in $\cO_L$, implying $\cO_K q \mid d_{L/K}$.
  Hence, for fixed $L$ there are only finitely many possibilities for $\cO_K q$, and moreover there are only finitely many possibilities for $L=K(\varepsilon)$ because $[T^\times:\cO_K^\times]$ is finite since $A$ is totally definite.
  \footnote{It should be pointed out that an even stronger statement is true. For any fixed totally real field $K$, there are only finitely many totally imaginary quadratic extensions that have larger unit group (i.e., weak unit defect), while all other totally imaginary quadratic extensions $L/K$ have $\cO_L^\times = \cO_K^\times$ (i.e., strong unit defect). This follows from $[\cO_L^\times:\mu(L)\cO_K^\times]\in \{1,2\}$ (\cite[Theorem 4.12]{washington97}): For $\varepsilon \in \cO_L^\times$ we have $\varepsilon^2 = \eta \zeta$ with $\zeta \in \mu(L)$ and $\eta \in \cO_K^\times$. But $\ord(\zeta) \mid [L:\bQ]=2[K:\bQ]$ and if $\gamma \in (\cO_K^\times)^2$, then $\eta \gamma \zeta$ yields the same extension. Since $[\cO_K^\times:(\cO_K^\times)^2] < \infty$ there are therefore only finitely many such extensions. This argument is due to Remak in \cite[\S3]{remak54}. }
  Thus there are, up to associativity, only finitely many such $q$.
\end{proof}

\begin{lemma} \label{lemma:exist-a}
  Let $T$ be a classical $\cO_K$-order in $A$.
  For every $M \in \bN$ there exists a $C \in \bN$ such that for all totally positive prime elements $q \in \cO_K$ with $q \in \nr_{A_\fp/K_\fp}(T_{\fp})$ for all $\fp \in \max(\cO_K)$ and $\Norm_{K/\bQ}(q) \ge C$
  \[
    \card{\{ a \in T \mid \nr(a) = q \text{ and } a^2 \ne -q \}} \ge M.
  \]
\end{lemma}

\begin{proof}
  Let $q$ be a totally positive prime element of $\cO_K$.
  We derive an upper bound with order of magnitude $\sqrt{\Norm_{K/\bQ}(q)} \log(\Norm_{K/\bQ}(q))^{2[K:\bQ]-1}$ on the number of elements $a \in T$ with $a^2 = -q$ (based on counting optimal embeddings). Comparing this to the lower bound of order of magnitude $\Norm_{K/\bQ}(q)^{1-\eta}$ for the number of elements $a \in T$ with $\nr(a) = q$ obtained from \autoref{prop:pos-def-quad} will give the result.

 If $a \in T$ with $\nr(a) = q$ and $a^2 = -q$, then $\cO_K[a] \subset T$ is isomorphic to the order $\cO_K[\sqrt{-q}]$ in the relative quadratic extension $K(\sqrt{-q})$.
  We determine an upper bound the number of embeddings of $\cO_K[\sqrt{-q}]$ into $T$.
  For this we may without loss of generality assume that $T$ is a classical Eichler order of squarefree level, for otherwise we may replace it by a classical Eichler order of squarefree level in which it is contained (e.g., a classical maximal order), and bound the number of embeddings there.

  Let $L = K(\sqrt{-q})$.
  Since $f = X^2 + q$ is the minimal polynomial of $\sqrt{-q}$ over $K$, we get for the different $\delta_{L/K}(\sqrt{-q}) = f'(\sqrt{-q}) = 2 \sqrt{-q}$. Therefore, we find for the conductor of $\cO_K[\sqrt{-q}]$ in $\cO_L$,
  \[
    \ff_{\cO_K[\sqrt{-q}]} = \delta_{L/K}(\sqrt{-q}) D_{L/K}^{-1} \;\mid\; 2 \cO_L.
  \]
  (cf. \cite[Proposition 4.12 and Theorem 4.8]{narkiewicz04}).
  Since $2\cO_L \subset \cO_K[\sqrt{-q}] \subset \cO_L$ and $\card{\cO_L / 2 \cO_L} = 2^{[L:\bQ]}$, there are at most $2^{2^{[L:\bQ]}}$ orders in $\cO_L$ that contain $\cO_K[\sqrt{-q}]$. For any such order $\cO$ with $\cO_K[\sqrt{-q}] \subset \cO \subset \cO_L$, we have
  \[
    h(\cO) = h(\cO_L) \f{\card{(\cO_L/\ff_{\cO})^\times}}{\card{(\cO/\ff_{\cO})^\times}} \,\le\, h(\cO_L) 2^{[L:\bQ]}
  \]
  (cf. \cite[\S I.12.9 and \S I.12.11]{neukirch92}).
  The number of optimal embeddings of $\cO$ into $T$ is bounded by a constant times $h(\cO)$, and hence the total number of embeddings of $\cO_K[\sqrt{-q}]$ into $T$ is bounded by a constant times $h(\cO_L)$, where the constant does not depend on $q$.
  Combining the upper bound
  \[
    h(\cO_L) \,\ll\, \sqrt{\abs{d_{L}}} \,\log(\abs{d_{L}})^{[L:\bQ]-1}
  \]
  (cf. \cite[Theorem 4.4]{narkiewicz04}), with
  \[
    d_{L} = \Norm_{L/\bQ}(D_{L/\bQ}) = \Norm_{L/\bQ}(D_{K/\bQ}) \Norm_{L/\bQ}(D_{L/K}) \le d_{K}^2 2^{[L:\bQ]} \abs{\Norm_{L/\bQ}(\sqrt{-q})} = d_{K}^2 2^{[L:\bQ]} \Norm_{K/\bQ}(q).
  \]
  (here $D_{L/K} \mid 2\sqrt{-q}\cO_L$ was used), we obtain
  \[
   h(\cO_L) \,\ll\, \sqrt{\Norm_{K/\bQ}(q)} \,\log{(\Norm_{K/\bQ}(q))}^{[L:\bQ]-1},
  \]
  and thus an upper bound of the same order for $\card{\{ a \in T \mid \nr(a) = q \text{ and } a^2 = -q \}}$.

  By \autoref{prop:pos-def-quad}, for every $\eta > 0$ and sufficiently large (in norm) $q$ with $q$ being locally represented everywhere by the norm form, $\card{\{ a \in T \mid \nr(a) = q \}}$ grows with order of magnitude $\Norm_{K/\bQ}(q)^{1 - \eta}$, and the claim follows, by choosing $\eta$ small enough, say $\eta < \f 1 4$.
\end{proof}

\begin{remark}
  For any classical $\cO_K$-order $T$ of $A$ there are infinitely many pairwise non-associated totally positive primes $q \in \cO_K$ that are locally represented everywhere by $\nr_{A/K}$ on $T$.
  This can easily be seen as follows: Let $\fD \fN \ideal \cO_K$ be the discriminant of $T$. If $\fp \in \max(\cO_K)$ with $\fp \nmid \fD \fN$, then $T_\fp \cong M_2(\cO_{K,\fp})$ and thus $\nr_{A_\fp/K_\fp}(T_\fp) = \cO_{K,\fp}$.
  If $\fp \mid \fD \fN$, since $\cntr(T_\fp) = \cO_{K,\fp}$, certainly still every square of $\cO_{K,\fp}$ is represented by $\nr_{A_\fp/K_\fp}$ on $T_\fp$ (in fact, if $\fp \mid \fD$ but $\fp \nmid \fN$ then $T_{\fp}$ is isomorphic to the unique classical maximal $\cO_{K,\fp}$-order in the unique quaternion division algebra over $K_\fp$, for which $\nr(T_\fp)=\cO_{K,\fp}$ also holds). By Hensel's Lemma therefore every totally positive prime element $q \in \cO_K$ with $q \equiv 1 \mod 4 \fD \fN$ is locally represented everywhere by $\nr_{A/K}$ on $T$. But there are infinitely many pairwise non-associated such primes, because every class of the ray class group $\Cl_{4 \fD \fN}^+(\cO_K)$ contains infinitely many pairwise distinct maximal ideals, and primes $q$ of the required form correspond exactly to the trivial class in $\Cl_{4 \fD \fN}^+(\cO_K)$.
\end{remark}

\begin{lemma} \label{lemma:an-atom}
  Let $q$ be a totally positive prime element of $\cO_K$.
  Let $I$ be a non-principal right $R$-ideal with $\nr(I) = q^m\cO_K$ for some $m \in \bN$,
  and $J$ be a left $S=\cO_l(I)$-ideal with $\nr(J)=q^n\cO_K$ for some $n \in \bN$ such that: $I \cong J$ (as left $S$-ideals) and $I$ (respectively $J$) is not contained in any principal left $S$-ideal except $S$ itself, and not contained in any principal right $\cO_r(I)$-ideal (respectively right $\cO_r(J)$-ideal) except $\cO_r(I)$ (respectively $\cO_r(J)$) itself.

  Assume further that $a \in S$ with $\nr(a) = q$ and $a^2 S \ne qS$.
  \begin{enumerate}
    \item For all $l \in \bN$, $(a^l \lbar{J} a^{-l}) a^l I$ is a principal right $R$-ideal and an atom of $\cH_{R^\bullet}$.
      In particular, $a^l q^{m} \in \cA(R^\bullet)$ for all $l \in \bN$.
    \item $\lbar{J} I \in \cA(\cH_{R^\bullet})$ if it is primitive. In particular if $m=n=1$ and $I\ne J$, then $\lbar{J} I \in \cA(\cH_{R^\bullet})$.
  \end{enumerate}
\end{lemma}

\begin{proof}
  Since $I$ is not contained in any principal right $R$-ideal, it is in particular not contained in $qR$, hence primitive. Similarly, $J$ is primitive. Let $\rf{M_1,\ldots,M_m} \in \sZ^*_{\cI_v(\alpha)}(I)$ and $\rf{N_1,\ldots,N_n} \in \sZ^*_{\cI_v(\alpha)}(J)$, with $M_1,\ldots,M_m, N_1,\ldots,N_n \in \cM_v(\alpha)$, be the unique rigid factorizations of $I$ and $J$.

  \begin{enumerate}
  \item
      Since $I \cong J$ as left $S$-ideals, $(a^l \lbar{J} a^{-l}) a^l I = a^l \lbar{J} I$ is principal.
      A rigid factorization of it is given by
      \[
        \rf{(a^{l} \lbar{N_n} a^{-l}), \ldots, (a^l \lbar{N_1} a^{-l}), (a^{l} S a^{-l}) a, (a^{l-1} S a^{-(l-1)}) a, \ldots, (aSa^{-1})a, M_1, \ldots, M_m}
      \]
      with $M_i,\, a^{l} \lbar{N_j} a^{-l},\, (a^{l-k} S a^{-(l-k)}) a \in \cM_v(\alpha)$ for $i \in [1,m]$, $j \in [1,n]$ and $k \in [0,l-1]$.
      By the restrictions imposed on $I$, $J$ and $a$, this is the only rigid factorization of $a^l \lbar{J} I$.
      Since any non-empty proper subproduct starting from the left (or the right) is non-principal, it is an atom in $\cH_{R^\bullet}$.
      The ``in particular'' statement follows by setting $J=I$, as then $a^l \lbar{J}I=a^l q^m R \in \cA(\cH_{R^\bullet})$ and because of \autoref{prop:fact-bij}, therefore $a^l q^m \in \cA(R^\bullet)$.

  \item By primitivity,
    \[
      \rf{\lbar{N_n}, \ldots, \lbar{N_1}, M_1, \ldots, M_m} \in \sZ^*_{\cI_v(\alpha)}(\lbar{J}I)
    \]
    is the unique rigid factorization of $\lbar{J}I$, and since as before no non-empty proper subproduct from the left (or the right) is principal, it is an atom in $\cH_{R^\bullet}$.
    For the ``in particular'' statement, note that if $m=n=1$ (i.e., $I$ and $J$ are both maximal left $S$-ideals), then $\lbar{J} I = qR$ if and only if $I = J$, and otherwise $\lbar{J} I$ is necessarily primitive.
  \end{enumerate}
\end{proof}

\begin{lemma} \label{lemma:a-is-nice}
  Let $I$ be a left $R$-ideal, $S = \cO_r(I)$, and $a \in R \cap S$. Then
  \[
    \prod_{i=1}^l (a^{l-i+1} R a^{-(l-i+1)}) a \cdot I = a^l I = (a^{l} I a^{-l}) a^l = (a^{l} I a^{-l}) \cdot \prod_{i=1}^l (a^{l-i+1} S a^{-(l-i+1)}) a
  \]
  with the left-most and the right-most expressions being proper products of
  \[
    (a^{l-i+1} R a^{-(l-i+1)}) a,\; I,\; (a^{l-i+1} S a^{-(l-i+1)}) a,\; a^l I a^{-l} \,\in\, \cI_v(\alpha).
  \]
  (The products have to be read in ascending order with ``$i=1$'' to the very left.)
\end{lemma}

\begin{proof}
  The formulas are clear, and so is that the products are proper ones. The key point is that these one-sided ideals are indeed integral. But this is so because $a \in S$, hence $a \in a^k S a^{-k}$ for all $k \in \bN_0$, implying that $(a^kSa^{-k})a \in \cI_v(\alpha)$, and similarly $a \in R$, thus $(a^k R a^{-k})a \in \cI_v(\alpha)$.
\end{proof}

\begin{lemma} \label{lemma:transpose-conj}
  Let $M$ be a maximal left $R$-ideal, and $N$ a maximal left $\cO_r(M)$-ideal.
  If $M \cdot N = N' \cdot M'$, then $\lbar{M} \cdot N' = N \cdot \lbar{M'}$.
\end{lemma}

\begin{proof}
  Since $\cO_r(\lbar{M}) = \cO_l(M) = \cO_l(N')$ and $\cO_r(N) = \cO_r(M') = \cO_l(\lbar{M'})$ the product is proper.
  We have
  \[
    \lbar{M} \cdot N' \cdot M' = \lbar{M} \cdot M \cdot N = \nr(M)\cO_l(N) \cdot N = N \cdot \cO_r(N)\nr(M) = N \cdot \lbar{M'} \cdot M',
  \]
  and thus $\lbar{M} \cdot N' = N \cdot \lbar{M'}$.
\end{proof}

\begin{proof}[Proof of \autoref{prop:delta-tot-def}]
  Let $p \in \cO_K$ be a totally positive prime element with $p\cO_K \nmid d_K \fD \fN$ and with $\nr(p)$ satisfying the bound of \autoref{prop:exist-max} for the classical maximal order $R$ (with $C=1$). Then there exists a maximal right $R$-ideal $U$ with $\nr(U) = p\cO_K$ that is non-free (i.e., non-principal) but is stably free (i.e., $[\nr(U)] = \vec 0$ in $\Cl^+(\cO_K)$). Let $U=U_0,\ldots,U_{r}$ be the maximal left $\cO_l(U)$-ideals of reduced norm $p\cO_K$ (of which there are $r + 1 = \Norm_{K/\bQ}(p) + 1$). By \autoref{lemma:exist-a}, there exists a totally positive prime element $q \in \cO_K$, $q \cO_K \nmid p d_K \fD \fN$, and an element
  \[
    a \,\in\, \cO_l(U) \,\cap\, \bigcap_{j=0}^r \cO_r(U_j) \quad\text{with}\quad \nr(a) = q \text{ and } a^2 \ne -q,
  \]
  and in fact, by \autoref{lemma:eps-in-center}, we may make this choice such that $a^2 \ne \varepsilon q$ for any $\varepsilon \in R^\times$.
  In addition, we may take $\Norm_{K/\bQ}(q)$ to be sufficiently large to satisfy the bound of \autoref{prop:exist-max} for $\cO_l(U)$ (with $C=2$). Then there exist distinct left $\cO_l(U)$-ideals $I$ and $J$ such that $I \cong J \cong U$ and $\nr(I) = \nr(J) = q \cO_K$.

  Set $S = \cO_r(I)$, and observe that $S \cong R$, because $U \cong I$.
  By \autoref{lemma:an-atom}, $(a^l \lbar{J} a^{-l}) a^l I \in \cA(\cH_{S^\bullet})$ for all $l \in \bN_0$, say $(a^l \lbar{J} a^{-l}) a^l I =y_lS$ with $y_l \in \cA(S^\bullet)$.
  We consider the principal right $S$-ideal $X_l = (a^l \lbar{J} a^{-l}) a^l I p \subset S$, say $X_l = x_l S$ with $x_l \in S^\bullet$. We will first determine all possible rigid factorizations of $X_l$ in $\cI_v(\alpha)$.
  As in \autoref{lemma:an-atom}, the right $S$-ideal $(a^l \lbar{J} a^{-l}) a^l I$ has reduced norm $q^{l+2} \cO_K$, is primitive, and thus possesses a unique rigid factorization,
  \[
    \rf{(a^{l} \lbar {J} a^{-l}), (a^{l} \cO_l(I) a^{-l}) a, (a^{l-1} \cO_l(I) a^{-(l-1)}) a, \ldots, (a\cO_l(I)a^{-1}) a, I} \;\in\; \sZ^*_{\cI_v(\alpha)}((a^l \lbar{J} a^{-l}) a^l I),
  \]
  with the $l+2$ factors $a^l \lbar{J} a^{-l},\, (a^{l-k} \cO_l(I) a^{-(l-k)}) a$ for $k \in [0,l-1]$ and $I$ all in $\cM_v(\alpha)$.

  For an element with a unique rigid factorization we make the convention of identifying the element and its factorization when this is notationally convenient. For principal ideals we omit the order and only write the generator if it is clear from the neighboring elements in the factorization what the order must be. For example, we can write the previous rigid factorization as $\rf{a^{l} \lbar{J} a^{-l}, a^l, I}$.

  $X_l$ has $(r+1)\binom{l+4}{2}$ rigid factorizations: They arise from the different rigid factorizations $\rf{U_i, \lbar{U_i}} \in \sZ^*_{\cI_v(\alpha)}(\cO_l(U)p)$ for $i \in [0,r]$ and the possible transpositions of $U_i$ and $\lbar{U_i}$.
  We denote the rigid factorization of $X_l$ that arises from $\rf{a^l J a^{-l}, a^l, U_i, \lbar{U_i}, I}$ by transposing the one-sided ideals of norm $p$ to the positions $m \in [-1,l+1]$ and $n \in [m,l+1]$ in the factorization by $F_{i,m,n}$: Here, the left-most position in the rigid factorization is denoted by $-1$, the right-most by $l+1$. So, by ``the rigid factorization obtained by transposing $U_i$ to the position $-1$ and $\lbar{U_i}$ to $l+1$'' we mean the unique rigid factorization of $X_l$ that has a factor of norm $p\cO_K$ as the first factor and as the last factor, and that can be transformed into $\rf{a^l\lbar{J}a^{-l},a^l, U_i, \lbar{U_i}, I}$ by transposition of maximal integral elements with coprime norm.

  For $i \in [0,r]$ let $V_i \in \cM_v(\alpha)$ and $M_i \in \cM_v(\alpha)$ be defined by $\lbar{U_i} I = M_i \lbar{V_i}$ under transposition, and let $W_i \in \cM_v(\alpha)$ and $N_i \in \cM_v(\alpha)$ be defined by $W_i \lbar{N_i} = \lbar{J} U_i$ under transposition.
 ($\{ V_i \mid i \in [0,r]\}$ is then the set of all $r+1$ left $S=\cO_r(I)$-ideals of reduced norm $p$. Similarly $\{ W_i \mid i \in [0,r] \}$ is then the set of all $r+1$ left $\cO_r(J)$-ideals of reduced norm $p$, and since $\cO_r(I) \cong \cO_r(J)$ the sets are actually conjugate under conjugation by an element of $A^\times$.)
  By \autoref{lemma:transpose-conj} we then also have $U_i M_i = I V_i$ and $ \lbar{W_i}\, \lbar{J} = \lbar{N_i}\, \lbar{U_i}$ under transposition.
  Using \autoref{lemma:a-is-nice} to see that $a$ transposes ``nicely'' with $U_i$ and $\lbar{U_i}$, we can explicitly describe all $F_{i,m,n}$ as follows:

  \begin{enumerate}
    \renewcommand{\theenumi}{{\textbf{Case \arabic{enumi}}}}
    \renewcommand{\labelenumi}{(\theenumi)}

    \item \label{C1}
      If $m=n=-1$:
      \[
        F_{i,m,n} = \rf{a^l W_i a^{-l}, a^l \lbar{W_i} a^{-l}, a^l \lbar{J}  a^{-l}, a^{l}, I}.
      \]

    \item \label{C2}
      If $m=-1$ and $0 \le n \le l$:
      \[
        F_{i,m,n} = \rf{a^l W_i a^{-l}, a^l \lbar N_i a^{-l}, a^{n}, a^{l-n} \lbar{U_i} a^{-(l-n)}, a^{l-n}, I}.
      \]

    \item \label{C3}
      If $0 \le m \le n \le l$:
      \[
        F_{i,m,n} = \rf{a^l \lbar J a^{-l}, a^{m}, a^{l-m} U_i a^{-(l-m)}, a^{n-m}, a^{l-n} \lbar{U_i} a^{-(l-n)}, a^{l-n}, I}.
      \]

    \item \label{C4}
      If $m=-1$ and $n=l+1$:
      \[
        F_{i,m,n} = \rf{a^l W_i a^{-l},  a^l \lbar{N_i} a^{-l}, a^{l}, M_i, \lbar{V_i}}.
      \]

    \item \label{C2s}
      If $0 \le m \le l$ and $n = l+1$:
      \[
        F_{i,m,n} = \rf{a^l \lbar J a^{-l}, a^{m}, a^{l-m} U_i a^{-(l-m)}, a^{l-m}, M_i, \lbar{V_i}}.
      \]

    \item \label{C1s}
      If $m=n=l+1$:
      \[
        F_{i,m,n} = \rf{a^l \lbar{J} a^{-l},  a^{l}, I, V_i, \lbar{V_i}}.
      \]
  \end{enumerate}
  For each of these rigid factorizations of the ideal $X_l$ in $\cI_v(\alpha)$ we can form minimal subproducts of principal one-sided ideals (starting from the left or the right) to obtain a representation of $X_l$ as a product in $\cH_{S^\bullet}$ (and hence a representation of $x_l$ as a product of elements of $S^\bullet$). But only when each of these minimal principal subproducts is an atom of $\cH_{S^\bullet}$ this gives rise to an actual rigid factorization of $x_l$ into atoms.
  We discuss the individual cases one-by-one:

  \begin{enumerate}
    \renewcommand{\theenumi}{{\textbf{Case \arabic{enumi}}}}
      \item 
        If $m=n=-1$:
        If $W_i$ is non-principal, then this does not give rise to a rigid factorization into atoms, as the first principal factor is $a^l(W_i \lbar{W_i}) a^{-l} = a^{l} (p\cO_r(J)) a^{-l}$, and this is not an atom (since there is at least one element in $\{ W_i \mid i \in [0,r] \}$ that is principal by \autoref{prop:exist-max}).
        If on the other hand $W_i$ is principal, then this gives rise to a rigid factorization of $X_l$ in $\cH_{S^\bullet}$ of length $3$, with atomic factors $a^{l} W_i a^{-l}$, $a^{l} \lbar{W_i} a^{-l}$ and $(a^{l} \lbar{J} a^{-l}) a^l I$, which in turn gives rise to a rigid factorization of length $3$ of $x_l \in S$.

      \item 
        If $m=-1$ and $0 \le n \le l$:
        \begin{enumerate}
          \item[\textbf{Case 2a}] If $U_i \cong I$:
            Then the last principal factor is necessarily $(a^{l-n} \lbar{U_i} a^{-(l-n)}) a^{l-n} I$. If $n < l$, then transposing $\lbar{U_i}$ to the right shows that this is not an atom in $\cH_{S^\bullet}$. If $n = l$ then $\lbar{U_i} I = M_i \lbar{V_i}$ is an atom if and only if $V_i$ is non-principal. Since also $U_i \cong J$, the factor $W_i \lbar{N_i} = \lbar{J}U_i$ is principal, and, because $\lbar{J}$ and $\lbar{U_i}$ are non-principal, this is either an atom (if $W_i$ is non-principal), or a product of two atoms (if $W_i$ is principal). So if $V_i$ is non-principal we get a rigid factorization of length either $l+2$ or $l+3$, and if $V_i$ is principal we get no rigid factorization into atoms.

          \item[\textbf{Case 2b}] If $U_i \not \cong I$:
            Then either there are no non-trivial principal factors (if $W_i$ is non-principal), or the first factor is $W_i$ and the remaining product does not factor into non-trivial principal factors.
        But then this second factor is not an atom, because after transposition of $\lbar{U_i}$ to the very left of the second factor (i.e, position $0$), we have a principal factor $\lbar{W_i}$. So in any case, this does not give rise to a rigid factorization into atoms.
        \end{enumerate}

      \item 
        If $0 \le m \le n \le l$:
        If $I \not \cong U_i$, then there are no non-trivial principal factors, and hence no rigid factorization into atoms is obtained.

        If $I \cong U_i$, then the first principal factor is $(a^l \lbar J a^{-l}) a^{m} (a^{l-m} U_i a^{-(l-m)})$, and the last one is $(a^{l-n}\lbar{U_i} a^{-(l-n)}) a^{l-n} I$.
        If $m > 0$ (or $n < l$), then by transposing $U_i$ to the left in the first factor (or $\lbar{U_i}$ to the right in the second factor) once, we see that this does not give rise to a rigid factorization into atoms.
        Consider now $m=0$ and $n=l$. If $V_i$ is principal, then $\lbar{J} U_i = V_i \lbar{K_i}$ implies that the first factor $a^l \lbar{J} U_i a^{-l}$ is no atom, and hence again we get no rigid factorization into atoms. Analogously we get no rigid factorization into atoms if $W_i$ is principal. If on the other hand $V_i$ and $W_i$ are both non-principal then $\lbar{J} U_i$ is an atom, and so is $\lbar{U_i} I$.
        Thus we obtain a rigid factorization of $X_l$ (and hence of $x_l$) of length $l+2$.
        (It is then in fact the same one as the one obtained from \ref{C2} in the same situation.)

      \item 
        If $m=-1$ and $n=l+1$:
        \begin{enumerate}
          \item[\textbf{Case 4a}] If $U_i \cong I$: Then $W_i \lbar{N_i}$ and $M_i \lbar{V_i}$ are both principal.
            If $W_i$ is non-principal, then $W_i \lbar{N_i} = \lbar{J} U_i$ is an atom since $J$ is non-principal. If $W_i$ is principal, then $W_i \lbar{N_i}$ is a product of two atoms. Similarly, $M_i \lbar{V_i}$ is either an atom or a product of two atoms.
            So in this case we get a rigid factorization of length $l+2$, $l+3$ or $l+4$.
            (In the case that $V_i$ is non-principal, it is the same one as in \ref{C2}. In the case that $V_i$ and $W_i$ are both non-principal it is the same as in \ref{C3} in the same situation.)

          \item[\textbf{Case 4b}] If $U_i \not \cong I$: Then $W_i \lbar{N_i}$ and $M_i \lbar{V_i}$ are both non-principal.
            If $W_i$ is principal, but $\lbar{V_i}$ is not, then the second principal factor is necessarily $(a^l \lbar{N_i} a^{-l}) a^{l} M_i \lbar{V_i}$, and this cannot be split as a non-trivial product of principal factors. But transposing $\lbar{V_i}$ to the very left in this factor gives a principal factor $\lbar{W_i}$, hence $(a^l \lbar{N_i} a^{-l}) a^{l} M_i \lbar{V_i}$ is not an atom. Arguing analogously, if $\lbar{V_i}$ is principal but $W_i$ is not, no rigid factorization into atoms is obtained.

            Finally, if $W_i$ and $V_i$ are both principal, we get a rigid factorization into $3$ atoms.
        \end{enumerate}

      \item 
        If $0 \le m \le l$ and $n = l+1$: This is analogous to \ref{C2}.

      \item 
        If $m=n=l+1$: This is analogous to \ref{C1}.
  \end{enumerate}

  Since there is at least one $i \in [0,r]$ for which $W_i$ is principal, we get at least one rigid factorization of $x_l$ of length $3$ from \ref{C1}.
  For $i=0$, $U_i \cong I$, so \ref{C4} gives at least one factorization with length in $[l+2,l+4]$.
  Note that which of the lengths in $[l+2,l+4]$ occur in \ref{C2}, \ref{C3}, and \ref{C4} depends only on the principality of certain one-sided ideals, and not on $l$. Thus we have shown that there exists a set $\emptyset \ne E \subset \{\, 2,3,4 \,\}$ such that, for any choice of $l \in \bN_0$,
  \[
    \sL_{S^\bullet}(x_l) = \{\, 3 \,\} \;\cup\; (l + E),
  \]
  and $x_l$ has the claimed form $x_l = y_l p$ with $y_l \in \cA(S^\bullet)$ and $p$ a totally positive prime element of $\cO_K$.
  Since $S \cong R$, the same is true for $R$.
\end{proof}

\begin{remark} \label{remark:alternative}
  \mbox{}
  \begin{enumerate}

  \item
    In the proof, the classical $\cO_K$-order
    \[
      T = \cO_l(U) \,\cap\, \bigcap_{j=0}^r \cO_r(U_j)
    \]
    is maximal at every prime $\fr \in \max(\cO_K)$ with $\fr \ne p \cO_K$ (thus $T_\fr \cong M_2(\cO_{K,\fr})$ if $\fr \nmid p\fD$ and $T_\fr$ is isomorphic to the unique classical maximal $\cO_{K,\fr}$-order in the unique quaternion division algebra over $K_\fr$ if $\fr \mid \fD$). At $\fp=p\cO_K$, it is not hard too see by local calculations that
  \[
    T_{\fp} \cong \left\{ \begin{pmatrix} a & b \\ p^2 c & a + p d \end{pmatrix} \;\bigg\lvert\; a,b,c,d \in \cO_{K,\fp} \right\}.
  \]
  But $T_{\fp}$ is not a classical Eichler order, and so neither is $T$.

\item If $\fr \in \max(\cO_K)$ we can find infinitely many pairwise non-associated totally positive prime elements $q \in \cO_K$ such that $\fr$ splits in $K(\sqrt{-q})$, and infinitely many pairwise non-associated totally positive prime elements $q \in \cO_K$ such that $\fr$ is inert in $K(\sqrt{-q})$: We may restrict ourselves to $q$ with $\Norm_{K/\bQ}(q)$ odd, and $q \cO_K \ne \fr$. Let $\fr' = \fr^{1 + \val_{\fr}(4)}$.
    If $-q \equiv 1 \mod \fr'$, then $-q$ is a square in $\cO_K/\fr'$ and hence $\fr$ splits in $K(\sqrt{-q})$. If $-q \equiv a \mod \fr'$, with $a$ a non-square in $\cO_K/\fr'$, then $\fr$ is inert in $K(\sqrt{-q})$. It therefore suffices to show that in every class of $(\cO_K/\fr')^\times$ there are infinitely many pairwise non-associated totally positive prime elements of $\cO_K$.
  Since we have the exact sequence
  \[
    \vec 1 \to \cO_K^{\times,+}/ \{ x \in \cO_{K}^{\times,+} \mid x \equiv_{\fr'} 1 \} \to (\cO_K/\fr')^\times \to \Cl_{\fr'}^+(\cO_K) \to \Cl^+(\cO_K) \to \vec 1,
  \]
  (cf. \cite[Lemma 3.2]{narkiewicz04}, \cite[Exercises VI.1.12, VI.1.13]{neukirch92}), it suffices that every class in the kernel of $\Cl_{\fr'}^+(\cO_K) \to \Cl^+(\cO_K)$ contains infinitely many pairwise non-associated prime elements. But this is so, because in fact every class of $\Cl_{\fr'}^+(\cO_K)$ contains infinitely many pairwise distinct maximal ideals (cf. \cite[Corollary 7 to Proposition 7.9]{narkiewicz04}).

\item Using the previous observation to find a suitable element $a$ in the proof, we can replace \autoref{lemma:exist-a} by a simpler one if $\fD \ne \cO_K$: Choosing the totally positive prime element $q \in \cO_K$ such that a prime divisor $\fr \mid \fD$ splits in $K(\sqrt{-q})$, the field $K(\sqrt{-q})$ does not embed into $A$ at all (see e.g. \cite[Theoreme III.3.8]{vigneras80} or \cite[Theorem 7.3.3]{maclachlan-reid03}).

  If $\fD = \cO_K$, we may make use of the fact that the particular classical order $T$ in the proof is contained in a classical Eichler order of squarefree level $\fp$. Taking $q$ such that $\fp$ is inert in $K(\sqrt{-q})$, the formulas for counting optimal embeddings (\cite[Corollaire III.5.12]{vigneras80}) show that no order of $K(\sqrt{-q})$ embeds into $T$.

    For this approach we only need the qualitative statement of \autoref{prop:pos-def-quad}, but not the order of magnitude.
  \end{enumerate}
\end{remark}

\begin{proof}[Proof of \autoref{prop:shift}]
  It suffices to prove the claim for $n=1$.
  Let $a$ in $R^\bullet$, and let
  \[
    \{ \rf{I_1^{(1)},\ldots, I_k^{(1)}}, \;\ldots,\; \rf{I_1^{(l)},\ldots, I_k^{(l)}} \} = \sZ^*_{\cI_v(\alpha)}( Ra) \subset \cF(\cM_v(\alpha))
  \]
  be the set of all rigid factorizations of $Ra$ in $\cI_v(\alpha)$.
  Using \autoref{prop:pos-def-quad}, we can choose a totally positive prime element $q \in \cO_K$ with $q \nmid \nr(a)$ and such that there exists an $x \in A^\times$ with $\nr(x) = q$ and
  \[
    x \;\in\; T = \bigcap_{i=1}^l \bigcap_{j=1}^k \cO_l(I_j^{(i)}) \cap \cO_r(I_j^{(i)}).
  \]
  We claim $\sL(xa) = 1 + \sL(a)$.
  The rigid factorizations of $Rxa$ in $\cI_v(\alpha)$ are given by all possible transpositions of $x$ to any position in
  \[
    \rf{I_1^{(i)},\ldots,I_m^{(i)},x},
  \]
  for all $i \in [1,l]$. But, since $x \in T$, it follows from \autoref{lemma:a-is-nice} that any such rigid factorization is of the form
  \[
    \rf{I_1^{(i)},\ldots, I_m^{(i)}, x, x^{-1} I_{m+1}^{(i)} x,\ldots, x^{-1} I_{k}^{(i)} x}.
  \]
  for $m \in [0,k]$. We see that for the principal subproducts this does not change anything except insert one additional factor (corresponding to $x$) at some position.
  Thus, for each $i \in [1,l]$, $\rf{I_1^{(i)}, \ldots, I_m^{(i)}, x, x^{-1}I_{m+1}^{(i)}x, \ldots, x^{-1} I_{k}^{(i)} x}$ gives rise to a rigid factorization of $ax$ in $R^\bullet$ of length $l+1$ if and only if $\rf{I_1^{(i)}, \ldots, I_k^{(i)}}$ gives rise to a rigid factorization of $a$ of length $l$.
\end{proof}

\section{Acknowledgments}

I thank Alfred Geroldinger for suggesting the topic, and him and Franz Halter-Koch for reading preliminary versions of this manuscript and providing many valuable comments.

\bibliographystyle{hyperabbrv}
\bibliography{maxord}

\begin{thebibliography}{10}

\bibitem{andersondd97}
D.~D. Anderson, editor.
\newblock {\em Factorization in integral domains}, volume 189 of {\em Lecture
  Notes in Pure and Applied Mathematics}, New York, 1997. Marcel Dekker Inc.

\bibitem{asano39}
K.~Asano.
\newblock Arithmetische {I}dealtheorie in nichtkommutativen {R}ingen.
\newblock {\em Jpn. J. Math.}, 16:1--36, 1939.

\bibitem{asano49}
K.~Asano.
\newblock Zur {A}rithmetik in {S}chiefringen. {I}.
\newblock {\em Osaka Math. J.}, 1:98--134, 1949.

\bibitem{asano50}
K.~Asano.
\newblock Zur {A}rithmetik in {S}chiefringen. {II}.
\newblock {\em J. Inst. Polytech. Osaka City Univ. Ser. A. Math.}, 1:1--27,
  1950.

\bibitem{asano-murata53}
K.~Asano and K.~Murata.
\newblock Arithmetical ideal theory in semigroups.
\newblock {\em J. Inst. Polytech. Osaka City Univ. Ser. A. Math.}, 4:9--33,
  1953.

\bibitem{asano-ukegawa52}
K.~Asano and T.~Ukegawa.
\newblock Erg\"anzende {B}emerkungen \"uber die {A}rithmetik in {S}chiefringen.
\newblock {\em J. Inst. Polytech. Osaka City Univ. Ser. A. Math.}, 3:1--7,
  1952.

\bibitem{baeth-wiegand13}
N.~R. Baeth and R.~Wiegand.
\newblock Factorization theory and decompositions of modules.
\newblock {\em
  \href{http://dx.doi.org/10.4169/amer.math.monthly.120.01.003}{Amer. Math.
  Monthly}}, 120(1):3--34, 2013.

\bibitem{berrick-keating00}
A.~J. Berrick and M.~E. Keating.
\newblock {\em An introduction to rings and modules with {$K$}-theory in view},
  volume~65 of {\em Cambridge Studies in Advanced Mathematics}.
\newblock Cambridge University Press, Cambridge, 2000.

\bibitem{brandt27}
H.~Brandt.
\newblock \"{U}ber eine {V}erallgemeinerung des {G}ruppenbegriffes.
\newblock {\em \href{http://dx.doi.org/10.1007/BF01209171}{Math. Ann.}},
  96(1):360--366, 1927.

\bibitem{brandt28}
H.~Brandt.
\newblock Idealtheorie in {Q}uaternionenalgebren.
\newblock {\em \href{http://dx.doi.org/10.1007/BF01459083}{Math. Ann.}},
  99(1):1--29, 1928.

\bibitem{chamarie81}
M.~Chamarie.
\newblock Anneaux de {K}rull non commutatifs.
\newblock {\em \href{http://dx.doi.org/10.1016/0021-8693(81)90318-5}{J.
  Algebra}}, 72(1):210--222, 1981.

\bibitem{chapman05}
S.~T. Chapman, editor.
\newblock \href{http://dx.doi.org/10.1201/9781420028249}{{\em Arithmetical
  properties of commutative rings and monoids}}, volume 241 of {\em Lecture
  Notes in Pure and Applied Mathematics}. Chapman \& Hall/CRC, Boca Raton, FL,
  2005.

\bibitem{claborn65}
L.~Claborn.
\newblock Dedekind domains and rings of quotients.
\newblock {\em Pacific J. Math.}, 15:59--64, 1965.

\bibitem{cohn85}
P.~M. Cohn.
\newblock {\em Free rings and their relations}, volume~19 of {\em London
  Mathematical Society Monographs}.
\newblock Academic Press Inc. [Harcourt Brace Jovanovich Publishers], London,
  second edition, 1985.

\bibitem{cohn06}
P.~M. Cohn.
\newblock \href{http://dx.doi.org/10.1017/CBO9780511542794}{{\em Free ideal
  rings and localization in general rings}}, volume~3 of {\em New Mathematical
  Monographs}.
\newblock Cambridge University Press, Cambridge, 2006.

\bibitem{conway-smith03}
J.~H. Conway and D.~A. Smith.
\newblock {\em On quaternions and octonions: their geometry, arithmetic, and
  symmetry}.
\newblock A K Peters Ltd., Natick, MA, 2003.

\bibitem{deuring68}
M.~Deuring.
\newblock {\em Algebren}.
\newblock Zweite, korrigierte Auflage. Ergebnisse der Mathematik und ihrer
  Grenzgebiete, Band 41. Springer-Verlag, Berlin, 1968.

\bibitem{estes91}
D.~R. Estes.
\newblock Factorization in quaternion orders over number fields.
\newblock In {\em The mathematical heritage of {C}. {F}. {G}auss}, pages
  195--203. World Sci. Publ., River Edge, NJ, 1991.

\bibitem{estes-nipp89}
D.~R. Estes and G.~Nipp.
\newblock Factorization in quaternion orders.
\newblock {\em \href{http://dx.doi.org/10.1016/0022-314X(89)90009-7}{J. Number
  Theory}}, 33(2):224--236, 1989.

\bibitem{fontana-houston-lucas12}
M.~Fontana, E.~Houston, and T.~Lucas.
\newblock {\em Factoring {I}deals in {I}ntegral {D}omains}, volume~14 of {\em
  Lecture Notes of the Unione Matematica Italiana}.
\newblock Springer, 2012.

\bibitem{frisch12}
S.~Frisch.
\newblock A construction of integer-valued polynomials with prescribed sets of
  lengths of factorizations.
\newblock {\em \href{http://dx.doi.org/10.1007/s00605-013-0508-z}{Monatsh.
  Math.}}, 2013.
\newblock to appear.

\bibitem{froehlich75}
A.~Fr{\"o}hlich.
\newblock Locally free modules over arithmetic orders.
\newblock {\em \href{http://dx.doi.org/10.1515/crll.1975.274-275.112}{J. Reine
  Angew. Math.}}, 274/275:112--124, 1975.
\newblock Collection of articles dedicated to Helmut Hasse on his seventy-fifth
  birthday, III.

\bibitem{gabriel-zisman67}
P.~Gabriel and M.~Zisman.
\newblock {\em Calculus of fractions and homotopy theory}.
\newblock Ergebnisse der Mathematik und ihrer Grenzgebiete, Band 35.
  Springer-Verlag New York, Inc., New York, 1967.

\bibitem{geroldinger09}
A.~Geroldinger.
\newblock \href{http://dx.doi.org/10.1007/978-3-7643-8962-8}{Additive group
  theory and non-unique factorizations}.
\newblock In {\em Combinatorial number theory and additive group theory}, Adv.
  Courses Math. CRM Barcelona, pages 1--86. Birkh\"auser Verlag, Basel, 2009.

\bibitem{geroldinger13}
A.~Geroldinger.
\newblock Non-commutative {K}rull monoids: a divisor theoretic approach and
  their arithmetic.
\newblock {\em Osaka Math. J.}, 50:503 -- 539, 2013,
  \href{http://arxiv.org/abs/1208.4202}{{\tt arXiv:1208.4202}}.

\bibitem{geroldinger-grynkiewicz09}
A.~Geroldinger and D.~J. Grynkiewicz.
\newblock On the arithmetic of {K}rull monoids with finite {D}avenport
  constant.
\newblock {\em \href{http://dx.doi.org/10.1016/j.jalgebra.2008.11.020}{J.
  Algebra}}, 321(4):1256--1284, 2009.

\bibitem{ghk06}
A.~Geroldinger and F.~Halter-Koch.
\newblock \href{http://dx.doi.org/10.1201/9781420003208}{{\em Non-unique
  factorizations}}, volume 278 of {\em Pure and Applied Mathematics (Boca
  Raton)}.
\newblock Chapman \& Hall/CRC, Boca Raton, FL, 2006.
\newblock Algebraic, combinatorial and analytic theory.

\bibitem{geroldinger-hassler08}
A.~Geroldinger and W.~Hassler.
\newblock Arithmetic of {M}ori domains and monoids.
\newblock {\em \href{http://dx.doi.org/10.1016/j.jalgebra.2007.11.025}{J.
  Algebra}}, 319(8):3419--3463, 2008.

\bibitem{geroldinger-yuan12}
A.~Geroldinger and P.~Yuan.
\newblock The set of distances in {K}rull monoids.
\newblock {\em \href{http://dx.doi.org/10.1112/blms/bds046}{Bull. London Math.
  Soc.}}, 44:1203–1208, 2012.

\bibitem{graetzer11}
G.~Gr{\"a}tzer.
\newblock \href{http://dx.doi.org/10.1007/978-3-0348-0018-1}{{\em Lattice
  theory: foundation}}.
\newblock Birkh\"auser/Springer Basel AG, Basel, 2011.

\bibitem{halimi12}
N.~H. Halimi.
\newblock Right {M}ori orders.
\newblock \href{http://arxiv.org/abs/1203.2785}{{\tt arXiv:1203.2785}}.
\newblock preprint.

\bibitem{halterkoch98}
F.~Halter-Koch.
\newblock {\em Ideal systems}, volume 211 of {\em Monographs and Textbooks in
  Pure and Applied Mathematics}.
\newblock Marcel Dekker Inc., New York, 1998.
\newblock An introduction to multiplicative ideal theory.

\bibitem{halterkoch10}
F.~Halter-Koch.
\newblock \href{http://dx.doi.org/10.1007/978-1-4419-6990-3_8}{Multiplicative
  ideal theory in the context of commutative monoids}.
\newblock In {\em Commutative algebra---{N}oetherian and non-{N}oetherian
  perspectives}, pages 203--231. Springer, New York, 2011.

\bibitem{jacobson43}
N.~Jacobson.
\newblock {\em The {T}heory of {R}ings}.
\newblock American Mathematical Society Mathematical Surveys, vol. I. American
  Mathematical Society, New York, 1943.

\bibitem{jespers86}
E.~Jespers.
\newblock On {$\Omega$}-{K}rull rings.
\newblock {\em Quaestiones Math.}, 9(1-4):311--338, 1986.
\newblock Classical and categorical algebra (Durban, 1985).

\bibitem{jespers-okninski07}
E.~Jespers and J.~Okni{\'n}ski.
\newblock {\em Noetherian semigroup algebras}, volume~7 of {\em Algebras and
  Applications}.
\newblock Springer, Dordrecht, 2007.

\bibitem{jespers-wang01}
E.~Jespers and Q.~Wang.
\newblock Noetherian unique factorization semigroup algebras.
\newblock {\em \href{http://dx.doi.org/10.1081/AGB-100107954}{Comm. Algebra}},
  29(12):5701--5715, 2001.

\bibitem{kainrath99}
F.~Kainrath.
\newblock Factorization in {K}rull monoids with infinite class group.
\newblock {\em Colloq. Math.}, 80(1):23--30, 1999.

\bibitem{kirschmer-voight10}
M.~Kirschmer and J.~Voight.
\newblock Algorithmic enumeration of ideal classes for quaternion orders.
\newblock {\em \href{http://dx.doi.org/10.1137/080734467}{SIAM J. Comput.}},
  39(5):1714--1747, 2010.

\bibitem{kirschmer-voight10cor}
M.~Kirschmer and J.~Voight.
\newblock Corrigendum: Algorithmic enumeration of ideal classes for quaternion
  orders.
\newblock {\em \href{http://dx.doi.org/10.1137/120866063}{SIAM J. Comput.}},
  41(3):714--714, 2012.

\bibitem{maclachlan-reid03}
C.~Maclachlan and A.~W. Reid.
\newblock {\em The arithmetic of hyperbolic 3-manifolds}, volume 219 of {\em
  Graduate Texts in Mathematics}.
\newblock Springer-Verlag, New York, 2003.

\bibitem{mcconnell-robson01}
J.~C. McConnell and J.~C. Robson.
\newblock {\em Noncommutative {N}oetherian rings}, volume~30 of {\em Graduate
  Studies in Mathematics}.
\newblock American Mathematical Society, Providence, RI, revised edition, 2001.
\newblock With the cooperation of L. W. Small.

\bibitem{narkiewicz04}
W.~Narkiewicz.
\newblock {\em Elementary and analytic theory of algebraic numbers}.
\newblock Springer Monographs in Mathematics. Springer-Verlag, Berlin, third
  edition, 2004.

\bibitem{neukirch92}
J.~Neukirch.
\newblock {\em Algebraic number theory}, volume 322 of {\em Grundlehren der
  Mathematischen Wissenschaften [Fundamental Principles of Mathematical
  Sciences]}.
\newblock Springer-Verlag, Berlin, 1999.
\newblock Translated from the 1992 German original and with a note by Norbert
  Schappacher, With a foreword by G. Harder.

\bibitem{rehm77-1}
H.~P. Rehm.
\newblock Multiplicative ideal theory of noncommutative {K}rull pairs. {I}.
  {M}odule systems, {K}rull ring-type chain conditions, and application to
  two-sided ideals.
\newblock {\em J. Algebra}, 48(1):150--165, 1977.

\bibitem{rehm77-2}
H.~P. Rehm.
\newblock Multiplicative ideal theory of noncommutative {K}rull pairs. {II}.
  {F}actorization of one-sided ideals.
\newblock {\em J. Algebra}, 48(1):166--181, 1977.

\bibitem{reiner75}
I.~Reiner.
\newblock {\em Maximal orders}.
\newblock Academic Press [A subsidiary of Harcourt Brace Jovanovich,
  Publishers], London-New York, 1975.
\newblock London Mathematical Society Monographs, No. 5.

\bibitem{remak54}
R.~Remak.
\newblock \"{U}ber algebraische {Z}ahlk\"orper mit schwachem {E}inheitsdefekt.
\newblock {\em Compositio Math.}, 12:35--80, 1954.

\bibitem{schmid09}
W.~A. Schmid.
\newblock A realization theorem for sets of lengths.
\newblock {\em \href{http://dx.doi.org/10.1016/j.jnt.2008.10.019}{J. Number
  Theory}}, 129(5):990--999, 2009.

\bibitem{schulze-pillot04}
R.~Schulze-Pillot.
\newblock \href{http://dx.doi.org/10.1090/conm/344/06226}{Representation by
  integral quadratic forms---a survey}.
\newblock In {\em Algebraic and arithmetic theory of quadratic forms}, volume
  344 of {\em Contemp. Math.}, pages 303--321. Amer. Math. Soc., Providence,
  RI, 2004.

\bibitem{steinberg10}
S.~A. Steinberg.
\newblock \href{http://dx.doi.org/10.1007/978-1-4419-1721-8}{{\em
  Lattice-ordered rings and modules}}.
\newblock Springer, New York, 2010.

\bibitem{swan80}
R.~G. Swan.
\newblock Strong approximation and locally free modules.
\newblock In {\em Ring theory and algebra, {III} ({P}roc. {T}hird {C}onf.,
  {U}niv. {O}klahoma, {N}orman, {O}kla., 1979)}, volume~55 of {\em Lecture
  Notes in Pure and Appl. Math.}, pages 153--223. Dekker, New York, 1980.

\bibitem{vigneras80}
M.-F. Vign{\'e}ras.
\newblock \href{http://dx.doi.org/10.1007/BFb0091027}{{\em Arithm\'etique des
  alg\`ebres de quaternions}}, volume 800 of {\em Lecture Notes in
  Mathematics}.
\newblock Springer, Berlin, 1980.

\bibitem{washington97}
L.~C. Washington.
\newblock \href{http://dx.doi.org/10.1007/978-1-4612-1934-7}{{\em Introduction
  to cyclotomic fields}}, volume~83 of {\em Graduate Texts in Mathematics}.
\newblock Springer-Verlag, New York, second edition, 1997.

\end{thebibliography}

\end{document}